\documentclass[12pt,oneside]{amsart}
\usepackage{amsmath,amssymb,cite,mathrsfs,tikz-cd}
\usepackage[all]{xy}
\usepackage{typearea} 
\usepackage{hyperref} 
\usepackage{color} 

\newcounter{maincounter}
\numberwithin{maincounter}{section}
\numberwithin{equation}{section}

\newtheorem{lemma}[maincounter]{Lemma}
\newtheorem{proposition}[maincounter]{Proposition}

\newtheorem{remark}[maincounter]{Remark}
\newtheorem{theorem}[maincounter]{Theorem}

\newtheorem{definition}[maincounter]{Definition}
\newtheorem{thm}{Theorem}[section]

\makeatletter
\newcommand{\neutralize}[1]{\expandafter\let\csname c@#1\endcsname\count@}
\makeatother
  
\newenvironment{thmbis}[1]
  {\renewcommand{\thethm}{\ref*{#1}$'$}%
   \neutralize{thm}\phantomsection
   \begin{thm}}
  {\end{thm}}

\def\AA{\mathbb{A}}

\def\NN{\mathbb{N}}
\def\RR{\mathbb{R}}
\def\TT{\mathbb{T}}
\def\CC{\mathbb{C}}
\def\ZZ{\mathbb{Z}}
\def\PP{\mathbb{P}}
\def\QQ{\mathbb{Q}}

\newcommand{\cal}{\mathcal}

\newcommand{\scrZ}{\mathscr{Z}}
\newcommand{\cA}{\cal{A}}
\newcommand{\cU}{\cal{U}}
\newcommand{\cB}{\cal{B}}

\newcommand{\sman}[1]{{#1}^{\mathrm{sm,an}}}
\newcommand{\sm}[1]{{#1}^{\mathrm{sm}}}
\newcommand{\anE}{\mathrm{an}}
\newcommand{\an}[1]{{#1}^{\anE}}
\newcommand{\stab}[1]{{\mathrm{Stab}(#1)}}

\newcommand{\IP}{{\PP}}

\newcommand{\IC}{{\CC}}
\newcommand{\IR}{{\RR}}
\newcommand{\IT}{{\TT}}
\newcommand{\IRan}{{{\RR}_{\rm an}}}

\newcommand{\IQbar}{{\overline{\QQ}}}
\newcommand{\Kbar}{{\overline{K}}}
\newcommand{\IZ}{{\ZZ}}
\newcommand{\IN}{{\NN}}
\newcommand{\IA}{{\AA}}
\newcommand{\IQ}{{\QQ}}

\newcommand{\cC}{{\mathcal C}}

\newcommand{\cL}{{\mathcal L}}
\newcommand{\cM}{{\mathcal M}}
\newcommand{\cO}{{\mathcal O}}

\newcommand{\cX}{{\mathcal{X}}}
\newcommand{\cY}{{\mathcal Y}}
\newcommand{\cZ}{{\mathcal Z}}

\newcommand{\imS}{{\rm Im}}
\newcommand{\im}[1]{\imS({#1})}
\newcommand{\imageS}{{\rm im}}
\newcommand{\image}[1]{\imageS({#1})}
\newcommand{\volS}{{\rm vol}}
\newcommand{\vol}[1]{\volS({#1})}

\newcommand{\mat}[2]{{\rm Mat}_{#1}({#2})}
\newcommand{\ssm}{\setminus}

\newcommand{\trdeg}[1]{{\mathrm{trdeg}}({#1})}

\newcommand{\en}[1]{{\rm End}({#1})}
\newcommand{\Hom}[1]{{\rm Hom}({#1})}

\newcommand{\codim}{{\rm codim}}

\renewcommand{\subset}{\subseteq} 
\renewcommand{\supset}{\supseteq}

\newcommand{\gra}[1]{\mathrm{Gr}({#1})}

\newcommand{\gl}[2]{{\mathrm {GL}}_{#1}({#2})}

\newcommand{\autS}{{\mathrm {Aut}}}
\newcommand{\aut}[1]{\autS({#1})}

\newcommand{\spec}[1]{\mathrm{Spec}\,{#1}}

\newcommand{\tor}[1]{{#1}_{\mathrm{tor}}}
\newcommand{\gal}[1]{{\mathrm{Gal}}({#1})}

\newcommand{\zeroset}[1]{\mathscr{Z}({#1})}

\newcommand{\unitdisc}{\Delta}

\begin{document}
\title[]{Heights in Families of Abelian Varieties and the Geometric Bogomolov Conjecture}
\author{Ziyang Gao and Philipp Habegger}

\address{CNRS, IMJ-PRG, 4 place de Jussieu, 75005 Paris, France;  Department of Mathematics, Princeton University, Princeton, NJ 08544, USA}
\email{ziyang.gao@imj-prg.fr}
\address{Department of Mathematics and Computer Science, University of Basel, Spiegelgasse 1, 4051 Basel, Switzerland}
\email{philipp.habegger@unibas.ch}

\subjclass[2000]{11G10, 11G50, 14G25, 14K15}

\maketitle
\begin{abstract}
On an abelian scheme over a smooth curve over $\IQbar$ 
a symmetric relatively ample line bundle defines a fiberwise
 N\'{e}ron--Tate height.
If the base curve is inside a projective space, we also have a height on
its $\IQbar$-points that serves as a measure of each fiber, 
an abelian variety.  
Silverman proved an asymptotic equality between these two heights 
on a curve in the abelian scheme. 
In this paper we prove an inequality between these heights on a
 subvariety of any dimension of the abelian scheme.
As an application we prove the
 Geometric Bogomolov Conjecture for
the function field of a curve
defined over $\IQbar$. 
Using Moriwaki's height we sketch how to extend our result when the
base field of the curve 
has characteristic $0$.
\end{abstract}
\tableofcontents

\section{Introduction}

In 1998, Ullmo \cite{Ullmo} and S.~Zhang \cite{ZhangEquidist} proved
the Bogomolov Conjecture over number fields. However its analog over
function fields, which came  to be  known as the Geometric Bogomolov
Conjecture, remains  open in full generality. 

The main goal of this paper is to prove a height inequality on a
subvariety of an abelian scheme over a smooth curve over $\IQbar$,
Theorem~\ref{TheoremMainResultHeightBound}. It is then not hard to
deduce the Geometric Bogomolov Conjecture over the function field of a
curve in the characteristic $0$ case. See
$\mathsection$\ref{SectionGBC} and Appendix~\ref{sec:appendixA}. Our
height inequality 
may be of independent interest and does not seem to
follow from the Geometric Bogomolov Conjecture. It can
 serve as a substitute in higher dimension of Silverman's Height Limit
Theorem \cite{Silverman}, used by Masser and Zannier \cite{MZ:AJM10}
to prove a first case of the relative Manin--Mumford Conjecture for
sections of the base curve; we refer to Pink's work \cite{Pink} and
Zannier's book \cite{ZannierBook} on such problems.

Let $k$ be an algebraically closed field of characteristic $0$, $K$ a
field extension of $k$, and $\Kbar$ a fixed algebraic closure of $K$.
Let $A$ be an abelian variety over $K$. We let $A^{\Kbar/k}$ denote
the $\Kbar/k$-trace of $A\otimes_{K} \Kbar$; it is an abelian variety
over $k$ and we have a trace map $\tau_{A,\Kbar/k}\colon
A^{\Kbar/k}\otimes_k \Kbar\rightarrow A\otimes_K \Kbar$, which is a
closed immersion since $\mathrm{char} (k) =0$. By abuse of notation we
consider $A^{\overline{K}/k} \otimes_k \overline{K}$ as an abelian
subvariety of $A\otimes_K {\overline{K}}$. We refer
to \S \ref{sec:basicnotation} for references and more information on
the trace.

Suppose now that $K$ is the function field of a
smooth projective curve over $k$. In particular, we have
$\trdeg{K/k}=1$. 

Let $L$ be a symmetric ample line bundle on $A$. We can attach to
$A,L,$ and $K$ the N\'{e}ron--Tate height $\hat{h}_{K,A,L}\colon
A(\overline{K})\rightarrow [0,\infty)$, see \S\ref{sec:heights} for
additional background on the N\'eron--Tate height.
This height satisfies: For any $P\in A(\overline{K})$ we have
\begin{equation*}
\hat{h}_{K,A,L}(P)=0\text{ if and only if }P \in 
(A^{\overline{K}/k} \otimes_k \overline{K})(\Kbar) + \tor{A},
\end{equation*}
here $\tor{A}$ denotes the subgroup of points of finite order of
$A(\Kbar)$.

A \textit{coset} in an abelian variety is the translate of an abelian
subvariety, we call it a \textit{torsion coset} if it contains a point
of finite order.

Our main result towards the Geometric Bogomolov Conjecture is the
following theorem. We first concentrate on the important case
$k=\IQbar$.

\begin{theorem}
\label{TheoremGBC}
We keep the notation from above and assume $k=\IQbar$. Let $X$ be an
irreducible, closed subvariety of $A$ defined over $K$ such that
$X\otimes_K \overline{K}$ is irreducible and not of the form
$B+({Z} \otimes_k {\overline{K}})$ for some closed irreducible
subvariety ${Z}$ of $A^{\overline{K}/k}$ and some torsion coset $B$ in
$A\otimes_K \overline{K}$. Then there exists a constant $\epsilon>0$
such that
\[
\{x \in X(\overline{K}): \hat{h}_{K,A,L}(x) \le \epsilon \}
\]
 is not Zariski dense in $X$.
\end{theorem}

In Appendix~\ref{sec:appendixA} we sketch a proof for when $k$ is any
 algebraically closed field of characteristic $0$ using Moriwaki's
 height.

Yamaki~\cite[Conjecture 0.3]{Yamaki:GeoBogo13} proposed a general
conjecture over function fields which we will call the Geometric
Bogomolov Conjecture; it allows $\trdeg{K/k}$ to be greater than $1$
and $k$ algebraically closed of arbitrary characteristic. The
reference to geometry distinguishes Yamaki's Conjecture from the
arithmetic counterpart over a number field.

The Geometric Bogomolov Conjecture was proven by
Gubler~\cite{Gubler:Bogo} when $A$ is totally degenerate at some place
of $K$. He has no restriction on the characteristic of $k$ and does
not assume that $K/k$ has transcendence degree $1$. When $X$ is a
curve embedded in its Jacobian $A$ and when $\trdeg{K/k}=1$, Yamaki
dealt with nonhyperelliptic curves of genus $3$
in \cite{Yamaki:nonhyperellipticcurve} and with hyperelliptic curves
of any genus in \cite{Yamaki:hyperellipticcurve}. If moreover
$\mathrm{char}(k)=0$, Faber~\cite{Faber:GBC} proved the conjecture for
$X$ of small genus (up to $4$, effective) and
Cinkir~\cite{CinkirBogomolov} covered the case of arbitrary genus.
Prior to these work, Moriwaki also gave some partial results
in \cite{Moriwaki:BogomolovInequality}.
Yamaki~\cite{YamakiBogomolovTrace} reduced the Geometric Bogomolov
Conjecture to the case where $A$ has good reduction everywhere and has
trivial $\overline{K}/k$-trace. He also proved the cases
$\mathrm{(co)dim} X=1$ \cite{YamakiNonDensity} and $\dim (A\otimes_K
{\overline{K}}/ 
(A^{\overline{K}/k} \otimes_k \overline{K})) \le
5$ \cite{YamakiGeoBogoDim5}. As in Gubler's setup, Yamaki works in
arbitrary characteristic and has no restriction on $K/k$. These
results involve techniques ranging from analytic tropical
geometry \cite{Gubler:trop} to Arakelov theory; the latter method
overlaps with Ullmo and S.~Zhang's original approach for number
fields.

Our approach differs and is based on a height inequality on a model of
$A$ to be stated below in Theorem~\ref{TheoremMainResultHeightBound}
(see Appendix~A for a version involving the Moriwaki height). In a
recent collaboration with Cantat and Xie \cite{CGHX:18} we were able
to resolve the Geometric Bogomolov Conjecture completely in
characteristic $0$. While the methods in \cite{CGHX:18} there were motivated by those
presented here, they do not bypass
through or recover a height inequality such as Theorem~\ref{TheoremMainResultHeightBound}.

To prove Theorem \ref{TheoremGBC} we must work in the relative
setting. Let us setup the notation. Let $S$ be a smooth irreducible
curve over $k$, and let $\pi \colon \cA \rightarrow S$ be an abelian
scheme of relative dimension $g\ge 1$. Let $A$ be the generic fiber of
$\cA \rightarrow S$; it is an abelian variety over $k(S)$, the
function field of $S$. We will prove the Geometric Bogomolov
Conjecture for $A$ and $K=k(S)$. Let us also fix an algebraic closure
$\overline{k(S)}$ of $k(S)$.

\begin{definition}\label{DefinitionSpecialSubvariety}
An irreducible closed subvariety $Y$ of $\cA$ is 
called a \textbf{generically special} subvariety of $\cA$, or just
generically special, if
it dominates $S$ and if
its geometric generic fiber $Y\times_S \spec{\overline{k(S)}}$ 
is a finite
 union of $
 (Z\otimes_k \overline{k(S)})+B$, where $Z$ is a closed
irreducible subvariety of $A^{\overline{k(S)}/k}$ and $B$ is a torsion
coset in $A\otimes_{k(S)} \overline{k(S)}$.
\end{definition}

For any irreducible closed subvariety $X$ of $\cA$, we set
\[
X^*=X\setminus \bigcup_{\substack{Y \subset X \\ Y\text{ is a generically
special}\\ \text{subvariety of $\cA$}}} Y.
\]
We start with the following proposition which clarifies the structure
of $X^*$. 
Its proof  relies on a uniform version of 
 Raynaud's \cite{Raynaud:MM} resolution of  the Manin--Mumford
Conjecture in characteristic $0$
as well as the Lang--N\'eron Theorem, the generalization of the
Mordell--Weil Theorem to finitely generated fields. 

\begin{proposition}\label{TheoremStarSetOpen}
Let $X$ and $\cA$ be as above. 
There are at most finitely many generically special subvarieties of $\cA$ that are contained in 
$X$, maximal with respect to the inclusion for this property. In particular, 
 $X^*$ is Zariski open in $X$ and it is empty if and only if $X$ is generically special.
\end{proposition} 


Let us now assume $k=\IQbar$ and turn to height functions.
We write $h(\cdot)$ for the absolute logarithmic Weil height on
projective space. 

Let $\overline{S}$ be a smooth projective curve over $\IQbar$ containing $S$ as a Zariski open and dense subset. Let $\overline \cM$ be an ample line bundle on $\overline{S}$ and let $\cM = \overline{\cM}|_S$. The Height Machine \cite[Chapter~2.4]{BG} attaches to $(\overline S,\overline{\cM})$ a function $\overline{S}(\IQbar)\rightarrow \IR$ that is well defined up to addition of a bounded function. Let $h_{S,\mathcal{M}}$ be the restriction to $S(\IQbar)$ of a representative of this class of functions. As $\overline \cM$ is ample on $\overline{S}$, we may take such a representative that $h_{S,\mathcal{M}}(s) \ge 0$ for each $s \in S(\IQbar)$.

Let $\mathcal{L}$ be a relatively ample and symmetric line bundle on $\cA/S$ defined over $\IQbar$. Then for any $s \in S(\IQbar)$, the line bundle $\mathcal{L}_s$ on the abelian variety $\cA_s= \pi^{-1}(s)$ is symmetric; note that $\cA_s$ is defined over $\IQbar$. Tate's Limit Process provides a fiberwise N\'eron--Tate height $\hat{h}_{\cA_s,\mathcal{L}_s} \colon \cA_s(\IQbar) \rightarrow [0,\infty)$. It is determined uniquely by the restriction of $\cL$ to $\cA_s$, there is no need to fix a representative here. Finally define $\hat{h}_{\cA,\cL} \colon \cA(\IQbar) \rightarrow [0,\infty)$ to be the \textit{total N\'eron--Tate height} given by $P \mapsto \hat{h}_{\cA_{\pi(P)},\cL_{\pi(P)}}(P)$ for all $P\in \cA(\IQbar)$.

These two height functions are unrelated in the following sense. It is
not difficult to construct an infinite sequence of points
$P_1,P_2, \ldots \in \cA(\IQbar)$ such that
$\hat{h}_{\cA,\cL}(P_i)$ is constant and
$h_{S,\cM}(\pi(P_i))$ unbounded; just take
$P_i$ of finite order  in $\cA_{\pi(P_i)}$ and the
sequence $\pi(P_i)$ of unbounded height. 

The main technical result of this paper is a height inequality that
relates these two heights on an irreducible  subvariety $X$ of $\cA$. 
The discussion in the last paragraph suggests that we should at least 
remove all curves in $X$ that dominate $S$ and contain infinitely many
points of finite order. This turns out to be insufficient and we must
also remove subvarieties that are contained in
constant  abelian subschemes of $\cA$. 
In fact, we must remove precisely 
generically the special subvarieties of $\cA$ from
Definition~\ref{DefinitionSpecialSubvariety} that are contained in
$X$. 


\begin{theorem}\label{TheoremMainResultHeightBound}
Let $\pi\colon \cA\rightarrow S$, $\cL$ and $\cM$ be as above with
$k=\IQbar$ and $\dim S = 1$. Let $X$ be a closed irreducible subvariety of $\cA$ over $\IQbar$ and let $X^*$ be as above Proposition~\ref{TheoremStarSetOpen}. Then there exists $c>0$
such that
\begin{equation}
\label{eq:TheoremMainResultHeightBound}
h_{S,\cM}(\pi(P)) \le
c \left(1+\hat{h}_{\cA,\cL}(P) \right) \quad \text{for all}\quad P \in X^*(\IQbar).
\end{equation}
\end{theorem}

Suppose $X$ dominates $S$, so  we think of $X$ as a family of $(\dim
X-1)$-dimensional varieties. Then our height inequality \eqref{eq:TheoremMainResultHeightBound} can be
interpreted as a uniform
version of the Bogomolov Conjecture along the $1$-dimension base $S$
if  $h_{S,\cM}(\pi(P))\ge 2c$. Indeed, then
$\hat{h}_{\cA,\mathcal{L}}(P) \ge \frac{1}{2c} h_{S,\cM}(\pi(P))$. 
From this point of view
it would be interesting to have
an extension of Theorem \ref{TheoremMainResultHeightBound} to $\dim S
> 1$. For the main obstacle to pass from $\dim S = 1$ to the general case, we refer to $\mathsection$\ref{SubsectionStrategyIntro}, Part~1 and above.

Theorem~\ref{TheoremMainResultHeightBound} was proven by the
second-named author \cite{Hab:Special} when $\cA$ is a fibered power
of a non-isotrivial $1$-parameter family of elliptic curves. This
theorem had applications towards special points
problems \cite[Theorems 1.1 and 1.2]{Hab:Special} and towards some
cases of the relative Manin--Mumford
Conjecture \cite{hab:weierstrass}.

After this work was submitted, Ben Yaacov and Hrushovski informed the
authors of their similar height inequality for a $1$-parameter family
of genus $g\ge 2$ curves in an unpublished
note \cite{BenYaacovHrushovksi} by reducing it to Cinkir's result \cite{CinkirBogomolov}.



In this paper we treat arbitrary abelian schemes over algebraic
curves, possibly with non-trivial isotrivial part, and hope to extend
the aforemention applications in future work.

Before proceeding, we point out that we shall prove
Theorem~\ref{TheoremMainResultHeightBound} for a particular 
relatively ample line bundle $\cL$ on $\cA/S$ that is fiberwise
symmetric and a particular ample
line bundle $\cM$ on $\overline{S}$. Then
Theorem~\ref{TheoremMainResultHeightBound} holds for arbitrary such
$\cL$ and $\cM$ by formal properties of the  Height Machine. 
Moreover we will prove the following slightly stronger form of Theorem~\ref{TheoremMainResultHeightBound}. 

We may attach a third height function on $\cA$ in the following way. Let $\cL'
= \cL \otimes \pi^*\cM$. By \cite[Th\'eor\`eme~XI~1.4]{Raynaud:LNM119}
and  \cite[Corollaire~5.3.3 and Proposition~4.1.4]{EGA2}, our abelian
scheme admits a closed immersion
$\iota \colon \cA\rightarrow\IP_{\IQbar}^M\times S$ over $S$ arising
from $(\cL')^{\otimes n}$ for some $n \gg 1$. As we will see in \S\ref{SubsectionEmbeddingAbSch} the existence of a closed immersion is more straight-forward if we allow ourselves to remove finitely many points from $S$, a procedure that is harmless in view of our application. Define the \textit{naive height} of $P$ to be $h_{\cA,\cL'}(P) = \frac{1}{n}h(P')+h_{S,\cM}(\pi(P))$ where $\iota(P) = (P',\pi(P)) \in \IP^M_\IQbar(\IQbar)\times S(\IQbar)$.


\begin{thmbis}{TheoremMainResultHeightBound}\label{TheoremMainResultHeightBoundStronger}
Let $\pi\colon \cA\rightarrow S$ and $\iota$ be as above with $k
= \IQbar$ and $\dim S = 1$. Let $X$ be a closed irreducible subvariety of $\cA$ over $\IQbar$ and let $X^*$ be as above Proposition~\ref{TheoremStarSetOpen}. Then there exists $c>0$ such that
\[
h_{S,\cM}(\pi(P)) \le h_{\cA,\cL'}(P) \le c \left(1 + \hat{h}_{\cA,\cL}(P) \right) \quad \text{for all}\quad P \in X^*(\IQbar) .
\]
\end{thmbis}


\subsection{Outline of Proof of
Theorems \ref{TheoremGBC} and \ref{TheoremMainResultHeightBound} and Organization of the Paper}\label{SubsectionStrategyIntro}

We give an overview of the proof of
Theorem \ref{TheoremMainResultHeightBound} in three parts. 

Consider an abelian scheme $\pi\colon \cA\rightarrow S$ over a smooth
algebraic curve $S$ 
of relative dimension $g \ge 1$.

The Ax--Schanuel Theorem \cite{Ax72} is a function theoretic version of the famous
and open Schanuel Conjecture in transcendence theory. Stated for algebraic
groups, the case of an abelian variety deals with algebraic
independence of functions defined using
the uniformizing map. It has seen many
applications to problems in diophantine geometry~\cite{ZannierBook}. 

For our purpose we need an Ax-Schanuel property for families of abelian varieties which is not yet available. However the assumption $\dim S = 1$ simplifies the situation: Instead of the full power of functional transcendence, we only need to study a functional constancy property. The first part of the proof  deals with this functional constancy property 
where the so-called Betti map plays the role of the uniformizing map.
We briefly explain this map and refer to \S\ref{SectionBetti} for more
details.

\medskip
\textbf{Part 1: The Betti Map and a Functional Constancy Property.}
Any point of $S(\IC)$ has a
complex neighborhood that we can biholomorphically identify with the
open unit disc $\unitdisc\subset\IC$. The fiber of $\cA\rightarrow S$
above a point $s\in \unitdisc$ is biholomorphic to a complex torus
$ \IC^{g}/\Omega(s)\IZ^{2g}$ where the columns of
$\Omega(s) \in \mat{g,2g}{\IC}$
are a period lattice basis. Of course $\Omega(s)$ is not
unique. The choice of a period lattice basis $\Omega(s)$ enables us to
identify $\cA_s$ with $\IT^{2g}$ as real Lie groups, with $\IT$ the
unit circle in $\IC$. As $\unitdisc$ is simply connected, we can arrange
that the period map $s\mapsto \Omega(s)$ is holomorphic on $\unitdisc$.
In turn we can identify $\cA_\unitdisc = \pi^{-1}(\unitdisc)$ with the
constant family $\unitdisc \times \IT^{2g}$ as families over $\unitdisc$.
This can be done in a way that we get group isomorphisms
$\cA_s(\IC)\rightarrow\IT^{2g}$ fiberwise. 
Note that the complex structure is lost, and that the isomorphism
in play is only real analytic.

The \textit{Betti map} $b \colon \cA_\unitdisc\rightarrow\IT^{2g}$ is the
composite of the isomorphism $\cA_\unitdisc \cong \unitdisc\times\IT^{2g}$
with the projection to $\IT^{2g}$. It depends on several
choices, but two different Betti maps on
$\cA_\unitdisc$ differ at most by composing with a continuous
automorphism of $\IT^{2g}$.


Let $X$ be an irreducible closed subvariety of $\cA$ that dominates
$S$. 
In  \S\ref{SectionDegenerate} we
study the restriction of $b$ to $\an{X}$; the superscript ${}^\anE$ denotes complex analytification. 
We say that $X$
is \textit{degenerate} if the restriction of
$b\colon \cA_{\unitdisc} \rightarrow \IT^{2g}$ to $\an{X}\cap\cA_\unitdisc$ has positive
dimensional fibers on a non-empty, open subset $\an{X}$; being open
refers to the complex topology. The main result of \S\ref{SectionDegenerate},
Theorem~\ref{PropositionDegenerate}, states that a degenerate
subvariety is generically special. This will allow us to explain the
analytic notion of degeneracy in purely algebraic terms. 

Let us give some ideas of what goes into the proof of Theorem
\ref{PropositionDegenerate}.

In general, the period mapping $s\mapsto \Omega(s)$ cannot extend to
the full base due to monodromy. This obstruction is a powerful tool in
our context. Indeed, fix a base point $s\in \an{S}$.
Monodromy induces a representation on the fundamental group
$\pi_1(\an{S},s)\rightarrow \mathrm{Aut}(H_1(\cA_s^\anE,\IZ))$. 
 Moreover, by transporting along the
fibers of the Betti map above a loop in $\an{S}$ we
obtain a representation 
$\pi_1(\an{S},s)\rightarrow \mathrm{Aut}(\cA_s^{\anE})$
whose  target is the group of real analytic automorphisms of $\cA_s$.
This new representation induces the representation on homology.
Moreover, we can identify $\mathrm{Aut}(\cA_s^{\anE})$ with $\gl{2g}{\IZ}$ because $\cA_s$
and $\IT^{2g}$ are isomorphic in the real analytic category.
So the canonical mapping
$\mathrm{Aut}(\cA_s^{\anE})\rightarrow \mathrm{Aut}(H_1(\cA_s^{\anE},\IZ))$
is an isomorphism of groups. 

 
Now suppose that $X$ is degenerate. The assumption $\dim S = 1$ forces that $X^{\anE} \cap \cA_{\Delta} = b^{-1}(b(X^{\anE} \cap \cA_{\Delta}))$.\footnote{This is no longer true if $\dim S > 1$, making the remaining argument in this part fail for $\dim S > 1$.} In other words the fibers $X_s=\pi|_{X}^{-1}(s)$ do not depend on $s \in \Delta$ for the identification $\cA_s^{\anE} = \mathbb{T}^{2g}$. So the action of 
$\pi_1(S^{\anE},s)$ on $\cA_s^{\anE}$ leaves $X_s$
invariant. Thus it suffices to understand subsets of $\cA_s$ that are invariant
under the action of a subgroup of $\gl{2g}{\IZ}$. We use Deligne's
Theorem of the Fixed Part \cite{Deligne:Hodge2} and the Tits
Alternative \cite{tits:freesubgroups} to extract information from this subgroup. Indeed, under a
natural hypothesis on $\cA$, the image of the
representation in $\gl{2g}{\IZ}$ contains a free subgroup on two
generators. In particular, the image is a group of exponential growth.
We then use a variant of the Pila--Wilkie Counting
Theorem \cite{PilaWilkie}, due to Pila and the second-named author,
and Ax's Theorem \cite{Ax72} for a constant abelian variety. From this
we will be able to conclude that $X$ is generically special if it is
degenerate.

Let us step back and put some of these ideas into a historic
perspective. In the special case where $\cA$ is the fibered power of
the Legendre family of elliptic curves, the second-named
author \cite{Hab:Special}  used local monodromy to
investigate degenerate subvarieties. In the current work local
monodromy is insufficient as $S$ could be complete to begin with. So
we need global information. Zannier introduced the point
counting strategy and together with Pila gave a new proof of the
Manin--Mumford Conjecture \cite{PilaZannier} using the Pila--Wilkie
Theorem \cite{PilaWilkie}. Masser and Zannier \cite{MZ:AJM10} showed
the usefulness of the Betti coordinates for problems in diophantine
geometry by solving a first case of the relative Manin--Mumford
Conjecture. Ullmo and Yafaev \cite{UllmoYafeav:14} exploited
exponential growth in groups in combination with the Pila--Wilkie
Theorem to prove their hyperbolic Ax--Lindemann Theorem for projective
Shimura varieties. A recent result of Andr\'e, Corvaja, and
Zannier \cite{ACZ:Betti} also deals with the rank of the Betti map on the
moduli space of principally polarized abelian varieties of a given
dimension. More recently, Cantat, Xie, and the authors \cite{CGHX:18} gave
a different approach to the functional constancy problem that does not rely on
the Pila--Wilkie Theorem but rather on  results from
dynamical systems.


\medskip
\textbf{Part 2: Eliminating the N\'eron--Tate Height.}
The second part of the proof deals with reducing the height inequality
in Theorem \ref{TheoremMainResultHeightBound} to one that only
involves Weil heights; we refer to \S\ref{sec:heights} for
nomenclature on heights.

We will embed $S$ into $\IP^m$ and $\cA$ into $\IP^M\times\IP^m$ such
that $\pi\colon \cA\rightarrow S$ is compatible with the projection
$\IP^M\times\IP^m\rightarrow\IP^m$ and other technical conditions are fulfilled. 
Let $h(P,Q) = h(P)+h(Q)$ for
$P\in\IP^M(\IQbar), Q\in \IP^m(\IQbar)$, where $h$ denotes
the absolute logarithmic Weil height on projective space. 
For an integer $N$ let $[N]$ denote the multiplication-by-$N$ morphism $\cA\rightarrow\cA$. 

Let $N\ge 1$ be a sufficiently large integer (which we assume to be a
power of $2$ for convenience). 
If $X$ is not
generically special we show in 
Proposition \ref{PropositionHeightChangeUnderScalarMultiplication} that
\begin{equation}
\label{eq:introheightlb}
 N^2 h(P) \le c_1 h([N](P)) + c_2(N)
 \end{equation}
for all $P\in U(\IQbar)$ 
where $U$ is Zariski open and dense in $X$ and
where $c_1>0$ and $c_2(N)$ are both
independent of $P$. Note that $U$ and $c_2(N)$  may depend on $N$. 

One is tempted to divide (\ref{eq:introheightlb}) by $N^2$
and take the limit $N\rightarrow\infty$ as in Tate's Limit Process.
However, this is not possible \textit{a priori},  as $U$
and $c_2(N)$ could both depend on $N$. So we mimic Masser's strategy of
``killing Zimmer constants''
explained in \cite[Appendix~C]{ZannierBook}.
This step is carried out in \S\ref{SectionNTBase} where we 
terminate Tate's Limit Process after finitely many steps when $N$ is
large enough in terms of $c_1$; for this it is
crucial that  $c_1$ is independent of $N$. 

\medskip
{\bf Part 3: Counting Lattice Points and an Inequality for the Weil Height.}
At this stage we have reduced proving
Theorem \ref{TheoremMainResultHeightBound} to 
 (\ref{eq:introheightlb})
if $X$ is not generically special.
Recall that from Part 1 of the proof we know that $X$ is not
degenerate.
Therefore,  the restricted Betti map
$b|_{\an{X}}:\an{X}\rightarrow\IT^{2g}$  has
discrete fibers. The image of this restriction has the same real
dimension as $\an{X}$ (the dimension is well-defined as the image is
subanalytic).

\smallskip
{\bf Part 3a: The Hypersurface Case.}
To warm up let us assume for the moment that $X$ is a hypersurface in $\cA$, so
$\dim X = g$. In this case the image  $b(\an{X}\cap\cA_\unitdisc)$ contains a
non-empty, open subset of $\IT^{2g}$. By a simple Geometry of Numbers
argument in the covering $\IR^{2g}\rightarrow\IT^{2g}$, the image contains $\gg N^{2g}$ points of  order
dividing $N$; the implicit constant is independent of $N$.
As the Betti map is a group isomorphism on each fiber of
$\cA$ we find that $X$ contains $\gg N^{2g}$ points of order dividing
$N$.

Obtaining
(\ref{eq:introheightlb}) requires an auxiliary rational map
$\varphi\colon \cA\dashrightarrow\IP^{g}$. Suppose for simplicity that
we can choose $\varphi$ such that $\varphi^{-1}([1:0:\cdots:0])$ is the image of the zero section
$S\rightarrow \cA$. Then
the composition $\varphi \circ [N]$ restricts to a rational
map $X\dashrightarrow\IP^{g}$
that maps the $\gg N^{2g}$ torsion points constructed above to
$[1:0:\cdots:0]$.

If we are lucky and all these torsion points are
isolated in the fiber of $\varphi\circ [N]$, then
$\deg(\varphi\circ[N]) \gg N^{2g}$.
A height-theoretic
lemma \cite{Hab:Special}, restated here as
Lemma~\ref{LemmaChangeOfHeightUnderRationalMap}, implies
(\ref{eq:introheightlb}) for $N$ a power of $2$.
The factor $N^2$ on the left in
(\ref{eq:introheightlb}) equals $N^{2\dim X}/N^{2(\dim X-1)}$ and has the
following interpretation. The numerator comes from the degree lower
bound as $\dim X = g$. The
denominator is a consequence of the following fact.
Given a suitable embedding of an abelian variety into some projective
space, the duplication morphism can be described by a collection of homogeneous
polynomials of degree $2^2=4$. So $[N]$ can be described  by homogeneous
polynomials of degree  $\le N^2$. 

If we are less lucky and some torsion point is not isolated in
$\varphi\circ [N]$, then an irreducible component of $\ker
[N]\subset\cA$
is contained in $X$. This situation is quite harmless, as roughly
speaking, it  cannot happen too often for a variety that is not generically
special.

The restriction $\dim X = g$ is more serious, however.
The second-named author was able to reduce \cite{Hab:Special}
to the hypersurface case  inside a fibered power of the Legendre
family of elliptic curves.  This is not possible
for general $\cA$,
 so we must proceed differently. 

\smallskip
{\bf Part 3b: The General Case.}
For general $X$ we will still  construct a suitable
$\varphi:X\dashrightarrow\IP^{\dim X}$ as above and apply
Lemma~\ref{LemmaChangeOfHeightUnderRationalMap}. As a stepping stone
we first construct in \S\ref{SectionAuxiliarySubvariety} an
auxiliary subvariety $Z$ of $\cA$ in sufficiently general position
such that
\begin{equation}
\label{eq:dimauxvariety}
\dim X +\dim Z =\dim \cA =
g+1.
\end{equation}
The rational map $\varphi$ is constructed using $Z$ 
in \S\ref{SectionHtTotalSpace}, and one should think of $Z$ as an irreducible
component of $\varphi|_X^{-1}([1:0:\cdots:0])$.
Being in general position and (\ref{eq:dimauxvariety}) mean that
$\varphi|_X$ has finite generic fiber on its domain.

Now we  want $\varphi\circ[N] \colon X\rightarrow \IP^{\dim X}$ to have degree  $\gg N^{2\dim X}$ and
for this 
it suffices to find  $\gg N^{2\dim X}$  isolated 
points in the preimage of $[1:0:\cdots:0]$.
As $\varphi(Z) = [1:0:\cdots:0]$ we  need to find
$\gg N^{2\dim X}$ isolated points in 
$X\cap [N]^{-1}(Z)$. (Additional verifications
must be made to ensure 
that isolated intersection points lead to isolated fibers.)



We ultimately construct these points using  the Geometry of Numbers.
More precisely, we need a  volume estimate and 
 Blichfeldt's Theorem.
Since $X$ is not degenerate we have a point around which the local
behavior of $X$ is similar to the local behavior of its image in
$\IT^{2g}$ under the Betti map. This allows us to linearize the
problem as follows. In order to count the number of such $X\cap
[N]^{-1}(Z)$, we can instead count points $x\in\IT^{2g}$, coming from
points of $X(\IC)$ under the Betti map, such that $Nx=z$ lies in the
image of $Z(\IC)$. If we let $x,z$ range over the image under the
Betti map of small enough open subsets of $X(\IC)$ and $Z(\IC)$, then
$Nx-z$ ranges over an open subset of $\IT^{2g}$.
This conclusion makes crucial use of the fact that $X$ is not
degenerate  and that $Z$ is in general position. Lifting via
the natural map $\IR^{2g}\rightarrow\IT^{2g}$ we are led to the
  counting lattices points. Indeed, we must construct elements of 
$N\tilde
x-\tilde z \in \IZ^{2g}$ where $\tilde x,\tilde z$ are
lifts of points $x,z$ as before.
We denote  the set of all possible $N \tilde x - \tilde z$ by $U_N$.
A careful volume estimate done in   \S\ref{SectionLatticePoints}
leads to
 $\vol{U_N}\gg N^{2\dim X}$. So we expect to find this many
lattice points.
But there is no reason to believe that $U_N$ is convex and
it is not hard to imagine open subsets of $\IR^{2g}$ of
arbitrary large volume that meet $\IZ^{2g}$ in the empty set. 
To solve this problem we apply Blichfeldt's Theorem, which claims that 
 some translate $\gamma+U_N$ of $U_N$ contains at least 
$\mathrm{vol}(U_N)$ lattice points in $\IZ^{2g}$. 

This approach  ultimately constructs
enough points to prove a suitable lower bound for the degree of
$\varphi\circ [N]$ and to complete the proof.
However, additional difficulties arise. For example, we must  deal with non-zero $\gamma$ and making sure that the
points constructed  are isolated  in $X\cap [N]^{-1}(Z)$.
These technicalities  are addressed
in \S\ref{SectionIntersection}.




\medskip
\textbf{The Remaining Results.}
In $\mathsection$\ref{SectionStarSetOpen} we prove
Proposition~\ref{TheoremStarSetOpen}. This section is mainly
self-contained and the main tool is a uniform version of the
Manin--Mumford Conjecture in characteristic $0$.

The proof of Theorem~\ref{TheoremGBC} in \S\ref{SectionGBC} follows
the blueprint laid out in~\cite{Hab:Special}. We need to combine our
height bound, Theorem~\ref{TheoremMainResultHeightBound}, with
Silverman's Height Limit Theorem \cite{Silverman} used in his
specialization result.

In Appendix~\ref{sec:appendixA} we sketch how to adapt our height
inequality, Theorem~\ref{TheoremMainResultHeightBound}, to more
general 
fields in characteristic $0$. 
This shows how to deduce Theorem~\ref{TheoremGBC} for any algebraically closed field
of
characteristic $0$. Appendix~\ref{sec:appendixB} contains some
comments on the situation when $\dim S>1$. Finally, in
Appendix~\ref{sec:appendixC} we give a self-contained and quantitative
version of Brotbek's Hyperbolicity Theorem~\cite{DamienHyperbolicity}
in the case of an abelian variety (which is much simpler than the
general case).

\subsection*{Acknowledgements} Both authors would like to thank the
organizers of the conference ``On Lang and Vojta's conjectures'' in
Luminy 2014, where our collaboration began. They would like to thank
Bas Edixhoven for providing the current proof of
Proposition~\ref{prop:uniformintersectionbound}. They would like to
thank Shouwu Zhang for providing the proof using the essential minimum
at the end of Appendix~\ref{sec:appendixA}. The authors also thank
Gabriel Dill and Walter Gubler for useful remarks. ZG would like to
thank the Department of Mathematics and Computer Science at the
University of Basel for the invitation and for the stimulating
atmosphere and SNF grant Nr. 165525 for financial support.


\section{Notation}
\label{sec:basicnotation}

Let $\IN = \{1,2,3,\ldots\}$ denote the set of positive integers.
Let $\IQbar$ be the algebraic closure of $\IQ$ in $\IC$. 


If $X$ is a variety defined over $\IC$, then we write
$\an{X}$ for $X(\IC)$ with its structure as a complex analytic space,
we refer to Grauert and Remmert's book \cite{CAS} for the theory
of complex analytic spaces. 

Given an abelian scheme $\cA$ over any base scheme 
and an integer $N$, we let
 $[N]$ denote the multiplication-by-$N$
morphism $\cA\rightarrow\cA$.
The kernel of $[N]$ is $\cA[N]$, it is a group scheme over the base of
 $\cA$. An endomorphism of $\cA$ is a morphism $\cA\rightarrow\cA$
 that takes the zero section to itself. 

If $A$ is an abelian variety over field $K$ and if $\overline K\supset
K$ is a given algebraic closure of $K$, then 
$\tor{A}$ denotes the group of points of finite order of $A(\overline K)$. 

Suppose $k$ is a subfield of $K$ whose algebraic closure in $K$ equals
$k$ and $\mathrm{char}(k)=0$.
We write $A^{K/k}$ for the $K/k$-trace of $A$ and let
$\tau_{A,K/k}\colon A^{K/k}\otimes_k K\rightarrow A$
denote the associated trace map,
we refer to \cite[$\mathsection$6]{Conrad} 
 for general facts and the universal property. 
Note that our notation $A^{K/k}$ is denoted by $\mathrm{Tr}_{K/k}(A)$ in \textit{loc.cit.}
By \cite[Theorem~6.2 and below]{Conrad}  $\tau_{A,K/k}$ is a closed immersion
since $\mathrm{char}(k)=0$.
We sometimes consider $A^{K/k}\otimes_k K$ as an abelian subvariety of
$A$. 

By abuse of notation we sometimes abbreviate $(A\otimes_K\overline
K)^{\overline K/k}$ 
by $A^{\overline K/k}$ 
and, in this notation,  consider $A^{\overline{K}/k} \otimes_k \overline{K}$ as an abelian subvariety of $A\otimes_k \overline{K}$.

\subsection{Heights}
\label{sec:heights}

A \textit{place} of a number field $K$ is an absolute value
$|\cdot|_v:K\rightarrow [0,\infty)$ whose restriction to $\IQ$ is
either the standard absolute value or a $p$-adic absolute value for
some prime $p$ with $|p|=p^{-1}$.
We set $d_v = [K_v:\IR]$ in the former and
$d_v = [K_v:\IQ_p]$ in the latter case.
The \textit{absolute, logarithmic, projective Weil height}, or just \textit{height}, of a
point $P = [p_0:\ldots:p_n]\in \IP_\IQ^n(K)$ with
$p_0,\ldots,p_n\in K$ is 
\begin{equation*}
h(P) = \frac{1}{[K:\IQ]}\sum_{v} d_v\log\max\{|p_0|_v,\ldots,|p_n|_v\}
\end{equation*}
where the sum runs over all places $v$ of $K$.
The value $h(P)$ is independent of the choice of projective
coordinates by the product formula. For this and other basic facts we
refer to \cite[Chapter 1]{BG}. 
Moreover, the height
does not change when replacing $K$ by another number
field that contains the $p_0,\ldots,p_n$.
Therefore, $h(\cdot)$ is well-defined on $\IP_\IQ^n(\overline K)$
where $\overline K$ is an algebraic closure of $K$.

In this paper we also require heights in a function field $K$. With our
 results in mind, we restrict to the case where $K=k(S)$ and
$S$ is a smooth projective irreducible curve over an algebraically
closed field $k$.
Let $\overline K$ be an algebraic closure of $K$. 
In this case, we can construct a height
$h_{K} : \IP_K^n(\overline K)\rightarrow \IR$ as follows.

The points
$S(k)$ correspond to the set of places $|\cdot|_v$ of $K$. They  extend in the usual manner to finite extensions of
$K$. If $P = [p_0:\cdots:p_n]\in \IP_K^n(\overline K)$ with
$p_0,\ldots,p_n\in K'$, where $K'$ is a finite extension of $K$, then we set
\begin{equation*}
h_K(P) = \frac{1}{[K':K]}\sum_{v} d_v \log\max\{|p_0|_v,\ldots,|p_n|_v\}
\end{equation*}
where $d_v$ are again local degrees such that the product formula
holds. We refer to  \cite[$\mathsection$1.3. and 1.4.6]{BG} for more details 
 or
\cite[$\mathsection$8]{Conrad} on 
generalized global fields.
In the function field case we keep $K$ in the subscript of $h_K$ to emphasize that $K$ is our
base field. 
Indeed, in the function field setting one must  keep
track of the base field ``at the bottom'' that plays the role
of $\IQ$ in the number field setting.

Now let $K$ be either a number field or a function field as above.
Suppose that $A$ is an abelian variety defined over $K$ 
that is embedding in some projective space $\IP_K^M$ with a symmetric line
bundle. 
Tate's Limit Process
induces the
\emph{N\'eron--Tate} or \emph{canonical height} on $A(\overline K)$.
If $K$ is a number field, we write
\begin{equation}
\label{eq:NTheight}
\hat h_{A}(P) = \lim_{N\rightarrow \infty} \frac{h([2^N](P))}{4^N}
\end{equation}
for the N\'eron--Tate height on $A(\overline K)$, we refer to 
\cite[Chapter 9.2]{BG} for details. The N\'eron--Tate height depends
also on the choice of the symmetric, ample line bundle, but we
do not mention it  in $\hat h_{A}$.

The construction in the function field is the same. For the same
reason as above
we  retain the symbol $K$
and write
$\hat h_{A,K}$ for the N\'eron--Tate height on $A(\overline K)$.

\subsection{Embedding our Abelian Scheme}\label{SubsectionEmbeddingAbSch}
In the paper we are often in the following situation. Let $k$ be an algebraically closed subfield of $\IC$. Let $S$ be a smooth irreducible algebraic curve over $k$ and
let $\cA$ be an abelian scheme of relative dimension $g\ge 1$
over $S$ with
structural morphism $\pi \colon \cA\rightarrow S$.

Let us now  see how
to  embed  $\cA$ into $\IP_S^M = \IP_k^M \times S$ for some $M>0$
after possibly removing finitely many points from $S$.
Note that removing finitely many points is harmless in the context of
our problems. Indeed, our  Theorem \ref{TheoremMainResultHeightBound} is not weakened by this 
 action and so we do it at leisure.

The generic fiber $A$ of $\cA \rightarrow S$ is an abelian variety defined
over the function field of $S$. 
Let $L$ be a symmetric ample line bundle on $A$. Then $L^{\otimes 3}$ is very ample. Replace $L$ by
$L^{\otimes 3g}$. A basis of $H^0(A,L)$ gives a projectively normal
closed immersion $A \rightarrow \IP^M_{k(S)}$ for some $M>0$.

We take the scheme theoretic image $\cA'$ of
$A\rightarrow\IP_{k(S)}^M\rightarrow \IP_S^M$,
hence $\cA'$ is the Zariski closure of
the image of $A$ in $\IP_S^M$ with the reduced induced structure.
After removing finitely many points of $S$ we obtain an abelian scheme
 $\cA'\subset \IP_S^M$ such that the morphism from $A$ to the generic
 fiber
 of $\cA'\rightarrow S$ is an isomorphism.
An  abelian scheme over $S$ is the
N\'eron model of its generic fiber, so the N\'eron mapping property
holds.
Therefore the canonical morphism
$\cA\rightarrow \cA'$ is an isomorphism
and we have thus constructed a closed immersion
\begin{equation*}
\iota_S \colon \cA \rightarrow \IP^M_S = \IP_k^M \times S.
\end{equation*}
Note that $L$ is the generic fiber of the relatively very ample line bundle $\cL = \iota_S^* \cO_S(1)$ on $\cA/S$. Moreover  for any $s \in S(k)$, we have that $\cL_s$ is the $g$-th tensor-power of a very ample line bundle on $\cA_s$. 

We may furthermore find an immersion (which need not be open or
closed) of $S$ into some $\IP_k^m$.
Composing yields the desired immersion
$\cA\rightarrow\IP_k^M\times\IP_k^m$.
By abuse of notation we consider $A\subset \IP_{k(S)}^M$ and
$\cA\subset\IP_k^M\times\IP_k^m$ from now on. Let us recapitulate.
\begin{enumerate}
\item [(A1)]We have an immersion
$\cA\rightarrow\IP_k^M\times\IP_k^m$ such that
the diagram involving $\pi\colon \cA\rightarrow S$
and the projection
$\IP_k^M\times\IP_k^m\rightarrow\IP_k^m$ commutes.
Moreover, for all $s\in S(k)$ the closed immersion $\cA_s\rightarrow \IP_k^M$
 is induced by a symmetric very ample line bundle. 
\end{enumerate}

Of course this immersion depends on the choice of the immersions of $A$
and of $S$.

The image of $A$ in $\IP_{k(S)}^M$ is projectively normal
and $[2]^* L$ is isomorphic to $L^{\otimes 4}$. 
Therefore, $[2]$ is represented globally by $M+1$ homogeneous polynomials of
degree $4$ on the image of $A$. Here the base field is the function
field $k(S)$. But we can extend it to the model after possibly
removing finitely many points of $S$. 
So 
 we may assume the following.
\begin{enumerate}
\item [(A2)] The morphism
$[2]$ is
represented globally on $\cA\subset\IP_k^M\times\IP_k^m$ by
$M+1$ bi-homogeneous polynomials, homogeneous of degree $4$ in
the projective coordinates of $\IP_k^M$
and homogeneous of a certain degree in the
projective coordinates of $\IP_k^m$.
\end{enumerate}

Finally, we explain why we took the additional factor $g$ in the
exponent $3g$ of $L^{\otimes 3g}$. By Proposition \ref{prop:abvarhyperbolic} 
 we have the additional and useful property. 
\begin{enumerate}
\item [(A3)] For given $s\in S(k)$ and $P\in \cA_s$, any generic hyperplane section of 
$\cA_s$ passing through $P$ does not contain a positive dimensional coset in
 $\cA_s$. 
\end{enumerate}
At the cost of possibly increasing the factor $g$ we could
also refer to Brotbek's deep result \cite{DamienHyperbolicity} for
more general projective varieties.

An immersion  $\iota \colon \cA\rightarrow\IP_k^M\times\IP_k^m$
for which (A1), (A2), and (A3) above are
satisfied
 will be called \textit{admissible}.

The construction above adapts easily to show the following fact. Let
$A$ be an abelian variety defined over $k(S)$. After possibly
shrinking $S$ we can realize $A$ as
the generic fiber of an abelian scheme $\cA\rightarrow S$
with an admissible immersion $\cA\rightarrow \IP_k^M\times\IP_k^m$.

If $k=\IQbar$   we have two  height
functions on $\cA(\IQbar)$. 

Say $P\in \cA(\IQbar)$,  we write $P=(P',\pi(P))$ with
$P'\in \IP_\IQbar^M(\IQbar)$ and
$\pi(P)\in \IP_\IQbar^m(\IQbar)$. Then
\begin{equation}\label{EqHeightTotal}
h(P) = h(P')+h(\pi(P))
\end{equation}
defines our first height $\cA(\IQbar)\rightarrow [0,\infty)$ which
we call the \textit{naive height} on $\cA$ (relative to the immersion
$\cA\subset\IP_\IQbar^M\times\IP_\IQbar^m$).

The second height is the \textit{fiberwise N\'eron--Tate}
or \textit{canonical height}
\begin{equation}
\label{eq:fiberwiseNT}
\hat h_{\cA}(P) = \hat h_{\cA_{\pi(P)}} (P),
\end{equation}
\textit{cf.} (\ref{eq:NTheight}).
We obtain a function
$\hat h_{\cA} : \cA(\IQbar)\rightarrow [0,\infty)$. It is quadratic
on each fiber
as the  line bundle on the generic fiber $A$ is symmetric and this 
extends along the fibers of $\cA\rightarrow S$. 

In the end we explain these height functions in terms of Height
Machine. Let $\overline{\cA}$, resp. $\overline{S}$, be the Zariski
closure of the image of the immersion
$\cA \subset \IP_\IQbar^M \times \IP_\IQbar^m$, resp.
$S \subset \IP_\IQbar^m$. If we let $\cL'= \cO(1,1)|_{\overline{\cA}}$
and $\cM = \cO(1)|_{\overline{S}}$, then
$h(\cdot)$ represents the class of functions $h_{\cA,\cL'}$ defined
up to $O(1)$ and $h\circ\pi$ represents $h_{S,\cM}\circ \pi$. And the
fiberwise N\'eron--Tate height $\hat{h}_{\cA}$ is the map $P \mapsto \hat{h}_{\cA_{\pi(P)},\cL_{\pi(P)}}(P)$, where $\cL = \iota_S^* \cO_S(1)$ as above.


\section{Proof of Proposition~\ref{TheoremStarSetOpen}}\label{SectionStarSetOpen}
The goal of this section is to prove Proposition~\ref{TheoremStarSetOpen}. In fact we will prove a statement of independent interest that implies Proposition~\ref{TheoremStarSetOpen}.

Let $k$ be an algebraically closed field of characteristic $0$. Let
$S$ be a smooth irreducible curve over $k$ and fix an algebraic
closure $\overline{K}$ of the function field $K=k(S)$. Let $A$ be an abelian variety over $\overline{K}$.

Furthermore, let $V_0$ be an irreducible variety defined over $k$ and $V = V_0\otimes_k  \overline{K}$. We consider $V_0(k)$ as a subset of $V(\overline K)$. 

The next proposition characterizes subvarieties $V\times A$ that contain a Zariski dense set of points in $\Sigma = V_0(k)\times A_{\mathrm{tor}} \subset V(\overline
K)\times A(\overline K)$. See Yamaki's \cite[Proposition 4.6]{YamakiBogomolovTrace} for a related statement. 

\begin{proposition}
\label{prop:specialpoints}
Suppose $A^{\overline K/k}=0$
and let $Y$ be
  an irreducible closed subvariety of $V\times A$. 
\begin{enumerate}
\item [(i)]
If $Y(\overline K) \cap
  \Sigma$ is Zariski dense in $Y$,
then 
$Y = (W_0\otimes_k  \overline K)\times (P+B)$ with $W_0\subset V_0$  an
irreducible closed subvariety, $P\in \tor{A},$ and
$B$ an abelian subvariety of $A$. 
\item[(ii)] 
There are at most finitely many subvarieties of the form 
$(W_0\otimes_k  \overline K)\times (P+B)$ (with $W_0,P,$ and $B$
as in (i)) that are contained in $Y$, maximal with respect to the
inclusion 
for this property.
\end{enumerate}
\end{proposition}

Part (i) implies (ii) for the following reason. 
If $W_0,P,$ and $B$ are as in the conclusion of (i), then it suffices
to observe that
$((W_0\otimes_k \overline K)\times (P+B))(\overline{K})\cap\Sigma$
is Zariski dense in $(W_0\otimes_k \overline K)\times (P+B)$. 

The assumption $\dim S=1$ is used only at one place in the proof. In Appendix~B, we will explain how to remove it.

\subsection{Proposition~\ref{prop:specialpoints} implies Proposition~\ref{TheoremStarSetOpen}}
Now we go back to the setting of Proposition~\ref{TheoremStarSetOpen}:
$S$ is a smooth irreducible curve over $k$ and
$\pi \colon \mathcal{A} \rightarrow S$ is an abelian scheme of
relative dimension $g\ge 1$.
Let $A$ denote the geometric generic fiber of $\pi$, it is an abelian
variety over $\overline K$. 

By \cite[Theorem~6.4 and below]{Conrad} there is a unique abelian subvariety
$A'\subset A$ such that $(A/A')^{\overline K/k}=0$
and such that we may identify $A'$ with 
$A^{\overline K/k}\otimes_k\overline K$. 

We  fix an 
abelian subvariety $A''\subset A$
with $A'+A''=A$ and such that $A'\cap A''$ is finite.
Then the addition morphism restricts to an isogeny 
$\psi \colon A' \times A''\rightarrow A$ and $(A'')^{\overline K/k}=0$.



Let $W_0$ be an irreducible closed subvariety of $A^{\overline K/k}$, $B$
 an abelian subvariety of $A''$, and $P\in A''(\overline K)$. 
We can map 
 $\psi((W_0\otimes_k\overline K)\times (P+B))\subset A$
to the generic fiber of $\cA$; its Zariski closure is a generically
special subvariety of $\cA$. 
Conversely, any generically special subvariety of $\cA$ arises this
 way.

Let $\cX\subset \cA$ be an irreducible closed subvariety that
dominates $S$ 
and $X\subset A$ its geometric generic fiber. 
We apply Proposition \ref{prop:specialpoints} where $A^{\overline
K/k},A''$
play the role of $V_0,A$ respectively.
There are at most finitely many subvarieties of $A$ of the form 
$(W_0\otimes_k\overline K)\times (P+B)$ 
that are contained in $\psi^{-1}(X)$, maximal for this property. 
This shows that there are at most finitely many generically
special subvarieties of $\cA$ that are contained in $\cX$, maximal for this property. 
\qed

\subsection{Proof of Proposition~\ref{prop:specialpoints}} Now we prove Proposition~\ref{prop:specialpoints}. To do this we require a uniform version of the
Manin--Mumford Conjecture in characteristic $0$.

\begin{theorem}[Raynaud, Hindry, Hrushovski, Scanlon] 
\label{thm:MMuniform}
Let $\overline K$ be as above, let $A$ be  an abelian variety, and
let $V$ be  an irreducible,
quasi-projective variety, both  defined over
  $\overline K$. 
Suppose  $Y$ is an irreducible closed
  subvariety of $V\times A$. For $v\in V(\overline K)$ we let
$Y_v$ denote the projection of $V\cap(\{v\}\times A)$ to $A$. 
Then there exists a finite set $M$ of abelian subvarieties of $A$ and
$D\in\IZ$ with
the following property. For all $v\in V(\overline K)$ the Zariski
closure 
of 
$Y_v(\overline K) \cap \tor{A}$ in $Y_v$ 
is a union of at most $D$ translates of members of $M$ by points of
finite order in $\tor{A}$. 
\end{theorem}
\begin{proof}
  Raynaud proved the Manin--Mumford Conjecture  in
  characteristic zero. 
Automatical uniformity then follows from Scanlon's 
\cite[Theorem 2.4]{Scanlon:AU}; see also 
 work of Hrushovski \cite{HrushovskiUniformMM} and Hindry's  \cite[Th\'eor\`eme~1]{Hindry:Lang} for $k=\overline{\IQ}$. 
 Indeed, the number of irreducible components 
 is uniformly bounded in an algebraic family. 
Moreover, it is well-known that
if an irreducible component of a fiber of an algebraic family is
a coset in $A$, then only
 finitely many possible underlying
 abelian subvarieties arise as one varies over the fibers.
\end{proof}

\begin{proof}[Proof of Proposition \ref{prop:specialpoints}]
We have already seen that it suffices to prove (i). 
We keep the notation $Y_v$ for fibers of $Y$ above $v\in V(\overline K)$ introduced in
Theorem \ref{thm:MMuniform}. Let $M$ and $D$ be as in this theorem.


For all $v\in V(\overline K)$ we have 
\begin{equation*}
  \overline{Y_v(\overline K)\cap \tor{A}} = \bigcup_{i} (P_{v,i} + B_{v,i})
\end{equation*}
where $P_{v,i}\in \tor{A}, B_{v,i} \in M,$ and the union has at most $D$
members. 
Note that for $v\in V(\overline K)$ any torsion coset contained in
$Y_v$ is contained in some 
$P_{v,i}+B_{v,i}$. 
Moreover, 
\begin{equation}
\label{eq:YcapSigmadoubleunion}
  Y(\overline K)\cap \Sigma = \bigcup_{v\in V_0(k)}\bigcup_{i}
\{v\}\times \left(P_{v,i} + \tor{(B_{v,i})}\right).
\end{equation}
We decompose $Y(\overline K)\cap \Sigma$ into a finite union of
\begin{equation*}
  \Sigma_B = \bigcup_{v\in V_0(k)  }\bigcup_{\substack{i \\  B_{v,i}=B}}
  \{v\}\times (P_{v,i} +\tor{B}) 
\end{equation*}
by collecting 
entries on the right of \eqref{eq:YcapSigmadoubleunion} that come from
$B\in M$.
The set $\Sigma_B$ must be Zariski dense in $Y$ 
for  some $B\in M$.
There is possibly more than one such 
 $B$  so 
we  choose  one that is maximal with respect to  inclusion.

We fix a finite field extension $F/K$ with $F\subset\overline K$, such
that $Y,A,$ and  $B$ are stable under the action of $\gal{\overline
K/F}$. 
For this proof we consider these three varieties and $V$ as over $F$.
Note that $A^{\overline K/k}=0$ remains valid.

Now suppose $v\in V_0(k)$ and let  $P_{v,i}$ be as in the definition of $\Sigma_B$, hence $B_{v,i}=B$ and
$P_{v,i}+B\subset Y$.

For all $\sigma \in \gal{\overline K/F}$ we have 
$\sigma(P_{v,i})+B=\sigma(P_{v,i}+B)  \subset \sigma(Y_v)=Y_v$, by our
choice of $F$ and since $\sigma$ acts trivially on $V_0(k)$.
So the torsion coset $\sigma(P_{v,i})+B$ is contained in  $P_{v,j} + B'$ for some $B'\in M$ and some
$j$. This implies $B\subset B'$. If $B\subsetneq B'$, then by maximality of
$B$ the Zariski closure $\overline{\Sigma_{B'}}$ is not all of $Y$.
After replacing $F$ by a finite extension of itself we may
assume $\sigma(\overline{\Sigma_{B'}}) =\overline{\Sigma_{B'}}$
for all $\sigma\in \gal{\overline K/F}$. 
In particular, $\{v\}\times (P_{v,i}+B) \subset \overline{\Sigma_{B'}}$. 
We remove such torsion cosets from the
union defining $\Sigma_B$ to obtain a
set  $\Sigma'\subset Y(\overline K)\cap \Sigma$ that remains Zariski
dense in $Y$. 

If $P_{v,i}+B$ is in the union defining $\Sigma'$,
then $B=B'$ and 
 $\sigma(P_{v,i}+B) = P_{v,j}+B$ for some $j$
and there are at most $D$ possibilities for $\sigma(P_{v,i}+B)$
with $\sigma \in \gal{\Kbar/F}$. 

Let $\varphi \colon A\rightarrow A/B$ be the canonical map, then
$\sigma(\varphi(P_{v,i})) =
\varphi(\sigma(P_{v,i})) = \varphi(P_{v,j})$.
We have proven that 
the  Galois orbit of $\varphi(P_{v,i})$ has at most $D$ elements, 
in particular $[F(\varphi(P_{v,i})):F]\le D$; recall that $D$
 is
 independent of $v$ and $i$. 

\medskip

\noindent {\bf Claim:} 
Without loss of generality we may assume that the torsion points $P_{v,i}$ 
contributing to $\Sigma'$ have uniformly bounded order. 
\medskip

Indeed, we may replace each $P_{v,i}$ by an element of $P_{v,i}+\tor{B}$. So by a
standard argument involving a complement of $B$ in $A$ it is
enough to show the following statement: 
The order of any point  in 
\begin{equation}
\label{eq:torsionboundedset}
  \{ P\in\tor{(A/B)} : [F(P):F]\le D \}
\end{equation}
 is
bounded in terms of $A/B$ and $D$ only. 

This is the only place in the proof of
 Proposition~\ref{prop:specialpoints}
 where we  use the hypothesis $\dim S=1$. In Appendix~B we will explain how to remove this hypothesis.

Let $\overline{S}'$ be an irreducible smooth projective curve with
$k(\overline{S}') = F$ and $P$ as in (\ref{eq:torsionboundedset}). The inclusion
$F \subset F(P)$ corresponds to a  finite covering
$\overline{S}'' \rightarrow \overline{S}'$ of degree $[F(P):F]$
where $\overline{S}''$ is another smooth projective curve with
function field $F(P)$. 
Then
$A/B$ has good reduction above $S'(k)\setminus Z$ for some finite
subset $Z$ of $S'(k)$, where we have identified $\overline{S'}(k)$ with
the set of places of $F$. Note that $S'$ and $Z$ are independent of $P$. All
residue characteristics are zero, so by general reduction theory of
abelian varieties we find that $F(P)/F$ is unramified above the places
in $S'(k) \setminus Z$. In other words, the finite morphism
$\overline{S}'' \rightarrow \overline{S}'$ is unramified above
$S' \setminus Z$. So we get a finite \'{e}tale covering of $S \setminus Z$ of degree $[F(P):F]$.

Let $L$ be the compositum  in $\overline K$ of all extensions of $F$ of degree
at most $D$ that are unramified above $S' \setminus Z$. Then
$L/F$  a finite
field extension
by \cite[Corollary~7.11]{Voelklein} if $k\subset\IC$
and for general $k$ of characteristic $0$
since the \'etale fundamental group of $S'\setminus Z$
is topologically finitely generated by
\cite[Expos\'e~XIII Corollaire 2.12]{SGA1}. 
In particular, $P \in (A/B)(L)$ for all $P$ in \eqref{eq:torsionboundedset}.

Now $(A/B)^{\overline K/k}=0 $  as the same holds for $A$.
The extension $L/k$ is finitely generated,
 so the Lang--N\'eron Theorem, cf. \cite[Theorem 1]{LangNeron} or  \cite[Theorem~7.1]{Conrad}, implies that $(A/B)(L)$ is a finitely generated group. Thus  $[N](P)=0$ for some $N\in\IN$ that is independent of $P$. Our claim follows. 

\medskip

Define a morphism
$\psi \colon V\times A\rightarrow V\times (A/B)$
by $\psi(v,t) = (v, [N]\circ\varphi(P))$. By choice of $N$ we
have $\psi(\Sigma') \subset V\times \{0\}$, so $\Sigma'
\subset V \times (\Theta + B)$ where $\Theta \subset \tor{A}$ is
finite. 
We pass to the Zariski closure and  find $Y\subset V\times (P+B)$ for
some $P\in \tor{A}$ as $Y$ is irreducible.

Let $p \colon V\times A\rightarrow V$ be the first projection, 
it is proper and $p(Y)$ is
Zariski closed in $V$. 
A fiber of $p|_{Y}$ containing a point of $\Sigma'$ contains a
subvariety of dimension $\dim B$. 
We use  that $\Sigma'$ is Zariski dense in $Y$ one last time together
with the Fiber Dimension Theorem \cite[Exercise II.3.22]{Hartshorne} to conclude $\dim B \le \dim Y - \dim
p(Y)$. As $Y\subset p(Y) \times (P+B)$ 
we conclude
\begin{equation}
\label{eq:Yisaproduct}
Y =  p(Y)  \times (P+B). 
\end{equation}

Finally,
$p(\Sigma')$ is Zariski dense in $p(Y)\subset V$.  But $p(\Sigma')$ consists of elements in
$V_0(k)$, with $k$  the base field of $V_0$. We conclude
$p(Y) = W_0\otimes_k K$ for some irreducible subvariety $W_0\subset
V_0$.
We conclude the proposition from (\ref{eq:Yisaproduct}). 
\end{proof}


\section{The Betti Map}\label{SectionBetti}

In this section we describe the construction of the Betti map. 

Let $S$ be a smooth, irreducible, algebraic curve over $\IC$
and suppose $\pi\colon \cA\rightarrow S$ is an abelian scheme
 of relative dimension $g$. We construct:
 \begin{proposition}\label{PropositionBetti}
Let $\cA$ and $S$ be as above. 
For all  $s\in S(\IC)$ there exists an open neighborhood $\unitdisc$ of $s$ in
$\an{S}$ and a real analytic mapping $b \colon \cA_\unitdisc \rightarrow
\IT^{2g}$, called Betti map,
with the following properties. 
\begin{enumerate}
\item [(i)] For each $s\in \unitdisc$ the restriction
$b|_{\cA^{\anE}_s} \colon \cA^{\anE}_s\rightarrow\IT^{2g}$
is a  group isomorphism.
\item[(ii)] For each $\xi\in\IT^{2g}$ 
the preimage $b^{-1}(\xi)$ is a complex analytic subset of $\cA_{\unitdisc}^{\anE}$.
\item[(iii)] The product $(b, \pi|_{\cA_{\unitdisc}}) \colon \cA_\unitdisc \rightarrow \IT^{2g}\times \unitdisc$ is
  real bianalytic.
\end{enumerate}
\end{proposition}

\begin{remark}\label{RemarkEssUniquenessOfBettiMap}
We remark that $b$ from the proposition above is not unique as we can compose it
with a continuous group endomorphism of $\IT^{2g}$. However, if $b,b' \colon \cA_{\unitdisc}\rightarrow\IT^{2g}$  both satisfy the conclusion of the proposition and if $\unitdisc$ is path-connected, then using homotopy and (iii) we find $b' = \alpha \circ b$ for some $\alpha \in \gl{2g}{\IZ}$.

\end{remark}

Before giving the concrete construction, let us explain the idea. Assume $S = \mathbb{A}_g$ is the moduli space of principally polarized abelian varieties with level-$3$-structure, and $\cA = \mathfrak{A}_g$ is the universal abelian variety. The universal covering $\mathfrak{H}_g^+ \rightarrow \mathbb{A}_g$, where $\mathfrak{H}_g^+$ is the Siegel upper half space, gives a polarized family of abelian varieties $\cA_{\mathfrak{H}_g^+} \rightarrow \mathfrak{H}_g^+$
\[
\xymatrix{
\cA_{\mathfrak{H}_g^+} := \mathfrak{A}_g \times_{\mathbb{A}_g} \mathfrak{H}_g^+ \ar[r] \ar[d] & \mathfrak{A}_g \ar[d] \\
\mathfrak{H}_g^+ \ar[r] & \mathbb{A}_g
}.
\]
For the universal covering $u \colon \IC^g \times \mathfrak{H}_g^+ \rightarrow \cA_{\mathfrak{H}_g^+}$ and for each $\tau \in \mathfrak{H}_g^+$, the kernel of $u|_{\IC^g \times \{\tau\}}$ is $\IZ^g + \tau \IZ^g$. Thus the map $\IC^g \times \mathfrak{H}_g^+ \rightarrow \IR^g \times \IR^g \times \mathfrak{H}_g^+ \rightarrow \IR^{2g}$, where the first map is the inverse of $(a,b,\tau) \mapsto (a+\tau b, \tau)$ and the second map is the natural projection, descends to a map $\cA_{\mathfrak{H}_g^+} \rightarrow \IT^{2g}$. Now for each $s \in S(\IC) = \mathbb{A}_g(\IC)$, there exists an open neighborhood $\unitdisc$ of $s$ in $\mathbb{A}_g^{\anE}$ such that $\cA_{\unitdisc} = (\mathfrak{A}_g)|_{\unitdisc}$ can be identified with $\cA_{\mathfrak{H}_g^+}|_{\unitdisc'}$ for some open subset of $\mathfrak{H}_g^+$. The composite $b \colon \cA_{\unitdisc} \cong \cA_{\mathfrak{H}_g^+}|_{\unitdisc'} \rightarrow \IT^{2g}$ is clearly real analytic and satisfies the three properties listed in Proposition~\ref{PropositionBetti}. Thus $b$ is the desired Betti map in this case. Note that for a fixed (small enough) $\unitdisc$, there are infinitely choices of $\unitdisc'$; but for $\unitdisc$ small enough, if $\unitdisc'_1$ and $\unitdisc'_2$ are two such choices, then $\unitdisc'_2 = \alpha \cdot \unitdisc'_1$ for some $\alpha \in \mathrm{Sp}_{2g}(\IZ)$.

Now let us give the concrete construction. Let $s_0\in \an{S}$. By Ehresmann's Theorem \cite[Theorem 9.3]{voisin:hodge1}, there is an open neigborhood $\unitdisc$ of $s_0$ in $\an{S}$ such that $\cA_{\unitdisc}=\pi^{-1}(\unitdisc)$ and $\cA_{s_0}\times\unitdisc$  are diffeomorphic as families over $\unitdisc$. 
The map $f$ in
\begin{equation}
\label{eq:trivialization}
  \begin{tikzcd}
\cA_{s_0} \times\unitdisc \arrow{rd} \arrow{r}{f} &   \cA_\unitdisc  \arrow{d}  \\
& \unitdisc
  \end{tikzcd}
\end{equation}
is  a diffeomorphism, the diagonal arrow is the natural projection,
and the vertical arrow is the restriction of the structural morphism. 
After translating 
we may assume that $(0,s)$ maps to the unit element in $\cA_s$ for all
$s\in\unitdisc$. We may assume that $\unitdisc$ is simply connected. 
Fiberwise we obtain a diffeomorphism $f_s : \cA_{s_0}\rightarrow
\cA_s$. 

As $\an{\cA}$ is a complex analytic space we may assume that the fibers of $f^{-1}$ in  \eqref{eq:trivialization} are complex analytic, see \cite[Proposition~9.5]{voisin:hodge1}.


We fix a basis $\gamma_1,\ldots,\gamma_{2g}$ of the $\IZ$-module $H_1(\cA^{\anE}_{s_0},\IZ)$. Each $\gamma_i$ is represented by a loop $\widetilde{\gamma_i}:[0,1]\rightarrow \cA_{s_0}^{\anE}$ based at the origin of $\cA_{s_0}^{\anE}$.


For all $s\in\unitdisc$ we
 have a map $H^1(\cA^{\anE}_{s},\IR)\rightarrow
H^1(\cA^{\anE}_{s_0},\IR)$ resp.
$H^1(\cA^{\anE}_{s},\IC)\rightarrow
H^1(\cA^{\anE}_{s_0},\IC)$ induced by $f_s$,
it is an isomorphism of  $\IR$- resp. $\IC$-vector spaces.
We denote the latter by $f_s^*$ and note that 
$f_s^*(\overline v) = \overline {f_s^*(v)}$ where
complex conjugation $\overline{\cdot}$ is induced by the real structure.

The Hodge decomposition yields 
\begin{equation*}
 H^1(\cA^{\anE}_{s},\IC)
=H^{0}(\cA^{\anE}_{s},\Omega^1)
\oplus \overline{H^{0}(\cA^{\anE}_{s},\Omega^1)}
\end{equation*}
where 
$H^{0}(\cA^{\anE}_{s},\Omega^1)$ is
the $g$-dimensional vector space of global holomorphic $1$-forms on
$\cA^{\anE}_s$. 
As $s$ varies over $\unitdisc$  we obtain a
collection
\begin{equation*}
  f_s^*  H^{0}(\cA^{\anE}_{s},\Omega^1)
\end{equation*}
 of subspaces of
$H^{1}(\cA^{\anE}_{s_0},\IC)$.
As $f_s^*$ commutes with complex conjugation we have
\begin{equation*}
H^1(\cA_{s_0}^{\anE},\IC) = f_s^* 
H^{0}(\cA^{\anE}_{s},\Omega^1) \oplus \overline{f_s^* 
(H^{0}(\cA^{\anE}_{s},\Omega^1))}.
\end{equation*}

For $s\in\unitdisc$ the image $f_s^*H^{0}(\cA^{\anE}_{s},\Omega^1)$ corresponds 
 to a point in the Grassmannian variety of
$g$-dimensional subspaces of $H^{1}(\cA^{\anE}_{s_0},\IC)$.
As a particular case of Griffith's Theorem, this association is a holomorphic function.
We  draw the following conclusion from  Griffith's result. 

Fix a basis
$\omega_1^0,\ldots,\omega^0_g$ of $H^{0}(\cA^{\anE}_{s_0},\Omega^1)$; 
then
$\omega_1^0,\ldots,\omega^0_g,\overline{\omega_1^0},\ldots,\overline{\omega_g^0}$
is a basis $H^1(\cA^{\anE}_{s_0},\IC)$. 
There exist holomorphic functions 
\begin{equation*}
a_{ij}:\unitdisc\rightarrow\IC\quad\text{and}\quad
b_{ij}:\unitdisc\rightarrow\IC \quad (1\le i,j\le g)
\end{equation*}
such that
\begin{equation*}
  f_s^* \omega_i(s) = \sum_{j=1}^g \left(
a_{ij}(s) \omega_j^0 + b_{ij}(s) \overline{\omega_j^0}\right)
\end{equation*}
for all $i\in \{1,\ldots,g\}$ and all $s\in\unitdisc$ where
$\omega_1(s),\ldots,\omega_g(s)$ is a basis
of $H^{0}(\cA^{\anE}_{s},\Omega^1)$
with $\omega_i(s_0) = \omega_i^0$ for all $i$. 

For $s\in\unitdisc$ we define the period matrix
\begin{equation*}
  \Omega(s) = \left( \int_{f_{s*}\widetilde{\gamma_j}}\omega_i(s)
  \right)_{\substack{1\le i\le g \\ 1\le j\le 2g}}\in \mat{g,2g}{\IC}
\end{equation*}
for all $s\in \unitdisc$; the integral is taken over the loop in
$\cA_{s_0}^{\anE}$ fixed above. 
Note that
\begin{equation*}
  \int_{f_{s*}\widetilde{\gamma_j}}\omega_i(s) 
 = \int_{\widetilde{\gamma_j}} f_s^* {\omega_i(s)}
= \sum_{j=1}^g \left(a_{ij}(s) \int_{\widetilde{\gamma}_j} \omega_j^0
  + b_{ij}(s) \int_{\widetilde{\gamma}_j} \overline{\omega_j^0}\right)
\end{equation*}
by a change of variables. 
So $\Omega(s)$ is holomorphic in $s$. 
In this notation and with 
$A(s) = (a_{ij}(s))\in \mat{g}{\IC}$
and 
 $B(s) = (b_{ij}(s))\in \mat{g}{\IC}$ we can abbreviate the above by
\begin{equation}
\label{eq:periodtransformation}
\left(
  \begin{array}{c}
    \Omega(s) \smallskip \\  \overline{\Omega(s)}
  \end{array}\right)
 =
\left(
\begin{array}{cc}
A(s) & B(s) \smallskip \\
\overline{B(s)} & \overline{A(s)}
  \end{array}\right)
\left(
\begin{array}{c}
    \Omega(0) \smallskip \\  \overline{\Omega(0)}
  \end{array}\right)
\end{equation}
here $\Omega(0) = \Omega(s_0)$. 
So the first matrix on the right of (\ref{eq:periodtransformation}) is invertible.

Let $P\in \cA^{\anE}$ where $s=\pi(P)\in\unitdisc$
and suppose $\gamma_P$ is a path in $\cA^{\anE}_s$ connecting $0$ and
$P$. 
Let $Q\in \cA^{\anE}_{s_0}$ with $f_s(Q)=P$
and $\gamma_Q$ the path in $\cA^{\anE}_{s_0}$ such that 
$f_{s*}\gamma_Q = \gamma_P$. 
We define 
\begin{equation}
  \label{eq:defcalLP}
  \mathcal{L}(P) = 
\left(
  \begin{array}{c}
    \int_{\gamma_P} \omega_1(s) \\
\vdots\\
\int_{\gamma_P} \omega_g(s) \\
  \end{array}\right)
=
\left(
  \begin{array}{c}
    \int_{\gamma_Q} f_s^*\omega_1(s) \\
\vdots\\
\int_{\gamma_Q} f_s^*\omega_g(s) \\
  \end{array}\right)
= (
  A(s)  B(s) 
)
\left(
\begin{array}{c}
  \mathcal{L}^*(Q)\smallskip \\
  \overline{\mathcal{L}^*(Q)}
\end{array}\right)
\end{equation}
where
\begin{equation*} 
  \mathcal{L}^*(Q) = 
\left(
  \begin{array}{c}
    \int_{\gamma_Q} \omega_1^0(s) \\
\vdots\\
\int_{\gamma_Q}\omega_g^0 (s) \\
  \end{array}\right).
\end{equation*}

Replacing $\gamma_P$ by another path connecting $0$ and $P$ in
$\cA_s^{\anE}$ will translate the value of $\mathcal{L}(P)$ by a
period in $\Omega(s)\IZ^{2g}$. By passing to the quotient we obtain
the
Albanese map $\cA_s^{\anE}\rightarrow\IC^g / \Omega(s)\IZ^{2g}$. It is
a group isomomorphism.

We set further
\begin{equation*}
  \tilde b(P) =  
\left(
  \begin{array}{c}
    \Omega(s) \smallskip \\  \overline{\Omega(s)}
  \end{array}\right)^{-1}
\left(
  \begin{array}{c}
    \mathcal{L}(P) \smallskip  \\    \overline{\mathcal{L}(P)}
  \end{array}\right).
\end{equation*}
and observe $\tilde b(P)\in\IR^{2g}$ as these are the coordinates of
$\mathcal{L}(P)$ in terms of the period lattice basis $\Omega(s)$. 

By replacing $\gamma_P$ by another path connecting $0$ and $P$ we find
that $\tilde b(P)$ is translated by a vector in $\IZ^{2g}$. 
Therefore, $\tilde b$ induces a real analytic map $b \colon \cA_\unitdisc \rightarrow
\IT^{2g}$, where $\IT$ is the circle group which we identify with $\IR/\IZ$. We will prove that $b$ satisfies the three properties listed in Proposition~\ref{PropositionBetti}.

On a given fiber, \textit{i.e.} for fixed $s$, the map $b$ restricts to a group
isomorphism $\cA_{s}\rightarrow\IT^{2g}$ as we have seen above. So part (i) of Proposition~\ref{PropositionBetti} holds.

Let us investigate such a fiber. For this 
we  recall (\ref{eq:defcalLP}).
By the period transformation formula (\ref{eq:periodtransformation}) we see
\begin{equation*}
  \tilde b(P) = \left(
  \begin{array}{c}
    \Omega(0) \smallskip \\ \overline{\Omega(0)}
  \end{array}\right)^{-1} \left(
  \begin{array}{c}
    \mathcal{L}^*(Q) \smallskip  \\    \overline{\mathcal{L}^*(Q)}
  \end{array}\right).
\end{equation*}

Fixing the value of $\tilde b$ amounts to fixing the value of $\mathcal{L}^*(Q)$.
As $\mathcal{L}^*$ induces the Albanese map on $\cA_s^{\anE}$, fixing
$b$ amounts to 
fixing $Q$. 
Recall that $Q$ maps to $P$ under the trivialization (\ref{eq:trivialization}). 
Therefore, a fiber of $b$ equals a fiber
of the trivialization. As these fibers are complex analytic we obtain
part (ii) of Proposition~\ref{PropositionBetti}.

Finally, the association
\begin{equation*}
  (\xi+\IZ^{2g},s)\mapsto 
\left(\left(
  \begin{array}{c}
    \Omega(0) \smallskip \\ \overline{\Omega(0)}
  \end{array}\right)(\xi+\IZ^{2g}),s\right) \mapsto (Q,s) \mapsto f_s(Q) \in \cA_\unitdisc
\end{equation*}
induces  the inverse of the product
$\cA_\unitdisc \rightarrow\IT^{2g}\times \unitdisc$. This is part (iii) of Proposition~\ref{PropositionBetti}. 


\section{Degenerate Subvarieties}\label{SectionDegenerate}

Let $S$ be a smooth irreducible algebraic curve over
$\mathbb{C}$ and let $\pi \colon \mathcal{A} \rightarrow S$ be an
abelian scheme of relative dimension $g\ge 1$. We define and
characterize the degenerate subvarieties of $\mathcal{A}$ in this
section. Let $Y$ be an irreducible  closed subvariety of $\mathcal{A}$
that dominates $S$.

Let $s_0 \in S(\IC)$ and let $\unitdisc \subset S^{\anE}$ be an open
neighborhood of $s_0$ in $S^{\anE}$ with the Betti map
$b \colon \cA_{\unitdisc} =\pi^{-1}(\unitdisc) \rightarrow \IT^{2g}$ as in
Proposition~\ref{PropositionBetti} with $\IT\subset\IC$ the circle group. We say that a point $P\in \sman{Y}\cap\cA_\unitdisc$ 
is \emph{degenerate}\index{Degenerate Point} for $Y$ if it is not isolated in
$b|_{\sman{Y}\cap\cA_\unitdisc}^{-1}(b(P))$. 
We say that $Y$ is \emph{degenerate}\index{Degenerate Subvariety} if 
there is a non-empty and open subset of 
$\sman{Y}\cap\cA_\unitdisc$
consisting  of points that are degenerate for $Y$. 

For technical purposes  our notation of degeneracy formally depends on
the choice of $\unitdisc$. But this dependency is harmless as we will see.

Recall that
generically special
subvarieties of $\cA$ were introduced in 
 Definition~\ref{DefinitionSpecialSubvariety}.
A generically special subvariety is degenerate.
In this section we
prove the converse.

\begin{theorem}\label{PropositionDegenerate}
An irreducible closed subvariety of $\mathcal{A}$
that is degenerate
is a generically special subvariety of $\mathcal{A}$.
\end{theorem}

This proposition, which has  a
definite Ax--Schanuel flavor,
is proved using
 a variant of the Pila--Wilkie Counting
Theorem for definable sets in an o-minimal structure. 
 Abundantly many rational points arise
from the exponential growth of a  
certain monodromy group.

\subsection{Invariant Subsets of the Torus}
We write $|\cdot|_2$ for the $\ell^2$-norm on $\IR^{n}$.

For $n\in\IN$ we consider 
 the real $n$-dimensional torus  $\IT^n$ equipped with the
standard topology. 
We will use the continuous left-action of $\gl{n}{\IZ}$ on $\IT^n$
and use the additive notation for $\IT^n$. 
Suppose $X$ is a closed subset of $\IT^n$
such that
\begin{equation*}
  \gamma(X)= X
\end{equation*}
for all $\gamma$ in a subgroup $\Gamma$ of $\mathrm{GL}_n(\IZ)$. 
What can we say about $X$?

To rule out subgroups that are too small
we ask that $\Gamma$  contains a (non-abelian) free subgroup on $2$ generators. 
Moreover, we will assume that  $X$ is sufficiently ``tame'' as a set. 

To formulate the last property precisely, let
 $\exp \colon \IR^n \rightarrow \IT^n$ denote the exponential map 
$(t_1,\ldots,t_n)\mapsto (e^{2\pi i t_1},\ldots,e^{2\pi i t_n})$. 
Let $X\subset \IT^n$ be a subset and 
\begin{equation*}
  \cX = \exp|_{[0,1]^{n}}^{-1}(X).
\end{equation*}
We will work in a fixed o-minimal structure and
 call $X$ definable if $\cX$
is a definable subset of $\IR^n$ in the given o-minimal structure. We
 refer to van den Dries' book \cite{D:oMin} for the theory of o-minimal
 structures. We will  work with $\IRan$, the o-minimal structure
 generated by restricting real analytic functions on $\IR^n$ to $[-1,1]^n$.

We say that  $X\subset\IT^n$ is  of \emph{Ax-type} if it satisfies the
following property. For any continuous, semi-algebraic map
$y \colon [0,1]\rightarrow \cX$ 
that is real-analytic on $(0,1)$, there is a closed subgroup
$G\subset\IT^n$ such that $\exp\circ y([0,1]) \subset
 y(0)+G \subset X$. 

The main example comes from a $g$-dimensional
 abelian variety $A$  defined over $\IC$. Indeed, then there is a real
 bianalytic map $\an{A}\rightarrow\IT^{2g}$. Moreover, the image of
 $X(\IC)$
is definable  and of Ax-type for 
 any Zariski closed subset of $A$; for the latter claim we refer to
Ax's Theorem \cite{AxSchanuel}.

\begin{lemma}
\label{lem:invariant}
Let $X\subset \IT^n$ be a closed definable set  of Ax-type.
Let $\Gamma$ be a free subgroup of $\gl{n}{\IZ}$   
on $2$ generators
such that $\gamma(X)= X$ for all $\gamma\in\Gamma$.
Then one of the following properties holds true:
\begin{enumerate}
\item [(1)] 
The set $X$ is contained in a finite union of closed and proper
 subgroups of  $ \IT^n$.
\item[(2)] 
There are a non-empty, open subset $U$ of $X$ and
 a closed, connected, infinite subgroup $G\subset\IT^n$ with 
$U+G \subset X$. 
\end{enumerate}
\end{lemma}
\begin{proof}
By assumption, $\Gamma$ is generated by elements $\gamma_1,\gamma_2$ 
 that do not satisfy any non-trivial relation.
Any element $\gamma\in \Gamma$ is uniquely
represented by a reduced word in $\gamma_1^{\pm 1},\gamma_2^{\pm 1}$
whose length is 
 $l(\gamma)$.
For all real $t\ge 1$ we have
\begin{equation*}
  \#\{\gamma\in \Gamma : l(\gamma)\le t \} \ge 2^t.
\end{equation*}

We define
$c_1 = \max\{2,|\gamma_1|_2,|\gamma_1|_2\}\ge 2$
and observe  $|\gamma|_2\le c_1^{l(\gamma)}$ for all 
$\gamma\in \Gamma$. 
The height $H(b)$ of any integral vector $b=(b_1,\ldots,b_m)\in\IZ^m$ 
is $\max\{1,|b_1|,\ldots,|b_m|\}$.
So
\begin{equation*}
  H(\gamma) \le c_1^{l(\gamma)}.
\end{equation*}

Let $T\ge c_1$ and let   $t = (\log T)/(\log c_1)\ge 1$.
There are at least
$2^{t} = T^{(\log 2) / \log c_1}$ elements $\gamma\in \Gamma$ 
with $H(\gamma) \le T$.

Let $x\in \cX=\exp|_{[0,1]^n}^{-1}(X)$. For all $\gamma\in \Gamma$ there is
$a=a_\gamma\in\IZ^n$ such that  $y_\gamma=\gamma x-a_\gamma\in \cX$.
Then $(x,\gamma,a_\gamma,y_\gamma)$ lies in the  definable set
\begin{equation*}
  \cZ = \left\{(x,\gamma,a,y) \in \cX \times\mathrm{GL}_n(\IR)\times
  \IR^n\times \cX : 
\gamma x - a = y\right\}.
\end{equation*}
We view it as a family of definable sets parametrized by $x\in \cX$
with fibers $\cX_x \subset\IR^{n^2+n+n}$. 
Moreover,
\begin{alignat}1
\label{eq:Habound}
H(a_\gamma) &\le 
  \max\{1,|a_\gamma|_2\} = \max\{1,|\gamma x -
y_\gamma|_2\}\\ &\le 
\max\{1,|\gamma|_2 |x|_2 + |y_\gamma|_2\}
\nonumber
 \le \sqrt{n}
  (|\gamma|_2 + 1)\le 2n^2 H(\gamma).
\end{alignat}


Let $c_2$ be the constant from the semi-rational variant of the
Pila--Wilkie Theorem \cite[Corollary~7.2]{HabeggerPilaENS} applied to the family $\cZ$ and $\epsilon =
(\log 2) / (2\log c_1)$. 
Here the coordinates assigned to $(\gamma,a)$ are treated as rational
and the coordinates assigned to $y$ are not. 
We fix $T$ large enough in terms of $c_1$ and
$c_2$, more precisely we will assume that $T\ge c_1$ and
\begin{equation}
\label{eq:manypoints}
  T^{(\log 2) / \log c_1} > c_2  (2n^2T)^{(\log 2) / (2\log c_1)}.
\end{equation}

We keep $x$ fixed and
vary $\gamma$. Let us first see how to reduce to the case  that many different $y_\gamma$ must arise
this way if $H(\gamma)\le T$.

Indeed, suppose $\gamma'\in\Gamma$ satisfies $H(\gamma')\le T$. Then $y'_{\gamma'} = \gamma' x - a'_{\gamma'}\in\cX$ for some
$a'\in\IZ$. 
If $y_\gamma=y'_{\gamma'}$, then $\gamma x - a_\gamma = \gamma' x -
a'_{\gamma'}$, so $\gamma x-\gamma' x \in\IZ^{n}$. Then $\exp(x)$ lies
in the closed subgroup of $\IT^n$ defined  by the kernel of
$\gamma^{-1}\gamma'-1\not=0$, \textit{i.e.} the largest subgroup of $\IT^n$ stabilized by $\gamma^{-1}\gamma'$. 
 So it lies in a finite union 
$G_1\cup\cdots\cup G_N$ of closed proper subgroups of $\IT^n$, each
defined as the subgroup stabilized by some $\gamma^{-1}\gamma'$ as above. 
Here $N$ is bounded only in terms of $T$ and thus only in terms of
$c_1,c_2,$ and $n$. It is independent of $x$. 

If $X\subset G_1\cup\cdots\cup G_N$, then we are in case (1). 

Otherwise $V = X\setminus (G_1\cup\cdots \cup G_N)$
lies open in $X$ and is non-empty. 

Now suppose  $x\in \cX$ with $\exp(x)\in V$
and $\gamma\in G$ with $H(\gamma)\le T$. 
Recall that 
 $y_\gamma = \gamma x - a_\gamma\in \cX$. 
By our choice of $V$ and the arguments above  
 the number of $y_\gamma$ that arise is the number of elements in
 $\Gamma$ of height at most $T$. 
This number is at least $T^{(\log 2) / \log c_1}$. 
Note that the height of $(\gamma,a_\gamma)$ equals $\max\{H(\gamma),H(a_\gamma)\}$
and this is at most $2n^2 T$ by (\ref{eq:Habound}).

By (\ref{eq:manypoints}) we have enough $y_\gamma$ to apply the
counting result \cite[Corollary~7.2]{HabeggerPilaENS}. 
We thus obtain continuous, definable maps
$\gamma\colon [0,1]\rightarrow \mathrm{GL}_n(\IR)$, 
$a\colon [0,1]\rightarrow \IR^n$, and
$y\colon [0,1]\rightarrow \cX$ 
such that $\gamma$ and $a$ are semi-algebraic, $y$ is
non-constant, and
\begin{equation*}
  \gamma(s)x - a(s) = y(s) 
\end{equation*}
for all $s\in [0,1]$. So $s\mapsto y(s)$ is semi-algebraic too
and $\exp\circ y([0,1]) \subset X$.
After rescaling $[0,1]$ we may
assume that $y$ is
real-analytic on $(0,1)$.
By looking at the proof of \cite[Corollary~7.2(iii)]{HabeggerPilaENS} we may  arrange 
 $\gamma(0)\in \Gamma$ and $a(0)\in\IZ^n$. 
Recall that $X$ is of Ax-type. So there is a closed subgroup $G'_x
\subset \IT^n$ with  
$\exp\circ y([0,1]) \subset  \exp(y(0)) +G'_x \subset X$.
We may assume that  $G'_x$ is connected.
Observe that $G'_x$
is infinite as $\exp\circ y$ is continuous and non-constant.
We find $\exp(x) + G_x \subset \gamma(0)^{-1}(X)=X$ where $G_x = \gamma(0)^{-1} G'_x$. 

We have proved that for any $x\in \cX$ with $\exp(x)\in V$ we have
\begin{equation*}
  \exp(x) + G_x \subset X
\end{equation*}
for some connected, closed, infinite subgroup $G_x\subset \IT^n$. 

For any connected closed subgroup $G\subset\IT^n$ we define
\begin{equation*}
  E(G) = \{z\in V : z+ G \subset X\}  = V\cap \bigcap_{g\in G} (X-g). 
\end{equation*}
Then $E(G)$ is closed in $V$.
Our conclusion from above can be restated as
\begin{equation*}
  V = \bigcup_{x\in \exp|_{\cX}^{-1}(V)}  E(G_x). 
\end{equation*}

By Kronecker's Theorem $\IT^n$ has  countably  many  closed subgroups.
So this union countains at most countably many different members. 
Now $V$, being non-empty, Hausdorff, and locally compact 
satisfies the hypothesis of 
 Baire's Theorem. 
Hence there exists an  connected, closed, infinite subgroup
$G\subset\IT^n$ 
such that $V\setminus E(G)$ 
is not dense in $V$. 
So $E(G)$ contains a non-empty and open subset of $X$, as claimed in
(2). 
\end{proof}

Now suppose that $A$ is an 
abelian variety of dimension $g\ge 1$ defined over $\IC$.




We attach to $A$ the associated 
  complex manifold  $\an{A}$ whose underlying set of points is $A(\IC)$. There  is a real bi-analytic    
map   $b \colon \an{A}\rightarrow\IT^{2g}$ which is a group isomorphism,
 we will not need to vary $A$ in a family here as in 
 Proposition~\ref{PropositionBetti}.


Suppose a group $\Gamma$ acts faithfully and continuously
 on $\an{A}$; we do not 
ask for elements of $\Gamma$ to act by  holomorphic maps.
Any continuous group automorphism of $\IT^{2g}$ can be identified with
an element of $\gl{2g}{\IZ}$. 
So using $b$ we may consider $\Gamma$ as a subgroup of
$\gl{2g}{\IZ}$.

We say that the action of $\Gamma$ is of \emph{monodromy-type} if $\gamma(B(\IC)) = B(\IC)$ for all $\gamma\in\Gamma$
and all abelian subvarieties $B\subset A$. 


Later we will study the action of the fundamental group of an
abelian scheme on a fixed fiber in sufficiently general position. 
This action will leave the abelian subvarieties of the said fiber
invariant and is thus of monodromy-type. 



\begin{proposition}
\label{prop:invariance}
  Let $A,g,b,$ and $\Gamma\subset\gl{2g}{\IZ}$ be above, so in particular 
  $\Gamma$ acts continuously on $\an{A}$ and is of monodromy-type.
  We assume in addition that $\Gamma$ 
 contains a free subgroup on $2$ generators and
 that there are no $\Gamma$-invariant elements in $\IZ^{2g}\setminus\{0\}$.
Let $Z$ be an
  irreducible closed subvariety
  of $A$ with $\gamma(Z(\IC))= Z(\IC)$ for all $\gamma \in
\Gamma$. 
Then one of the following properties holds:
\begin{enumerate}
\item [(1)] The subvariety $Z$ is contained in 
a proper torsion coset in $A$. 
\item[(2)] There exists an abelian subvariety $B\subset A$ with $\dim
  B\ge 1$ and $Z+B = Z$. 
\end{enumerate}
\end{proposition}
\begin{proof}
  We write $X$ for the image of $Z(\IC)$ under the real analytic
  isomorphism
$b\colon \an{A}\rightarrow\IT^{2g}$. 
Then 
$X$ is closed and definable  in the sense as introduced before Lemma \ref{lem:invariant}. 
By Ax's Theorem \cite{AxSchanuel}, the set $X$ is 
of Ax-type. We apply Lemma \ref{lem:invariant} to a free subgroup of
  $\Gamma$ on $2$ generators.

If we are in case (1) of Lemma \ref{lem:invariant}, then $X$ is
contained in a finite union  of proper closed subgroups
$G_1,\ldots,G_N\subsetneq \IT^{2g}$. By the Baire Category Theorem we may assume that $X\cap G_1$ has non-empty interior in
$X$. 

The analytification $\an{Z}$ is an irreducible
complex analytic space and $\sman{Z}$ is
arc-wise connected by \cite[Theorems~9.1.2 and 9.3.2]{CAS}. 
Moreover, $\sman{Z}$ is an open and dense subset of $\an{Z}$. 

Let $P,Q\in \sman{Z}$ and suppose $b(P)$ lies in the interior of
$X\cap G_1$.
We can connect $P$ and $Q$ via an arc $[0,1]\rightarrow \sman{Z}$
whose restriction to $(0,1)$ is piece-wise real analytic on finitely
many pieces. A neighborhood of 
$b(P)$ in $X$ lies in $G_1$ 
and $G_1$ is defined globally by  relations in integer
coefficients. 
By analytic continuation
we find that $b(Q)\in G_1$. In
particular,
$b(\sman{Z})\subset G_1$ and thus $b(\an{Z})\subset G_1$.
So $\an{Z}$ is contained in the proper
subgroup $b^{-1}(G_1)\subsetneq \an{A}$. 

 The sum of
sufficiently many copies of $Z - Z$ is an abelian subvariety $B$ of
$A$. 
 We have $B\not=A$ because $B(\IC)$ lies in $b^{-1}(G_1)$. 
So $Z\subset P+B$ for some $P\in A(\IC)$.
Moreover, any coset in $A$
containing $Z$ must contain $P+B$. 

Let $B'$ be the complementary abelian subvariety of $B$ in $A$ with
respect to a fixed polarization, see \cite[$\mathsection$5.3]{CAV}.
So $B+B'=A$ and $B\cap B'$ is finite. By the former property 
 we may assume $P\in B'(\IC)$. 

By hypothesis we have 
$Z(\IC) = \gamma(Z(\IC)) \subset \gamma(P)+\gamma(B(\IC)) = \gamma(P)+ B(\IC)$
for all $\gamma\in \Gamma$. Thus $\gamma(P)-P \in B(\IC)$ for all
$\gamma\in\Gamma$. 
As $B'$ is invariant under $\gamma$  we find 
 $\gamma(P)-P\in (B\cap B')(\IC)$. 
So $\gamma(Q)-Q=0$ for all $\gamma\in\Gamma$ where $Q=[\#B\cap B'](P)$. 

The point $b(Q)\in\IT^{2g}$ is the image of some $t\in\IR^{2g}$ under
the canonical map $\IR^{2g}\rightarrow\IT^{2g}$. 
Our action of $\Gamma$ on $\an{A}$ was defined using $b$ and $\Gamma$
acts on $\IT^{2g}$ via a matrix in $\mat{2g}{\IZ}$. 
We find 
that $\gamma(t)-t\in\IZ^{2g}$ for all $\gamma\in\Gamma$
with the standard action of $\gl{2g}{\IZ}$ on $\IR^{2g}$. 

Thus $t\in\IR^{2g}$ is the solution of a system of inhomogenous linear equations,
parametrized by $\Gamma$,
with integral coefficients and integral solution vector.
The corresponding homogeneous equation has only the trivial solution
as
there are no non-trivial $\Gamma$-invariant vectors in $\IZ^{2g}$. 
So $t$ was the unique solution and we conclude $t\in\IQ^{2g}$.
Therefore $Q$ and thus $P$ have finite order. 
So $P+B$ is a torsion coset in $A$ and we are
in case (1) of the current proposition.

Now suppose we are in case (2) of Lemma \ref{lem:invariant} and $U$
and $G$ are as given in therein. Then 
 $\bigcap_{P\in b^{-1}(G)} (Z-P)$ 
is Zariski closed in $Z$ since $0\in G$. By Lemma \ref{lem:invariant} 
 its complex points contain 
$b^{-1}(U)$, which is Zariski dense in $Z$. 
So $Z-P = Z$ for all $P\in b^{-1}(G)$. 
This equality continues to hold for $\IC$-points in the Zariski closure $B$
of $b^{-1}(G)$ in $A$. 
As $G$ is a connected subgroup of $\an{A}$ we find that $B$ is an
abelian subvariety of $A$. Moreover, $\dim B\ge 1$ since $G$ is
infinite. So we are in case (2) of the proposition. 
\end{proof}

\subsection{Degeneracy and Global
Information}

\newcommand{\bp}{s} 

Let $S$ be an irreducible and smooth curve over $\IC$ and
let $\cA$ be an abelian scheme over $S$
of relative dimension $g\ge 1$. 

Recall that Betti maps were  introduced in \S\ref{SectionBetti}.
Around  each point of $ \an{S}$ we fix an open neighborhood 
 in $\an{S}$ and a Betti map as in Proposition \ref{PropositionBetti}. This yields
 an open cover of $\an{S}$ which we now refine for our application
 later on. After shrinking each member
 we may assume that each member
  is bounded and diffeomorphic to an open subset of $\IR^2$. 
 As $\an{S}$ is paracompact we may refine this cover to obtain
  an open cover  of $\an{S}$
that is locally finite. Each member of this cover is relatively
compact. We may refine the
cover again and assume that a finite intersection of members is
empty or contractible, see Weil's treatment \cite[\S 1]{Weil:deRham}. A non-empty open subset of $\an{S}$ is 
naturally a Riemann surface; if it is contractible then
it is homeomorphic to the open unit disc. Therefore, a finite
intersection of members of our cover is empty or homeomorphic to the
open unit disc. 

Let $\bp \in \an{S}$ be a base point. We describe the monodromy representation of $\pi_1(\an{S},s)$ using the Betti map.

Let $\gamma \colon [0,1]\rightarrow \an{S}$ be  a loop around $\bp$.  We can find a Betti map in a neighborhood around each point of
$\gamma([0,1])$. As this image is compact we find
$0=a_0<a_1<\cdots <a_n = 1$  such that
 $\gamma([a_{i-1},a_i])\subset \unitdisc_i$ 
where $\unitdisc_i$ is a member of the cover above and $b_i$
is its associated Betti map.

We can glue the Betti maps as follows. For each $i\in \{1,\ldots,n-1\}$
we have $s_i= \gamma(a_i) \in \unitdisc_i\cap\unitdisc_{i+1}$. 
So ${b_{i}}|_{\cA_{s_i}^{\anE}}\circ ({b_{i+1}}|_{\cA_{s_i}^{\anE}})^{-1}$ 
is a continuous group isomorphism
$M \colon \IT^{2g}\rightarrow \IT^{2g}$,
thus represented by a matrix in $\gl{2g}{\IZ}$. 
On replacing  $b_{i+1}$ by
$M\circ b_{i+1}$ we may arrange that 
$b_i$ and $b_{i+1}$ coincide on $\cA_{s_i}^{\anE}$.  

Now $\gamma(0)=\gamma(1)=\bp$ and both $b_1$ and $b_n$  define
homeomorphisms $\cA_{\bp}^{\anE}\rightarrow\IT^{2g}$. 
By composing we obtain a homeomorphism $\cA_{\bp}^{\anE}\rightarrow\cA_{\bp}^{\anE}$
that is a group isomomorphism. 
This homeomorphism induces an automorphism of 
the $\IZ$-module $H^1(\cA_{\bp}^{\anE},\IZ)$ that depends on the loop
$\gamma$.
Another  loop that is
homotopic to $\gamma$ relative $\{0,1\}$ will lead to the same
automorphism of  $H^1(\cA_{\bp}^{\anE},\IZ)$. The
induced mapping 
 $\pi_1(\an{S},\bp) \rightarrow \mathrm{Aut}(H^1(\cA_{\bp}^{\anE},\IZ))$
is the monodromy representation from
\cite[\S 3.1.2]{voisin:hodge2}. We denote its dual by
\begin{equation}
\label{eq:monodromyrep}
  \rho \colon \pi_1(\an{S},\bp)\rightarrow \mathrm{Aut}(H_1(\cA_{\bp}^{\anE},\IZ)).
\end{equation}

\begin{proposition}
\label{prop:globalinfo}
In the notation above there is a group homomorphism
\begin{equation}\label{eq:monoaction}
  \widetilde\rho=  \widetilde\rho_{\cA} \colon \pi_1(\an{S},\bp) \rightarrow \{ \text{homeomorphisms
} \an{\cA}_{\bp}\rightarrow \an{\cA}_{\bp}\text{ that are group homomorphisms}\}
\end{equation}
 that satisfies 
\begin{equation}
\label{eq:functoriality}
  \widetilde \rho(h)_* = \rho(h)\quad\text{for all}\quad
h\in \pi_1(\an{S},\bp)
\end{equation}
with the following properties.
\begin{enumerate}
\item [(i)]
There exists a
path-connected open neighborhood 
 $\unitdisc \subset \an{S}$ 
of $\bp$ and $b$ a Betti map on $\cA_\unitdisc$ as
in Proposition \ref{PropositionBetti}. 
Let $Y \subset \cA$ be an irreducible closed subvariety such that
 $P\in \an{Y}$ with $\pi(P)=\bp$ is not isolated in the fiber of 
$b|_{\an{Y}\cap \cA_\unitdisc}$. 
Then $\widetilde\rho(h)(P) \in \an{Y}$
for all
$h\in \pi_1(\an{S},\bp)$. Moreover, if $P$ has finite order $N$ in
$\cA_{s}(\IC)$ then $\dim_{P} Y\cap \cA[N]\ge 1$. 
\item[(ii)]
Let $\cB$ be a further abelian scheme over $S$ 
and $\alpha \colon \cA\rightarrow\cB$ a morphism of abelian schemes over
$S$. 
Then 
\begin{equation*}
\widetilde \rho_\cB(h)(\alpha|_{\an{\cA}_{s}}) = 
(\alpha|_{\an{\cA}_{s}})\widetilde \rho_\cA(h)
\end{equation*}
for all $h\in \pi_1(\an{S},s)$ 
\end{enumerate}
\end{proposition}

Although the Betti map $b$ in
 Proposition \ref{PropositionBetti} is not uniquely determined, Remark~\ref{RemarkEssUniquenessOfBettiMap} 
implies that the non-isolation condition in the hypothesis above is
 independent of any choice of $b$. 

Before we come to the proof we will patch together the Betti maps and extract global
information.

Suppose $i\in \{1,\ldots,n-1\}$ and set $\unitdisc=\unitdisc_i\cap \unitdisc_{i+1}\ni\gamma(a_i)$. We consider the two real bi-analytic maps
\begin{equation*}
  b^*_i |_{\cA_\unitdisc}\text{ and }b^*_{i+1} |_{\cA_\unitdisc} \colon \cA_\unitdisc\rightarrow\IT^{2g}\times\unitdisc
\end{equation*}
where the star signifies passing to the product as in
Proposition \ref{PropositionBetti}(iii). By composing we obtain
\begin{equation}
\label{eq:T2gfiberedXi}
  b^*_{i+1} |_{\cA_\unitdisc}\circ (b^*_{i}|_{\cA_\unitdisc})^{-1} \colon \IT^{2g}\times \unitdisc\rightarrow\IT^{2g}\times\unitdisc
\end{equation}
which is, over each fiber of $\unitdisc$, a continuous group isomorphism $\IT^{2g}\rightarrow\IT^{2g}$. By construction it is the identity over $\gamma(a_i)\in\unitdisc$. Each continuous group isomorphism of $\IT^{2g}$ is represented by a matrix in $\gl{2g}{\IZ}$. By homotopy, \eqref{eq:T2gfiberedXi} is the identity above all points in the path component of $\unitdisc$ containing $\gamma(a_i)$. But $\unitdisc$ is path connected by construction, and therefore $b_i |_{\cA_\unitdisc}=b_{i+1} |_{\cA_\unitdisc}$ for all $i\in \{1,\ldots,n-1\}$.

\begin{proof}[Proof of Proposition \ref{prop:globalinfo}]
Let $\bp,Y,$ and $P$ be as in the hypothesis. 
We abbreviate
$Y_{\unitdisc_1} = \an{Y}\cap\cA_{\unitdisc_1}$, it is a complex analytic
space.

We will  transport $P$ in $\an{\cA}$ above along a loop
$\gamma$ in $\an{S}$ based at $\bp$ and keep the  Betti coordinates fixed.
After completing  the loop we will have returned to the fiber
$\cA_{\bp}$. But $P$ will have transformed according to the monodromy
representation \eqref{eq:monodromyrep}. 
The degeneracy condition imposed on $P$ implies that this new point lies
again in $Y$. This is guaranteed by the fact that the Betti fibers are
complex analytic, see (ii) of Proposition~\ref{PropositionBetti} and
our hypothesis $\dim S = 1$. 

Let us check the details.
We set $P_0=P$ and  $\xi = b_1(P_0)$ and define
\begin{equation*}
  Z_1 = b_1^{-1}(\xi).
\end{equation*}
So $Z_1$ is a complex analytic 
 subset of the complex analytic space $\cA_{\unitdisc_1}$ by (ii) of Proposition~\ref{PropositionBetti}. Therefore, $Z_1 \cap Y_{\unitdisc_1}$ is complex analytic in $Y_{\unitdisc_1}$.
As $P_0$ is not isolated in $Z_1\cap Y_{\unitdisc_1}$, we find
  $\dim_{P_0} Z_1 \cap Y_{\unitdisc_1} \ge 1$, see \cite[Chapter~5]{CAS} for the
dimension theory of complex analytic spaces. 

If $P=P_0$ happens to be a point of finite order $N$ in $\cA_{\pi(P)}(\IC)$,
then all points of $Z_1$  have order $N$ in their respective
fibers as $\beta$ is fiberwise a group isomorphism. 
From the degeneracy of $P$  we conclude $\dim_{P} Y\cap\cA[N]\ge 1$ and this yields the
second claim of (i).

The natural projection $Z_1 \mapsto \unitdisc_1$ is holomorphic
and a homeomorphism.
So $\dim_Q Z_1 \le \dim_{\pi(Q)}\unitdisc_1 = 1$ for all $Q\in Z_1$. So we conclude 
  $\dim_{P_0} Z_1 \cap Y_{\unitdisc_1} = \dim_{P_0} Z_1 = 1$ and $\dim Z_1 = 1$.
The singular points of $Z_1$ are isolated in $Z_1$, see \cite[Chapter~6, \S
2.2]{CAS}. Since $Z_1$ is homeomorphic to $\unitdisc_1$  and
the latter is homeomorphic to the open unit disc we conclude that the
smooth locus of $Z_1$ is path connected. Therefore, we can apply the
Identity Lemma \cite[Chapter 9, \S 1.1]{CAS} to conclude that 
$Z_1\cap Y_{\unitdisc_1} = Z_1$, hence
\begin{equation*}
  Z_1\subset Y_{\unitdisc_1}. 
\end{equation*}
In particular, the point $P_1 = b_1^{-1}(\xi,\gamma(a_1))\in Z_1$ also lies in
$Y_{\unitdisc_1}$.

Observe that we used 
the fact that $\an{S}$ is a curve  in a crucial
way. 
Indeed, for higher dimensional $S$ we cannot exclude $\dim Z_1\cap
Y_{\unitdisc_1} < \dim Z_1$ in the paragraph above. This makes applying the Identity Lemma
impossible. 

We have reached $\gamma(a_1)$ and will continue on 
 the circuit along $\gamma$.
However,  by construction $b^*_1$ and $b^*_2$ agree on $\cA_{s_1}^{\anE}$ where $s_1 =
\gamma(a_1)$. They also agree on $\cA_s^{\anE}$ for all $s$ sufficiently
close to $s_1$. 
Let $t_1,t_2,\ldots$ be a sequence of elements in $[0,a_1]$ with limit
$a_1$. Then ${b^*_1}^{-1} (\xi,\gamma(t_k))$ converges to $P_1$ as $k\rightarrow\infty$. For $k$
sufficiently large we have $\gamma(t_k)\in \unitdisc_2$ and therefore
${b^*_1}^{-1}(\xi,\gamma(t_k)) = {b^*_2}^{-1}(\xi,\gamma(t_k))$.
So $P_1\in \pi_1^{-1}(\unitdisc_1\cap \unitdisc_2)$
 is not isolated in the fiber of $b_2 :
 \cA_{\unitdisc_2}\rightarrow\IT^{2g}$ restricted to $Y_{\unitdisc_2}$ above $\xi$. 

Now we repeat the process and transport $P_1$ along
$\gamma([a_1,a_2])$ to obtain $P_2\in Y_{\unitdisc_2}$ with $\pi(P_2) =
\gamma(a_2)$ that is not isolated
in $b_3|_{Y_{\unitdisc_3}}$. Eventually, we will have returned to the fiber $\cA_{\bp}$.
The final point lies in $Y_{\bp}^{\anE}$ and it is obtained from
$P_0 \in Y^{\anE}$ by a continuous group automorphism of
$\cA_{\bp}^{\anE}$ that depends on the homotopy class of $\gamma$
relative to $\{0,1\}$. More precisely, by construction the final point
is $\widetilde\rho([\gamma])(P_0)$ where
\begin{equation*}
  \widetilde\rho \colon \pi_1(\an{S},\bp) \rightarrow \{ \text{homeomorphisms
} \an{\cA}_{\bp}\rightarrow \an{\cA}_{\bp}\text{ that are group homomorphisms}\}
\end{equation*}
is a group homomorphism that is compatible with the monodromy representation \eqref{eq:monodromyrep}, indeed
\begin{equation*}
  \widetilde \rho(h)_* = \rho(h)\quad\text{for all}\quad
h\in \pi_1(\an{S},\bp)
\end{equation*}
and part (i) follows. 

The proof of  (ii) relies on (\ref{eq:functoriality}) and 
some  basic functoriality. Let $\bp\in \an{S}$. 
 A homomorphism
$\alpha:\cA\rightarrow\cB$ of abelian schemes over $S$ induces  a
group homomorphism $(\alpha|_{\an{\cA}_\bp})_*:H_1(\an{\cA}_\bp,\IZ)\rightarrow
H_1(\an{\cB}_{\bp},\IZ)$. Moreover, this group homomorphism
 is equivariant with respect to the action
of $\pi_1(\an{S},\bp)$ on both homology groups. 
By abuse of notation let $\widetilde \rho$ denote the continuous action of
$\pi_1(\an{S},\bp)$ on $\cA_{\bp}$ and $\cB_{\bp}$
and $\rho$ the induced action on homology. 
We find 
\begin{equation*}
\left(\widetilde \rho(h)\alpha|_{\cA_{\bp}^{\anE}}\right)_*
= \rho(h)(\alpha|_{\cA_{\bp}^{\anE}})_*
= (\alpha|_{\cA_{\bp}^{\anE}})_*\rho(h)
= (\alpha|_{\cA_{\bp}^{\anE}}\widetilde \rho(h))_*
\end{equation*}
for all $h\in \pi_1(\an{S},\bp)$, the first and third
equality follow from (\ref{eq:functoriality}), the second one follows since the monodromy action commutes
with homomorphisms of abelian varieties.
Both self-maps $\widetilde \rho(h)\alpha|_{\an{\cA}_{\bp}}$
and $\alpha|_{\an{\cA}_{\bp}}\widetilde \rho(h)$ 
are continuous group endomorphisms of $\an{\cA}_\bp$,
which is homeomorphic to $\IT^{2g}$. 
As their induced maps on homology coincide, they must coincide as well.
\end{proof}

\subsection{Monodromy on Abelian Schemes}

Let $S$ be an irreducible and smooth curve over $\IC$ and
let $\cA$ be an abelian scheme over $S$
of relative dimension $g\ge 1$. 
We write $\overline{\IC(S)}$ for an algebraic closure
of the function field $\IC(S)$  of $S$. 


For a base point $\bp\in S(\IC)$ the monodromy representation is (\ref{eq:monodromyrep}).
Let $G_s$ denote the  Zariski closure of $\Gamma_s = \rho(\pi_1(\an{S},s))$ in 
$\autS_\IQ{H_1(\an{\cA}_s,\IQ)}$ and let $G^0_s$ be its connected
component containing the unit element.
Deligne proved in \cite[Corollaire~4.2.9]{Deligne:Hodge2}
that
$G^0_s$ is a semisimple algebraic group. 

The next lemma uses  a Theorem of  Tits
 connected to his famous ``alternative''.


\begin{lemma}
\label{lem:expgrowth}
In the notation above suppose that $G^0_s$ is not trivial, 
then any finite index subgroup of  $\Gamma_s$ has 
a free subgroup on $2$ generators. 
\end{lemma}
\begin{proof}
Let $\Gamma'$ be a finite index subgroup of $\Gamma_s$.
As $G^0_s$ is of finite index in $G_s$
we see that   $\Gamma'\cap G_s^0(\IQ)$
 lies Zariski dense in $G_s^0$.
Our lemma follows from \cite[Theorem~3]{tits:freesubgroups}
applied to $G_s^0$ and 
$\Gamma'\cap G_s^0(\IQ)$. 
\end{proof}

Certainly, $G_s$ and $G^0_s$ etc. depend on $s$. 
However their isomorphism classes do not
and  the index $[G_s: G^0_s]$ is independent of $s\in \an{S}$, see
the comments  before Zarhin's \cite[Theorem~3.3]{Zarhin:08}.

\begin{lemma}
\label{lem:zerotrace}
Let $A$ be the generic fiber of
$\cA\rightarrow S$, it is an abelian variety over $\IC(S)$. 
If $s\in S(\IC)$ and $H_1(\an{\cA}_s,\IZ)$ has a non-zero
element that is invariant under the monodromy action
(\ref{eq:monodromyrep}), then the $\IC(S)/\IC$-trace of $A$ is non-zero.
\end{lemma}
\begin{proof}
We write 
$H_1(\an{\cA}_s,\IZ)^{\rho}$
for the elements in $H_1(\an{\cA}_s,\IZ)$ that are invariant under
(\ref{eq:monodromyrep}). 
A conclusion of  Deligne's Theorem of the Fixed Part, see
\cite[Corollaire~4.1.2]{Deligne:Hodge2}, 
implies that
the weight $-1$ Hodge structure on 
$H_1(\an{\cA}_s,\IZ)$ restricts to a Hodge structure  on
$H_1(\an{\cA}_s,\IZ)^{\rho}$.

It is well-known that Hodge substructures of 
$H_1(\an{\cA}_s,\IZ)$ come from abelian subvarieties of $\cA_s$. 
Hence
$H_1(\an{\cA}_s,\IZ)^{\rho}$ 
gives rise to an abelian subvariety
$B \subset \cA_s$ of dimension $\frac 12 \mathrm{Rank}\,
H_1(\an{\cA}_s,\IZ)^{\rho}$. 
As $H_1(\an{\cA}_s,\IZ)^{\rho}\not=0$ by hypothesis  we have $\dim B\ge 1$. 

Then $\cB = B\times_{\spec{(\IC)}} S$ is a constant abelian scheme over $S$. 
The monodromy representation
$\pi_1(\an{S},s) \rightarrow \aut{H_1(\cB_s^{\anE},\IZ)}$ is
certainly trivial. The inclusion $\cB_s \rightarrow\cA_s$ induces a
 homomorphism
$H_1(\cB_s^{\anE},\IZ)\rightarrow H_1(\cA_s^{\anE},\IZ)$ and
the restriction of $\rho$ from \eqref{eq:monodromyrep} to the image of this homomorphism 
 is trivial.
A theorem of Grothendieck \cite{grothendieck:absch} implies, that any element in 
\begin{equation*}
\Hom{\cB_s,\cA_s}\cap \Hom{H_1(\cB_s^{\anE},\IZ),H_1(\cA_s^{\anE},\IZ)}
\end{equation*}
is induced by the restriction of a morphism $\varphi
\colon  \cB \rightarrow \cA$ over $S$
to $\cB_s$ such that $\varphi\circ 0_{\cB} = 0_{\cA}$
where $0_{\cA}\colon S\rightarrow \cA$ and $0_{\cB}\colon S\rightarrow\cB$ are the
zero sections. See also \cite[4.1.3.2]{Deligne:Hodge2}.

The  restriction of $\varphi$ to the generic fiber of $\cB$ 
is a homomorphism
$B\otimes_{\IC} \IC(S) \rightarrow \cA\times_{S} \spec{\IC(S)}=A$
of abelian varieties over $\IC(S)$. 
If the $\IC(S)/\IC$-trace of 
$A$ is trivial, then the said  homomorphism is trivial.
In this case, the morphism $\varphi$ and the zero section both extend
$B\otimes_{{\IC}} {\IC(S)} \rightarrow A$
to a morphism $\cB\rightarrow \cA$.
As the generic fiber lies Zariski dense in $\cA$ we find that
 $\varphi$ is the zero section. But then $B$ must be
trivial and this is a contradiction. 
\end{proof}


For us, an abelian subscheme of $\cA$ is the image of an
endomorphism of $\cA$.
We call $s \in S(\IC)$ \emph{extendable} for $\cA$ if any abelian
subvariety $B_s \in \cA_s$ extends to an abelian subscheme $\cB$ of
$\cA$, \textit{i.e.} there exists an abelian subscheme $\cB$ of $\cA$
such that $\cB \cap \cA_s = B_s$.

For readers who are familiar with Hodge theory, extendable points of $S$ are closely related to Hodge generic points. We shall not go into details and but state the following corollary of a result of Deligne for our purpose.
 \begin{lemma}
\label{lem:endgeneric}
In the notation above suppose $G^0_s = G_s$ for some $s\in S(\IC)$.
There is an at most countable infinite subset
of $S(\IC)$ whose complement consists only of 
extendable points of $\cA$. 
\end{lemma}
\begin{proof}
We refer to \cite[Corollary~3.5 and the preceding comments]{Zarhin:08} for this result. In fact in the reference, it is pointed out that the extendable points are precisely the Hodge generic points under this mild assumption ($G^0_s = G_s$ for some $s\in S(\IC)$).

More precisely \cite[Corollary~3.5 and the preceding
comments]{Zarhin:08} says that any $s \in S(\IC)$ outside an at most
countably infinite set $\Sigma$ satisfies the following property: For
any $\alpha_s\in \en{\cA_s}$  there exists $n\in \IN$ such that
$n\alpha_s$ is the restriction of an endomorphism of $\cA$.
 Now for any $s \in S(\IC)\setminus \Sigma$, any abelian subvariety $B_s$ of $\cA_s$ is the image of some $\alpha_s \in \en{\cA_s}$. There exists $n \in \IN$ such that $n \alpha_s$ is the restriction of an element $\alpha \in \en{\cA}$. And then we can take $\cB$ to be the image of $\alpha$.
\end{proof}



Let $Y$ be an irreducible closed subvariety of $\cA$ that dominates
$S$. Then $Y$ is flat  over $S$ by  \cite[Proposition III.9.7]{Hartshorne}.
 We write $Y_s$ for the fiber of
 $Y\rightarrow S$ above $s$ with the reduced induced  structure.
By \cite[Corollary III.9.6]{Hartshorne} we see that  $Y_s$ is equidimensional of dimension
$\dim Y-1$. 

We say that $Y$ is \emph{virtually monodromy invariant} above $s\in
S(\IC)$ if there exists an irreducible component $Z$ of $Y_s$ and a
subgroup $G \subset \pi_1(\an{S},s)$ of finite index such that
\begin{equation*}
\widetilde\rho(\gamma) (Z(\IC)) = Z(\IC)\quad\text{for all}\quad
\gamma\in G
\end{equation*}
for the representation $\widetilde\rho$ defined in 
Proposition \ref{prop:globalinfo}.

\begin{lemma}
\label{lem:degenerate}
In the notation above we suppose $Y$ is an irreducible closed
subvariety of $\cA$ that dominates $S$. 
We assume  that there is an uncountable set
$M\subset S(\IC)$ satisfying all of the following properties:
\begin{enumerate}
\item [(i)]
for all irreducible $S'$
that are finite and \'etale over $S$ the  generic fiber of
$\cA\times_S S'\rightarrow S'$ has trivial $\IC(S')/\IC$-trace,  
\item[(ii)] all elements in $M$ are extendable for $\cA$ (see Lemma~\ref{lem:endgeneric} and above for definition),
\item[(iii)] and the variety $Y$ is virtually monodromy invariant above all
elements in $M$. 
\end{enumerate}
 Then there exist an abelian
scheme $\cC$ over $S$ and a homomorphism $\cA\rightarrow\cC$ of
abelian schemes over $S$
whose kernel contains $Y$ and  has dimension $\dim Y$.
\end{lemma}
\begin{proof}
Our proof is by induction on 
\begin{equation*}
\dim \cA.
\end{equation*}
The small possible value is $2$ as we
require $g\ge 1$. We call this the minimal case and we treat it
directly below.

Let $s\in S(\IC)$ be arbitrary for the moment. If  $G_s^0$, defined near
 the beginning of this subsection, is trivial
 then the image of $\pi_1(\an{S},s)$ under (\ref{eq:monodromyrep}) is finite.
By the Riemann Existence Theorem there is an irreducible curve $S'$ that is finite and \'etale
 over $S$ such that the monodromy representation of the fundamental
 group of $S'$ at some base point $s'\in S'(\IC)$ on 
 $H_1(\an{(\cA\times_{S} S')}_{s'},\IZ)$ is trivial.
Recall that  $g\ge 1$. By
 Lemma \ref{lem:zerotrace}  the generic fiber of
 $\cA\times_S S'\rightarrow S'$ has non-zero $\IC(S')/\IC$-trace, contradicting our
 hypothesis. Therefore, $\dim G_s^0 \ge 1$. 

Let $s\in M$ with $M$ as in the hypothesis. A finite index subgroup of $\pi_1(\an{S},s)$ acts
on $Z_s(\IC)$  via
(\ref{eq:monoaction}) where $Z_s$ is an irreducible component of $Y_s$. 
We write  $\Gamma'_s$ for  the image of this finite index subgroup under
the monodromy representation (\ref{eq:monodromyrep}).
Then $\Gamma_s'$ has a free subgroup on $2$ generators
 by Lemma~\ref{lem:expgrowth}. 
We invoke Lemma \ref{lem:zerotrace} by using hypothesis (i) and
passing to a covering of $S$   and find that no non-zero
element of $H_1(\an{\cA}_s,\IZ)$ is invariant under the action
of $\Gamma'_s$. 

We aim to apply 
 Proposition \ref{prop:invariance}. 
But first let us verify that $\Gamma'_s$ is of monodromy type with
 respect to corresponding Betti map. Indeed, an abelian
 subvariety $B$ of $\cA_s$ 
extends to an abelian subscheme $\cB$ of 
$\cA$ by hypothesis (ii). 
Then $\widetilde\rho_\cA(\gamma)(\an{B})=\widetilde\rho_\cA(\gamma)(\iota(\an{B}))
= \iota(\widetilde\rho_\cB(\gamma)(\an{B}))=\an{B}$ by
 Proposition \ref{prop:globalinfo}(ii)
for all $\gamma\in\pi_1(\an{S},\bp)$
where $\iota:\cB\rightarrow\cA$ is the inclusion. 

By Proposition \ref{prop:invariance}. 
we are in one of two
cases for any given $s\in M$. Let $M_{1,2}$ be the set of $s\in M$
such that we are in case $1,2$, respectively. As $M= M_1\cup M_2$ one among
 $M_1, M_2$ is uncountable. 

\medskip
\noindent \fbox{Case 1:} The set $M_1$ is uncountable.

For all $s\in M_1$ the subvariety $Z_s$ is contained in the translate
 of a proper abelian subvariety $B_s$ of $\cA_s$ by a point $P_s$ of
 finite order $N_s\in\IN$. As $M_1$ is
 uncountable and $\IN$ is countable, we may replace $M_1$ by an
 uncountable subset and assume that there exists $N\in\IN$ such that
 $[N] P_s = 0$ for all $s\in M_1$.

Let us treat  the minimal case $\dim\cA=2$ now. 
Then
 $B_s=\{0\}$ and thus
 $Z_s=\{P_s\}$ for all $s\in M_1$. But then 
$Y$ contains an infinite, and hence Zariski dense, set of points lying
 in $\ker([N]\colon \cA \rightarrow \cA)$. 
This completes the proof in the
 minimal case as we can take $\cC = \cA$ and $[N]\colon \cA \rightarrow \cA$. 

We now treat the non-minimal  case $\dim\cA\ge 3$. 
By condition (ii) there exists an abelian subscheme $\cB(s)$ of $\cA$ such that $\cB(s) \cap \cA_s = B_s$ for any $s \in M_1$. But $M_1$ is uncountable and $\cA$ has only countably many abelian subschemes, so we may replace $M_1$ by an uncountable subset and assume that there exists an abelian subscheme $\cB$ of $\cA$ with $\cB(s) = \cB$, \textit{i.e.} $\cB \cap \cA_s = B_s$, for all $s \in M_1$.

We have $[N]Z_s \subset \cB \cap \cA_s$ for all $s \in M_1$. But $\bigcup_{s \in M_1} [N]Z_s$ is Zariski dense in $[N]Y$ by dimension reasons, so $[N]Y \subset \cB$ by taking the Zariski closures on both sides.

Clearly $\cB$ satisfies the analog trace condition (i) of the current
lemma by basic properties of the trace. Any $s \in M_1$ is
extendable for $\cB$ because it is extendable for $\cA$ and $\cB$ is
an abelian subscheme of $\cA$.
Finally $[N]Y$, as a subvariety of $\cB$, is virtually monodromy invariant at each $s \in M_1$. To see this it suffices to prove that $[N]Y$ is virtually monodromy invariant as a subvariety of $\cA$ by Proposition~\ref{prop:globalinfo}(ii). But then it suffices to show that $[N]Z_s$ is an irreducible component of $[N]Y_s$. This is true because $[N]Z_s$ is Zariski closed (as $[N]$ is proper) and $\dim [N]Y_s = \dim [N]Y - 1 = \dim Y -1 = \dim Z_s = \dim [N]Z_s$.

We observe that
$\dim \cB = \dim B_s + 1 \le (\dim \cA_s-1)+1  = \dim \cA - 1$.
By induction  there is an abelian scheme $\cC$ over
$S$ and a homomorphism $\psi \colon  \cB\rightarrow\cC$ of abelian
schemes 
over $S$ 
whose kernel contains $[N]Y$ and 
$\dim\ker\psi = \dim [N]Y = \dim Y$. Then $(\ker\psi)^{\circ}$, 
the identity component \cite[\S 6.4]{NeronModels} of $\ker\psi$,
 has dimension $\dim Y$ and is an abelian subscheme of $\cB$ and hence of $\cA$. There exists an integer $m \in \IN$ such that $[m]\ker\psi \subset (\ker\psi)^{\circ}$. In particular $[mN]Y \subset (\ker\psi)^{\circ}$. Note that $\dim (\ker\psi)^{\circ} = \dim \ker\psi = \dim Y$.

Now it suffices to take $\cA \rightarrow \cC$ to be the composition $\cA\xrightarrow{[mN]} \cA \rightarrow \cA/(\ker\psi)^{\circ}$.

 
\medskip
\noindent \fbox{Case 2:} The set $M_2$ is uncountable.

For all $s\in M_2$ there exists an abelian subvariety $B_s\subset \cA_s$
with $\dim B_s\ge 1$ and $Z_s+B_s = Z_s$. 
Note that $\dim Y\ge 2$ since $\dim Z_s \ge 1$, so we are
not in the minimal case. 

By condition (ii) there exists an abelian subscheme $\cB(s)$ of $\cA$
such that $\cB(s) \cap \cA_s = B_s$ for any $s \in M_2$. Since $M_2$
is uncountable and $\cA$ has only countably many abelian subschemes,
we may replace $M_2$ by an uncountable subset and assume that there
exists an abelian subscheme $\cB$ of $\cA$ with $\cB(s)
= \cB$, \textit{i.e.} $\cB \cap \cA_s = B_s$, for all $s \in M_2$.

We shall work with the abelian scheme $\cA/\cB$ over $S$. Let $\varphi \colon \cA \rightarrow \cA/\cB$ be the natural quotient. Then any fiber of $\varphi$ has dimension $\dim_S \cB = \dim \cB -1$.  The
condition $Z_s+B_s = Z_s$ implies that the fibers of $\varphi|_{Z_s}$
have dimension $\dim B_s = \dim\cB -1$ for all $s \in M_2$. Since $\bigcup_{s \in M_2} Z_s$ is Zariski dense in $Y$ by dimension reasons, we see that a general fiber of $\varphi|_Y$ has dimension $\dim \cB-1$. Thus by Fiber Dimension Theorem we have
\[
\dim Y = \dim \cB -1+ \dim \varphi(Y). 
\]

Clearly $\cA/\cB$ satisfies the analog trace condition (i) of the current lemma by basic properties of the trace. Any $s \in M_2$ is extendable for $\cA/\cB$ because any abelian subvariety of $(\cA/\cB)_s$ is the quotient of an abelian subvariety of $\cA_s$ and $s$ is extendable for $\cA$. Finally $\varphi(Y)$ is virtually monodromy invariant above all points in $M_2$ by Proposition~\ref{prop:globalinfo}(ii).

Now since $\dim(\cA/\cB) = \dim \cA - \dim B_s \le \dim \cA -1$, there exist by induction an abelian scheme $\cC$ over $S$ and a homomorphism $\psi \colon \cA/\cB \rightarrow \cC$ whose kernel has dimension $\varphi(Y)$ and contains $\varphi(Y)$. Then $Y \subset \ker(\psi\circ \varphi)$ since $\varphi(Y) \subset \ker(\psi)$. But
\begin{equation*}
\dim \ker(\psi \circ \varphi) 
= \dim_S \cB+ \dim \ker(\psi) = \dim\cB -1 +\dim\varphi(Y) =\dim Y. 
\end{equation*}
So $\psi\circ \varphi \colon \cA \rightarrow \cC$ is what we desire.
\end{proof}

\subsection{End of  the Proof of Theorem~\ref{PropositionDegenerate}}
Now we are ready to prove Theorem~\ref{PropositionDegenerate}.

Let $Y$ be an irreducible closed
subvariety that is degenerate. We want to prove that $Y$ is generically special.

Note that being generically special is a property on the geometric
generic fiber. Moreover, it is enough to show that one irreducible
component on the geometric generic fiber of $Y$ has the property
stated in Definition~\ref{DefinitionSpecialSubvariety}.
We will remove finitely points from 
 $S$ and replace $S$ by a finite and \'etale covering $S'$ which we
 assume to be
 irreducible throughout this proof. 
Observe that the base change $Y'$ of $Y$ may no longer be irreducible. 
But it is \'etale over $Y$ and thus reduced. 
In particular, $Y'$ is flat over $Y$ and thus over $S$. 
It follows that $Y'$ is
equidimensional of dimension $\dim Y$
by \cite[Corollary~III.9.6]{Hartshorne}.
Note that if $U$ is an open subset of $\an{Y}$ consisting of
degenerate points for $Y$, then its preimage will be open in $\an{Y'}$
and consist of degenerate points. 

So to ease notation we will write $S'=S$ below and take $Y$ to be an
irreducible component of $Y'$.

Let $A$ be the generic fiber of $\mathcal{A}$. 
After possibly removing finitely many points from $S$ and replacing by
a finite \'etale covering   we may assume 
 that $A^{\overline{\IC(S)}/\IC} =
A^{\IC(S)/\IC}$. We also assume
that all abelian subvarieties of $A\otimes_{\IC(S)}\overline{\IC(S)}$
are defined over $\IC(S)$. 
By passing to a further finite \'etale covering we may assume that 
$\cA$ 
 satisfies the
hypothesis of Lemma \ref{lem:endgeneric}. 
Let $\Sigma\subset S(\IC)$ be a countable subset such that
any element in $S(\IC)\setminus \Sigma$ is extendable for $\cA$.

Let $U$ be a non-empty, open subset of $\an{Y}\cap\cA_\unitdisc$
consisting of degenerate points for $Y$;
here is $\unitdisc$ as
above Theorem~\ref{PropositionDegenerate}.

If $s\in \an{S}$, then $\pi_1(S^{\anE},s)$ 
acts on $\an{\cA}_s$ via
 (\ref{eq:monoaction}). 
We have proven in Proposition \ref{prop:globalinfo} 
that $\widetilde\rho(\gamma)(P) \in \an{Y}$ for all $P\in U$ and all $\gamma\in\pi_1(S^{\anE},s)$. This property continues to hold with $U$ replaced by the  union $\bigcup_{\gamma} \widetilde\rho(\gamma)(U)$ over $\pi_1(S^{\anE},s)$. Note that $U$ is open and invariant under the action of the fundamental group.

Let $Z$ be an
 irreducible
component  of $Y_s$ 
with $\an{Z}\cap U\not=\emptyset$. 
The representation $\widetilde\rho$ maps $\an{Z}\cap U$ into 
$Y^{\anE}_s$. As everything is real analytic we see that 
for each $\gamma\in\pi_1(S^{\anE},s)$ there is an irreducible
 component $Z^{\prime}$ of $Y_s$ such that 
$\widetilde\rho(\gamma) (\an{Z}\cap U) \subset Z^{\prime\anE}\cap U$.
Because all irreducible components of $Y_s$ have dimension equal to
 $\dim Y-1$ and  by the Invariance of Domain Theorem 
we conclude that $Z^\prime$ is uniquely determined by 
$\widetilde\rho(\gamma) (\an{Z}\cap U)\subset Z^{\prime\anE}\cap U$
among all irreducible components of $Y_s$. 
We conclude that $\pi_1(S^{\anE},s)$
 acts on the finite set of irreducible components of $Y_s$ that meet
 $U$. 
Therefore, 
$\widetilde\rho(\gamma) (\an{Z}\cap U) \subset\an{Z}\cap U$
for all $\gamma$ in
a finite index subgroup  of $\pi_1(S^{\anE},s)$.

The smooth locus of $\an{Z}$ is path-connected,
lies dense in $\an{Z}$, and contains a point of
 $\an{Z}\cap U$. 
By fixing  piece-wise real analytic paths
we find that 
$\widetilde\rho(\gamma) (\an{Z}) \subset\an{Z}$
for all $\gamma$ in the finite index subgroup mentioned before.


The arguments above show that  $Y$ is virtually monodromy invariant
above  $s$. %
Clearly, $U\setminus\Sigma$ 
is an uncountable set as $U$ is open in $\an{S}$ and
non-empty.

Let us suppose
$A^{\overline{\IC(S)}/\IC} = 0$
for the moment. 
We  can apply Lemma~\ref{lem:degenerate} to $Y,\cA,$
and $M$ equal to the set of $s$ obtained from  $U\backslash \Sigma$
and conclude that $Y$ is an irreducible component of a subgroup
scheme of $\cA$ that is generically special. 
This completes the proof of  Theorem~\ref{PropositionDegenerate}
in the current case. 

Let us turn to the general case. Recall that 
 $\pi \colon \cA\rightarrow S$
 is an abelian scheme with generic fiber $A$ whose 
 ${\IC(S)/\IC}$-trace is $A^{{\IC(S)}/\IC}$.
We take $A_0$ to be
 $A^{{\IC(S)}/\IC} \otimes_\IC \IC(S) \subset A$.
So the ${\IC(S)}/\IC$-trace of $A/A_0$ vanishes,
\textit{cf.} \cite[Theorem~6.4 and the following comment]{Conrad}.
 Moreover, $(A/A_0)^{\overline{\IC(S)}/\IC}=0$
as $A^{\overline{\IC(S)}/\IC} = A^{\IC(S)/\IC}$. 

By \cite[Proposition 3 \S 7.5]{NeronModels}, the N\'eron model $\cB$ of $A/A_0$ is an abelian scheme over $S$ and 
sits in the short exact sequence of
abelian schemes over $S$
\begin{equation*}
  0 \rightarrow A^{{\IC(S)}/\IC} \times S  \rightarrow  \cA \overset{\varphi}{\rightarrow} \cB\rightarrow
 0.
\end{equation*}

In $A$ we fix an abelian subvariety $C$ that meets
$A_0$
in a finite set and with $A_0 + C=A$. Let $\cC$ be the
N\'eron model of $C$. It is an abelian scheme over $S$ and 
we may assume $\cC\subset\cA$. 
The restriction $\varphi|_{\cC}\colon \cC\rightarrow\cB$ is dominant
and proper, hence surjective. It is fiberwise an isogeny of abelian
varieties. We conclude $(A^{{\IC(S)}/\IC} \times S)+\cC = \cB$. 

As $Y$ is degenerate there exists an open and non-empty
$U$ subset of $\an{Y}$ of degenerate points. 
By shrinking $U$ we may assume that
$\varphi|_{Y}:Y\rightarrow \varphi(Y)$ is smooth at all points of $U$. So $\varphi(U)$ is  open in
$\an{\varphi(Y)}$.
This set consists of degenerate points for $\varphi(Y)$.
By the previous case ($A^{\overline{\IC(S)}/\IC} = 0$),
the set of $P\in U$ such that $\varphi(P)$ has finite
order in the corresponding fiber of $\cB$ lies Zariski dense in $Y$.

We consider such a $P$ and suppose $\varphi(P)$ has  order $N$
and write 
$P = Q+R$ with  $Q\in (A^{{\IC(S)}/\IC} \times S)(\IC)$ and $R\in\cC(\IC)$, where
$Q,R$ lie in the same fiber above $S$ as $P$. So
$0=[N](\varphi(P))=\varphi([N](R))$. As $R\in \cC(\IC)$  it
must have finite order $N'$.
Moreover, $R \in Y - \sigma_Q$ where
$\sigma_Q$ is 
the image of a constant section $S \rightarrow A^{{\IC(S)}/\IC} \times
S$
with value $Q$.

The Betti map is constant on sufficiently small open subsets of
$\sigma_Q$ as $A^{{\IC(S)}/\IC} \times S$ is a constant abelian scheme. Therefore, $R$ is a degenerate point of $Y-\sigma_Q$. 

Recall that the  order of a point is constant on a fiber of the Betti
map. By the second claim in Proposition \ref{prop:globalinfo}(i) there exists an
irreducible component $C\subset \cC[N']$ containing $R$ with 
$C\subset Y -\sigma_Q$.  

We conclude that $P$ is a point of $\sigma_Q+C$, a generically special
subvariety of $\cA$. As this holds for a Zariski dense set of $P$ in
$Y$ we conclude from Proposition \ref{TheoremStarSetOpen} that $Y$ is generically special.


\section{Construction of the Auxiliary Variety}\label{SectionAuxiliarySubvariety}
In this section we work in the category of schemes over an algebraically closed subfield $F$
of $\IC$. 
We abbreviate $\IP_F^m$ by $\IP^m$ throughout this section. 
Suppose  $S$ is  a smooth irreducible algebraic curve.
Let $\cA$ be an abelian scheme of relative dimension $g\ge 1$
over $S$ with
structural morphism $\pi \colon \cA\rightarrow S$.
For a closed subvariety $X\subset \cA$ and $s\in S(F)$ we 
write $X_s = \pi^{-1}(s)$.

We assume that $\cA$ comes equipped with  an admissible immersion 
$\cA\rightarrow\IP^M\times \IP^m$ 
as in \S \ref{SubsectionEmbeddingAbSch}, \textit{i.e.}, it satisfies conditions (A1), (A2), and (A3) in $\mathsection$\ref{SubsectionEmbeddingAbSch}. 
In particular,  each fiber $\cA_s$ of $\pi$ with $s\in S(F)$
is an abelian variety in $\IP^M$. On this projective space we 
let $\deg(\cdot)$ denote the degree of an algebraic set.

In this section $X$ will denote  an  irreducible, closed subvariety of $\mathcal{A}$ 
that dominates $S$ and with $X\not=\cA$. 
Hence $\pi|_X \colon X\rightarrow S$ is surjective as $\pi|_X$ is proper. 
We write $\dim X = \dim \cA-n=g+1-n$ where $n\ge 1$ is the codimension
of $X$ in $\cA$.

Let $\unitdisc\subset S^{\mathrm{an}}$ be a non-empty open subset with
Betti map $b \colon \pi^{-1}(\unitdisc)=\cA_{\unitdisc} \rightarrow \IT^{2g}$. See
Proposition~\ref{PropositionBetti}, we recall that $\IT$ denotes the circle group.
It is convenient to write $X_{\unitdisc} = \an{X}\cap \cA_{\unitdisc}$.


The following convention will be used in this section. If $P$ is a point on a real (resp. complex) analytic manifold $Y$, then 
$T_P(Y)$ denotes the tangent space of $Y$ at $P$. This is an $\IR$-
resp. 
 $\IC$-vector space, depending on whether $Y$ is a real
or complex analytic manifold. If $Z$ is another real (resp.  complex)
analytic manifold  and $f \colon Y\rightarrow Z$
is a real (resp. complex) analytic mapping, then $T_P(f)$ denotes the 
differential $T_P(Y)\rightarrow T_{f(P)}(Z)$. It is $\IR$- (resp.
$\IC$-)linear. Let $\image{T_P(f)}$ denote the image of $T_P(f)$ in
$T_{f(P)}(Z)$. 

Recall that $\sman{X}$ is the complex analytic space
attached to  the smooth locus $\sm{X}$ of $X$. 
If $P\in \cA_{\unitdisc}\cap \cA(F)$ then $b|_{\sman{X}\cap\cA_{\unitdisc}}
\colon \sman{X}\cap\cA_{\unitdisc}\rightarrow\IT^{2g}$ is a real analytic map. The
condition in the proposition below concerns the image of its
differential.

\begin{proposition}\label{PropositionConstructionOfAuxiliarySubvariety}
We keep the notation from above and assume that $X$ is not generically
special. 
Suppose $P \in \cA_\unitdisc\cap\cA(F)$ with $\pi(P)=s$ such that $P$ is a
smooth point of 
$X_{s}$ and of $X$ with
\begin{equation}
\label{eq:dimTPhyp}
  \dim \image{T_P(b|_{\sman{X}\cap\cA_{\unitdisc}})} = 2 \dim X. 
\end{equation}
Then
there exists a closed irreducible subvariety $Z\subset \cA$ over $F$ with the following properties. 
\begin{enumerate}
\item [(i)] We have $\dim Z = n$ and $Z$ dominates $S$. 
 \item[(ii)] We have that $P$ is a smooth point of
   $Z_s$ and of $Z$. 
 \item[(iii)] The fiber $Z_s$ does not contain any positive dimensional coset in $\cA_s$.
  \item[(iv)] There exists $D\ge 1$ such that  $\deg  Z_t \le D$
for all $t\in S(\IC)$. 
\item[(v)] We have $\image{T_P(b|_{\sman{X}\cap\cA_{\unitdisc}})} \cap
  \image{T_P(b|_{\sman{(Z_s)}})} = 0$ in $T_{b(P)}(\IT^{2g})$. 
\end{enumerate}
Moreover, the set
\[
\{ t \in S(\IC) :  \text{$Z_t$ contains a positive dimensional coset
  in $\cA_t$}\}.
\]
is finite.
\end{proposition}


Condition (v) implies that $Z_s$ and $X_s$ intersect transversally in $\cA_s$. Condition (i) implies that $Z_t$ is equidimensional of dimension $n-1$ for all $t\in S(F)$
by \cite[Corollary III.9.6 and Proposition III.9.7]{Hartshorne}.


We will prove this proposition in the next few subsections, see $\mathsection$\ref{SubsectionConstructionStep1}-\ref{SubsectionConstructionStep2} for the construction of $Z$ and $\mathsection$\ref{SubsectionControlOfBadFibers} for the ``Moreover'' part. But first, let
us relate  its hypothesis \eqref{eq:dimTPhyp} to our notion of
generically special. A crucial point is to use Theorem~\ref{PropositionDegenerate}.

\begin{lemma}\label{LemmaNonDegeneratePoint}
Suppose that $X$ is not generically special.
Then there exists $P\in \sm{X}(F)$ with $\pi(P)\in \unitdisc$ and $P\in (X_{\pi(P)})^{\mathrm{sm}}(F)$ that satisfies
(\ref{eq:dimTPhyp}). 
\end{lemma}
\begin{proof}
Let us consider the restriction 
\begin{equation*}
  b|_{\sman{X}\cap\cA_\unitdisc }:\sman{X}\cap\cA_\unitdisc\rightarrow\IT^{2g}. 
\end{equation*}
Observe that domain and target are smooth manifolds of
dimension $2\dim X$ and $2g$, respectively. 

Let $r \in \{0,\ldots,2g\}$ denote the largest possible rank of 
$T_{P}(b|_{\sman{X}\cap\cA_\unitdisc })$ as $P$ ranges over the domain. 
Then there exists an open and non-empty subset $\mathfrak{U}$ of
$\sman{X}\cap\cA_\unitdisc$ on which the rank is $r$. It follows from
\cite[Appendix II,~Corollary 7F]{Whitney}, that any fiber
of $b|_{\mathfrak{U}} \colon \mathfrak{U}\rightarrow
b(\mathfrak{U})$ is a smooth manifold of dimension $2\dim X -r$. 

By hypothesis and Theorem~\ref{PropositionDegenerate} the variety $X$ is
not degenerate. In particular, 
 there exists $P\in \mathfrak{U}$ that is not degenerate for $X$.
So the  fiber of
$b|_{\mathfrak{U}}$ 
through $P$ contains $P$
as an isolated point. So we have $r=2\dim X$.

By continuity we may assume that \eqref{eq:dimTPhyp} holds for all
$P\in \mathfrak{U}$, after possibly shrinking $\mathfrak{U}$. 

On the other hand the set $U = \{ P \in \sm{X}(F) : \pi|_X \colon X\rightarrow S \text{ is
smooth at } P\}$ is Zariski open and dense in $X$. So
$U(F) \cap \mathfrak{U} \not= \emptyset$ because $F$ is algebraically
closed and dense in $\IC$. Now any point $P \in U(F) \cap \mathfrak{U}$ satisfies the desired properties as $S$ is smooth. 
\end{proof}

For the further construction of $Z$ we assume that  $P$ is as in this lemma. 

\subsection{The First Four Properties}\label{SubsectionConstructionStep1}

We show how to construct $Z$ satisfying the first four properties in
the proposition.
Indeed, our construction will show that a generic choice, in a
suitable sense, of $Z$ will
suffice for (i)-(iv). Later on we will see how to obtain in addition
(v) and deduce the final statement.

Let $P$ be as in the hypothesis of Proposition \ref{PropositionConstructionOfAuxiliarySubvariety}. Recall that $\cA$ comes with an admissible immersion
$\cA\rightarrow \IP^M\times \IP^m$ as
in \S \ref{SubsectionEmbeddingAbSch}. 
 Observe that $\cA_s\subset\IP^M$ is  Zariski closed, irreducible, and contains $P$ as a smooth point as it is an abelian variety.
By property (A3) 
a generic homogeneous linear form $f \in F[X_0,\ldots,X_M]$ vanishing
at $P$ satisfies the following property. The intersection of the zero
locus $\zeroset{f}$ of $f$ with $\cA_s$ contains no positive dimensional
cosets in $\cA_s$. Here generic means that we may allow the coefficients of $f$ to come from a Zariski open dense subset of all possible coefficient vectors.

According to Bertini's Theorem there are linearly independent homogeneous linear forms
$f_1,\ldots,f_{g+1-n} \in F[X_0,\ldots,X_M]$ 
 such that their set of common zeros
$\zeroset{f_1,\ldots,f_{g+1-n}}$ in $\IP^M$ intersects $\cA_s$ in a
   Zariski
closed set $Z'$ that is smooth at $P$ and of dimension 
 $\dim \cA_s -
(g+1-n) = g - (g+1-n) = n-1$. 
 If $n\ge 2$ we  may arrange that $Z'$ is irreducible by applying a
 suitable variant of Bertini's Theorem. 
By the previous paragraph, we can arrange that $Z'$ contains no positive dimensional cosets in $\cA_s$. We will see that this establishes (ii), (iii), and (iv) with our choice of $Z$ below. 

Note that a generic choice of $(f_1,\ldots,f_{g+1-n})$ in $F[X_0,\ldots,X_M]^{\oplus (g+1-n)}$, where each
entry has degree one, that vanishes
at $P$ will have the property described in the previous paragraph. 
Here generic means that we may allow the coefficient vector attached to
$(f_1,\ldots,f_{g+1-n})$  to come from a Zariski open dense subset of
all possible coefficient vectors that lead to linear forms with
coefficients in $F$ vanishing
at $P$.
We may arrange $f_1$ to be an $f$ as in the last paragraph, so $Z'$
contains no coset of positive dimension.


Each irreducible component of 
\begin{equation}
\label{eq:IPMprimeZcap}
\left( \zeroset{f_1,\ldots,f_{g+1-n}} \times \IP^m \right) \cap \cA
\end{equation}
has dimension at least $n$. 
Suppose $Z$ is an irreducible component of \eqref{eq:IPMprimeZcap}
 that  contains $P$. By the Fiber Dimension Theorem we find
$\dim Z_s  \ge  \dim Z - \dim \pi(Z) \ge \dim Z - 1$. As $\dim Z' =
n-1$ and $Z_s \subset Z'$ we conclude $\dim Z \le n$. 
Thus $\dim Z = n, \dim \pi(Z)=1,$ and $\dim Z_s=n-1$. This implies both claims in
(i).

If $n = 1$, then $\dim Z_s = 0$ and hence $P$ is smooth in $Z_s$. If $n \ge 2$, then $Z_s = Z'$ and hence $P$ is smooth in $Z_s$ by construction.
Now as $P$ is smooth in $Z_s$ and $s$ is smooth in $S$, $P$ is also smooth in $Z$. This establishes (ii).

If $n = 1$, then $\dim Z_s = 0$ and  (iii) clearly holds. If $n \ge 2$,
then by construction $Z_s$ satisfies (iii). In both cases, 
 $Z_t$ 
is a union of  irreducible components 
of  $\zeroset{f_1,\ldots,f_{g+1-n}} \cap \cA_t$
for all but at most finitely many  $t \in S(\IC)$. For these $t$ 
we conclude $\deg  Z_t \le \deg  \cA_t$ from
B\'{e}zout's Theorem. But $\cA \rightarrow S$ is a flat family
embedded in $\IP^M \times S \rightarrow S$, so
$\deg \cA_t \le D$ for some $D\ge 1$ depending only on $\cA$ 
and the immersion. We can take care of the remaining finitely many fibers by
increasing $D$ if necessary. Thus we have established (iv).

\subsection{The Fifth Property}\label{SubsectionConstructionStep2}

\subsubsection{Linear Algebra}

For a  $\IC$-vector space $T$ we write $T_\IR$ for $T$ with its
natural structure as an $\IR$-vector space. For example, if $T$ is
finite dimensional, then $\dim T_{\IR} = 2\dim T$. A vector subspace 
$V_0$  of $T_{\IR}$ is naturally an $\IR$-vector space. 
We denote by $\IC V_0$ the smallest vector subspace of $T$ containing
$V_0$. For example, if $V_0 = \IR v_1 + \cdots + \IR v_k$, then $\IC
V_0 = \IC v_1 + \cdots + \IC v_k$.
Let $J$ denote the multiplication by $\sqrt{-1}$ map $J \colon T\rightarrow
T$.
Then $(\IC V_0)_{\IR} = V_0 + J V_0$.
A  vector subspace of $T_\IR$ is a  vector
subspace of $T$ if and only if it is
 $J$-invariant.

In this section $g\ge 1$ is an integer. We show that an even
 dimensional
 real
subspace of $\IC^g$ intersects some complex subspace of complementary
real dimension transversally. 


\begin{lemma}
\label{lem:genericCvs}
Let $T$ be a $\IC$-vector space of dimension $g$ and suppose $W$
is a vector subspace  of $T$ with $\dim W = m$. Let $V_0$ be a vector subspace of $T_{\IR}$ of
 dimension $2m+2k$ that contains $W$. Then there
 exists a vector subspace $V$ of $T$ of dimension $g-(m+k)$ such that $V \cap V_0 = 0$.
\end{lemma}

Before proving this lemma, let us do the following preparation.

\begin{lemma}
Let $\IC^{2k}$ be the standard complex vector space of dimension $2k$
and let $\IR^{2k} \subset \IC^{2k}$ be the real part of
$\IC^{2k}$, \textit{i.e.} $\IC^{2k}
= \IR^{2k} \oplus \sqrt{-1} \IR^{2k}$. Then there exists a vector subspace $V$ of $\IC^{2k}$ of dimension $k$ such that $V \cap \IR^{2k} = 0$. 
\end{lemma}
\begin{proof}
For any $j = 1,\ldots,2k$, we let $e_j=(0,\ldots,0,1,0,\ldots,0) \in \IC^{2k}$ be the vector with the $j$-th entry being $1$ and the other entries being $0$.

For any $i = 1,\ldots,k$, we let $v_i = e_{2i-1} + \sqrt{-1} e_{2i} \in \IC^{2k}$. We show that the complex vector space $V = \IC v_1 + \cdots + \IC v_k$ satisfies the desired property.

Indeed, $\dim V = k$ as $v_1,\ldots,v_k$ are $\IC$-linearly
independent. 
So it remains to show $V \cap \IR^{2k} = 0$. Any vector in $V$ is of the form $c_1 v_1 + \cdots + c_k v_k$ for some $c_1,\ldots,c_k \in \IC$. If $c_1 v_1 + \cdots + c_k v_k \in \IR^{2g}$, then we have
\[
c_i \in \IR\text{ and } \sqrt{-1}c_i \in \IR\text{ for all }i=1,\ldots,k.
\]
Thus $c_1=\cdots=c_k=0$.
\end{proof}

\begin{lemma}
Let $U$ be a $\IC$-vector space of dimension $2k$
 and let $V_0$ be a vector subspace of $U_{\IR}$ of dimension $2k$
 such that $\IC V_0 = U$. Then there exists a vector subspace $V$ of $U$ of dimension $k$ such that $V \cap V_0 = 0$.
\end{lemma}
\begin{proof}
We take a basis of $V_0$ which is an $\IR$-vector space, and  call it $e_1,\ldots,e_{2k}$. Since $\IC V_0 = U$, we have $U = \IC e_1 + \cdots + \IC e_{2k}$. But $\dim U = 2k$, so $e_1,\ldots,e_{2k}$ form a basis of $U$.

Now under the identification $U = \IC^{2k}$ via the basis $e_1,\ldots,e_{2k}$, the vector subspace $V_0$ of $U$ becomes the real part of $\IC^{2k}$. We can apply the previous lemma to conclude.
\end{proof}

Now we are ready to prove Lemma~\ref{lem:genericCvs}. 

\begin{proof}[Proof of Lemma~\ref{lem:genericCvs}]
We begin by showing that we can reduce to the case $m=0$. If the lemma
is known when $m=0$, then we apply it to the $\IC$-vector space $T/W$
and the image of the $\IR$-vector space $V_0$ in this quotient to get
a  vector subspace $V'$ of $T/W$ of dimension $g-(m+k)$. 
Let $W^\perp$ be a vector subspace of $T$ 
with $W+W^\perp = T$ and $W\cap W^\perp=0$. 
Then the natural linear map $W^\perp \rightarrow T/W$ is an isomorphism.
The preimage of $V'$ under this map 
 is the vector subspace which we desire.

Now we treat the case $m=0$, note that $W=0$ in this case.
As above we write $J$ for multiplication by $\sqrt{-1}$ on $T$. 
Then $(\IC V_0)_{\IR} = V_0 + J V_0$.

\noindent \fbox{Case (i)} The $\IR$-vector space $V_0$ contains no non-zero vector subspace of $T$.

In this case $V_0\cap J V_0$, being a $J$-invariant vector subspace of
 $V_0$, must be trivial.
 So $\dim (\IC V_0)_{\IR} = \dim V_0 + \dim J V_0 = 2k+2k=4k$ and hence
 $\dim \IC V_0 = 2k\le g$. Thus we can apply the previous lemma to $U = \IC
 V_0$ and $V_0$ to get a  vector subspace $V'$ of $\IC V_0$ of
 dimension $k$ such that $V' \cap V_0 = 0$. Then it suffices to take
 $V = V' + V''$ for any vector subspace  $V''\subset T$
with $\IC V_0+V''=T$ and $\IC V_0\cap V''=0$. 

\noindent \fbox{Case (ii)} General case.

We write $V_0^J$ for the largest $J$-invariant vector subspace
of $V_0$. 
As it is $J$-invariant by definition,
 we consider it as a $\IC$-vector space.
 Then $T' = T/V_0^J$ is a $\IC$-vector space of dimension $g- \dim
 V_0^J$, and $V_0' = V_0 / V_0^J$ is a  vector subspace of $T'_{\IR}$ of dimension $2 (k - \dim V_0^J)$.

We claim that $V_0'$ contains no non-zero vector subspace of the
$\IC$-vector space $T'$. If $U'$ is a vector subspace of
$T'$ with 
$U'\subset V_0'$, then its preimage under the quotient $T \rightarrow T' =
T/V_0^J$ is a vector subspace of $T$ that is contained in $V_0$ and that
contains $V_0^J$. The maximality of $V_0^J$ yields $U'=0$.

Now we can apply case (i) to $T'$ and $V_0' \subset T'_{\IR}$ to get a
vector subspace $V'$ of $T'$ of dimension $(g-\dim V_0^J) - (k-\dim
V_0^J) = g-k$ such that $V' \cap V_0' = 0$. Let $V''$ be the preimage
of $V'$ under the quotient $T \rightarrow T' = T/V_0^J$. Then $V''$ is
a vector subspace of $T$ with dimension $g-k + \dim V_0^J$ such that
$V'' \cap V_0 = V_0^J$. Recall that $V_0^J$ is a vector subspace of
the $\IC$-vector space $T$, and hence a vector subspace of $V''$. Now it suffices to let $V$ be any complement of $V_0^J$ in $V''$.
\end{proof}

Let $g$ and $T$ be as in Lemma \ref{lem:genericCvs} and suppose $k\ge 0$ is an
integer. 
Let  $\gra{T_\IR,2k}$ denote the set of all $2k$-dimensional vector subspaces
of $T_\IR$. 
On identifying $T_\IR$ with
$\IR^{2g}$ we may use Pl\"ucker coordinates
to identify $\gra{T_\IR,2k}$ with a closed subset of $\IP^{N}(\IR)$
where
$N={2g \choose
2k}-1$.

Note that $\IP^N(\IR)$ is equipped with the archimedean topology that makes it a
compact Hausdorff space. We will use this topology and  its induced
subspace topology on  $\gra{T_\IR,2k}$. 

Multiplication by $\sqrt{-1}$ induces an $\IR$-linear automorphism 
$T_\IR\rightarrow T_\IR$ and hence a selfmap 
$\gra{T_{\IR},2k}\rightarrow \gra{T_{\IR},2k}$. 
By the Cauchy-Binet Formula this selfmap can be described 
on $\gra{T_{\IR},2k}\subset \IP^N(\IR)$
by linear forms. 
Its fixpoints are precisely the $2k$-dimensional
 vector subspaces of $T_{\IR}$
that are $k$-dimensional vector subspaces of 
$T$. 
We write $\gra{T,k}$ for the set of these fix points. It is a closed
subset of $\gra{T_{\IR},2k}$. 
In this notation we can use Lemma \ref{lem:genericCvs} to prove the following.

\begin{lemma}\label{LemmaGenericSprime}
Let $T$ be a $\IC$-vector space of dimension $g$ 
and suppose $W$ is a vector subspace of $T$ with
$\dim W=m\le g-1$.
Let $V$ be a vector subspace of $T_{\IR}$ of dimension $2(m+1)$
 that contains $W$. 
There exists a non-empty open (in the archimedean topology) subset $\mathfrak{U}\subset \gra{T,g-m-1}$
such that  $V'\cap V=0$ for all $V'\in \mathfrak{U}$. 
\end{lemma}
\begin{proof}
By the Cauchy-Binet Formula the set 
\begin{equation*}
\{ V' \in \gra{T_\IR,2(g-m-1)} : V'\cap V = 0 \} 
\end{equation*}
is the complement in $\gra{T_\IR,2(g-m-1)}$ of the zero set of a
homogeneous linear
polynomial defined on the projective coordinates of $\IP^N$. 
So the set in question is open in $\gra{T_\IR,2(g-m-1)}$. Its
intersection $\mathfrak{U}$ with $\gra{T,g-m-1}$ lies open in $\gra{T,g-m-1}$. 
But $\mathfrak{U}\not=\emptyset$ by Lemma \ref{lem:genericCvs} applied to the case $k=1$. 
\end{proof}

\subsubsection{Verification of (v) in  Proposition \ref{PropositionConstructionOfAuxiliarySubvariety}}
We retain  the conventions made in
$\mathsection$\ref{SectionAuxiliarySubvariety} and
$\S$\ref{SubsectionConstructionStep1}. So $P$ is as in
the hypothesis.
We set up the various vector spaces needed in Lemma
\ref{LemmaGenericSprime}. 

For $T$ we take the tangent space $T_P(\an{\cA}_s)$ which is a
$\IC$-vector space of dimension $g$. 
Note that $P$ and $\pi$ are defined over $F$, so this complex space $T$ also descends to $F$.

For $W$ we take the image of $T_P(\sman{(X_s)})$ under the
linear map
\[
T_P(\sman{(X_s)})\rightarrow T_P(\an{\cA}_s)
\]
induced by the inclusion
$\sm{(X_s)}\rightarrow \cA_s$  (recall  that $P$ is a smooth
point of $X_s$). 
Its dimension equals $\dim X_s = \dim X - 1 =
g-n\le g-1$ which we define as $m$. 

Finally,  we take $V = \image{T_P(b|_{\sman{X}\cap\cA_\unitdisc})}$ which is a vector
subspace of the $\IR$-vector space $T_{b(P)}(\IT^{2g})$. 
As $b|_{\an{\cA}_s} \colon \an{\cA}_s \rightarrow \IT^{2g}$ is an
isomorphism of real analytic spaces, we can identify 
$T_P(\an{\cA}_s)=T$ with $T_{b(P)}(\IT^{2g})$ as $\IR$-vector spaces. 
Therefore, $V\subset T_{\IR}$ as in the setup of Lemma
\ref{LemmaGenericSprime}. Note that $V$ does not carry a complex
structure as we treat $\IT^{2g}$  as a real analytic space.

Our hypothesis (\ref{eq:dimTPhyp}) implies 
$\dim V = 2 \dim X = 2(m+1)$. 
So the hypothesis of Lemma \ref{LemmaGenericSprime} is satisfied.

So there exists $\mathfrak U$ open in  $\gra{T,g-m-1}$, the latter is
a compact Hausdorff space. Its points correspond to 
$(g-m-1)$-dimensional vector subspaces of the $\IC$-vector space $T$. 

In  $\mathsection$\ref{SubsectionConstructionStep1} we saw that
 a generic choice of $f_1,\ldots,f_{g+1-n}$ vanishing at $P$ yields
 properties (i)-(iv) in the proposition. To obtain (v) we must
 make sure that $V\cap \image{T_P(b|_{\sman{(Z_s)}})}=0$. 
According to Lemma
 \ref{LemmaGenericSprime} this holds if
$\image{T_P(b|_{\sman{(Z_s)}})}$, a vector  subspace of the $\IC$-vector space
$T_P(\an{\cA}_s)$,  under the identification
of $T_{b(P)}(\IT^{2g})$ with $T_P(\an{\cA}_s)$   made above,
lies in  $\mathfrak U$.
Ranging over all possible choices of $f_1,\ldots,f_{g+1-n}$ as in 
$\mathsection$\ref{SubsectionConstructionStep1} yields 
points in $U(F)$ for some Zariski open dense subset $U\subset\gra{T,g-m-1}$.
As $\mathfrak U$ is open in the archimedean topology
and since $U(F)$ lies dense in $\gra{T,g-m-1}$
we have $\mathfrak
U \cap U(F)\not=\emptyset$. 
Any element in $\mathfrak U\cap U(F)$ is sufficient and this completes
the proof.\qed

\subsection{A Detour to B\'ezout's Theorem}

In this subsection we prove the following degree bound on long
intersections. It will be used to prove the ``Moreover'' part of Proposition~\ref{PropositionConstructionOfAuxiliarySubvariety} in the next subsection.
In this subsection we temporarily allow $F$ to be any algebraically
closed field of characteristic $0$. 

\begin{proposition}\label{prop:uniformintersectionbound}
Suppose $V_1,\ldots,V_m$ are irreducible closed subvarieties of
$\IP_F^n=\IP^n$ such that $\deg(V_i)\le \delta$ for all $i \in \{1,\ldots,m\}$. 
Let $C_1,\ldots,C_r$ be all the irreducible components of $V_1\cap \cdots \cap V_m$ of top dimension which we denote by $k$, then 
\begin{equation}
\label{eq:uniformdegreebound}
\sum_{i=1}^r \deg(C_i) \le \delta^{n-k}. 
\end{equation}
\end{proposition}

The crucial aspect of (\ref{eq:uniformdegreebound}) is that the
right-hand side is independent of $m$.

\begin{lemma}\label{LemmaSubvarietyReplacedByHypersurfaces}
Let $V$ be an irreducible closed subvariety of $\IP^n$ of degree
$\delta$. Then there exist finitely many irreducible hypersurfaces of $\IP^n$ of degree at most $\delta$ such that $V$ is their intersection.
\end{lemma}
\begin{proof}
This is Faltings's  \cite[Proposition 2.1]{Faltings:DAAV}. 
\end{proof}

\begin{proof}[Proof of Proposition~\ref{prop:uniformintersectionbound}]
By Lemma~\ref{LemmaSubvarietyReplacedByHypersurfaces}, we may assume
that every $V_i$ is an irreducible hypersurface for all
$i\in\{1,\ldots,m\}$. Then $V_i = \zeroset{f_i}$ is the zero locus of an irreducible homogeneous polynomial $f_i \in F[X_0,\ldots,X_n]$ of degree at most $\delta$.

We shall prove inductively on $s\in \{1,\ldots,n-k\}$ that there exist
hypersurfaces
$H_1,\ldots,H_{n-k}$ (possibly reducible) of degree at most $\delta$
such that
for all $s\in \{1,\ldots,n-k\}$,
\begin{enumerate}
\item[(i)]  each irreducible component of
$\bigcap_{j=1}^{s} H_j$ has  dimension at most $n-s$
\item[(ii)] and $C_i \subset  \bigcap_{j=1}^{s} H_j$ for each $i\in \{1,\ldots,r\}$.
\end{enumerate}
Assume this for $s=n-k$. Then each $C_i$, being of dimension $k$, is an irreducible component of $\bigcap_{j=1}^{n-k} H_j$, and thus
\[
\sum_{i=1}^r \deg C_i \le \prod_{j=1}^{n-k} \deg H_j \le \delta^{n-k}
\]
by B\'ezout's Theorem, \textit{cf.} \cite[Example 8.4.6]{Fulton}
which holds here even though the hypersurfaces $H_j$ may be reducible.

Let us take $H_1 = V_1$. Then $\deg H_1 \le \delta$.


Now suppose we have constructed $H_1,\ldots,H_{s-1}$ for some $2 \le s \le n-k$.

Let $W_1,\ldots,W_t$ be the irreducible components of
$H_1 \cap \cdots \cap H_{s-1}$. Let $l\in \{1,\ldots,t\}$. Since $s \le n-k$, we have $\dim
W_l > k$. By assumption each irreducible component of
$\bigcap_{i=1}^m \zeroset{f_i} = \bigcap_{i=1}^m V_i$ has dimension at most $k$, so there
exists some $i_0 \in \{1,\ldots,m\}$ such that $\widetilde f_l = f_{i_0}$ does not
vanish on $W_l$. 
Then $\widetilde f_l \in F[X_0,\ldots,X_n]$ has degree at most $\delta$
and vanishes on $C_1\cup\cdots\cup C_r$. 
We may assume that $\deg \widetilde f_l = \delta$ after possibly
multiplying $\widetilde f_l$ with a homogeneous polynomial of a suitable
degree in general position.

Let $F[X_0,\ldots,X_n]_{\delta}$ be the union of $0$ and the homogeneous polynomials in $F[X_0,\ldots,X_n]$ of degree $\delta$. We can identify $F[X_0,\ldots,X_n]_{\delta}$ with $\mathbb{A}^{\binom{n+\delta}{n}}(F) = F^{\binom{n+\delta}{n}}$. 
Then
\[
\{f \in \mathbb{A}^{\binom{n+\delta}{n}}(F) : 
f\text{ vanishes on $C_1\cup\cdots\cup C_r$}
\}
\]
is the set of $F$-points of a linear subvariety
 $L\subset \mathbb{A}^{\binom{n+\delta}{n}}$, and  $\{f \in L(F):
 f|_{W_l} \not= 0\}$ defines a   Zariski open $U_l$ 
in $L$ that is non-empty as $\widetilde f_l\in
 U_l(F)$. 
Now $L$ is irreducible and so $U_l$ is Zariski open and
 dense in $L$. 
In particular, the intersection 
\[
\Theta = \bigcap_{l=1}^t U_l(F)
\]
is non-empty. 

Now fix any $f_s\in\Theta$ and  let $H_s = \zeroset{f_s}$. Then $H_s$
has degree at most $\delta$ and no irreducible component of
$H_1 \cap \cdots \cap H_{s-1} \cap H_s$ is an irreducible component of
$H_1 \cap \cdots \cap H_{s-1}$. So the irreducible components of
$H_1\cap\cdots \cap H_{s}$ have dimension at most $n-s$ using (i) in
the case $s-1$.  Property (ii) clearly holds by the construction of
the $U_l$. 
\end{proof}

\subsection{Control of Bad Fibers}\label{SubsectionControlOfBadFibers}
In this subsection we prove the ``Moreover'' part of Proposition~\ref{PropositionConstructionOfAuxiliarySubvariety}.

Recall our setting: $\pi \colon \cA \rightarrow S$ is an abelian
scheme of relative dimension $g \ge 1$ over a smooth irreducible
curve, defined over $F\subset \IC$.
For simplicity we assume $F=\IC$. 
We have constructed an auxiliary subvariety $Z$ of $\cA$ in Proposition~\ref{PropositionConstructionOfAuxiliarySubvariety}. It remains to show that 
\[
\{t \in S(\IC) :  \text{$Z_t$ contains a positive dimensional coset
  in $\cA_t$}\}.
\]
is finite. It fact, we show that it follows from condition (iii) of Proposition~\ref{PropositionConstructionOfAuxiliarySubvariety}. More precisely we shall prove the following result.

\begin{proposition}
Let $Z$ be an irreducible closed subvariety of $\cA$ dominating $S$. Suppose $s \in S(\IC)$ such that 
 $Z_{s}$ contains no positive dimensional cosets in $\cA_{s}$. 
Then
\begin{equation*}
  \{ t \in S(\IC) :  \text{$Z_{t}$ contains a positive dimensional coset
  in $\cA_{t}$}\}
\end{equation*}
is finite. 
\end{proposition}
\begin{proof}
  Let $\ell$ be a prime that we will choose in terms of $\cA,Z,$ and $s$
  later on. 

We begin by introducing full level $\ell$ structure. We will take care
to ensure that various quantaties are uniform in $\ell$. 

Let $S'$ be an irreducible, quasi-projective curve over $\IC$ that
is also finite and \'etale over $S$ such the base change $\cA' = \cA
\times_S S'$ admits all $\ell^{2g}$ torsion sections 
\begin{equation*}
 S' \rightarrow \cA'[\ell]. 
\end{equation*}

We write $Z' = Z\times_S S'$ which comes with a closed immersion
$Z' \rightarrow \cA'$. Observe that $Z$ may now longer be irreducible.
But $Z'\rightarrow S'$ is flat as $Z\rightarrow S$ is, cf.
\cite[Proposition III.9.7]{Hartshorne}. So all irreducible
components of $Z'$ dominate $S'$.
The morphism $Z' \rightarrow Z$ is finite and flat since $S'\rightarrow S$ is. 
Therefore $Z'$ is equidimensional of dimension $\dim Z$
by \cite[Corollary III.9.6]{Hartshorne}. 
Finally, $Z'$ is reduced since
$Z'\rightarrow Z$ is \'etale and $Z$ is reduced.
For any $t' \in S'(\IC)$ above $t \in S(\IC)$,  
we may identify the fiber $Z_t$ with $Z'_{t'}$ and the fiber $\cA_t$ with $\cA_{t'}$. 
Let $Z'_1,\ldots,Z'_r$ be the irreducible components of $Z'$.
Both $(Z'_i)_{t'}$
and $Z_t$ are equidimensional of dimension $\dim Z'-1=\dim Z -1$ as $Z'\rightarrow S'$ and $Z\rightarrow
S$ are flat, cf. \cite[Corollary III.9.6]{Hartshorne}. 
Since   $(Z'_i)_{t'}\subset Z_t$  an irreducible component of $(Z'_i)_{t'}$ is also an irreducible
component of $Z_t$.
Moreover, $Z_{t}$  contains a positive dimensional coset in $\cA_{t}$
if and only if one among $(Z'_1)_{t'},\ldots,(Z'_r)_{t'}$ contains a positive
dimensional coset in $\cA'_{t'}$. 

To prove the proposition we may thus suppose that
$S=S'$, $\cA=\cA'$, and $Z$ is some $Z'_i$. In particular, $\ell^{2g}$ distinct torsion sections
$S\rightarrow \cA[\ell]$ exist.

For any non-zero torsion section  $\sigma \colon S\rightarrow \cA[\ell]$ 
we define
\begin{equation*}
  Z(\sigma) = Z \cap (Z- \sigma) \cap (Z- [2]\circ\sigma)\cap
\cdots\cap (Z-[\ell-1]\circ\sigma)
\end{equation*}
by identifying a section $S\rightarrow \cA$ with its image in
$\cA$. 
Then $Z(\sigma)$ is Zariski closed in $\cA$. 

Now suppose $t\in S(\IC)$ such that $Z_t$ contains $P+B$ where
$P\in \cA_t(\IC)$ and $B\subset \cA_t$ is an abelian subvariety of
positive dimension. 
Therefore, $B[\ell]$ is a non-trivial group and there exists a section
 $\sigma:S\rightarrow\cA[\ell]$ such that $\sigma(t)\in B[\ell]\backslash\{0\}$. 
Hence $\sigma(t)+B=B$ and we find
\begin{equation*}
P +   B  = P + B -[k](\sigma (t))   \subset Z_t-[k](\sigma (t))
\quad\text {for all}\quad k\in\IZ.
\end{equation*}
This implies $P+B \subset Z(\sigma)_t$. In particular, $t\in
\pi(Z(\sigma))$. 

Now $\pi$ is a proper morphism and so $\pi(Z(\sigma))$ is Zariski
closed in $S$ for all of the finitely many $\sigma$ as above. 
In order to prove the proposition it suffices to show that $s$
from Proposition \ref{PropositionConstructionOfAuxiliarySubvariety} does
not lie in any $\pi(Z(\sigma))$ if $\sigma\not=0$, for then all $\pi(Z(\sigma))$ are finite. We will prove that
$Z(\sigma)_{s} = \emptyset$ for all non-zero sections $\sigma \colon
S\rightarrow\cA[\ell]$. 

Recall that the admissible immersion from the beginning of this
 section induces a  
 polarization on $\cA_{s}$ and, as usual, 
we  use $\deg(\cdot)$ to denote the
 degree. 
This polarization and $\cA_s$ do not depend on the base changed
 defined using $\ell$. 
Let us assume $Z(\sigma)_{s} \not= \emptyset$.  This will lead to a contradiction for
$\ell$ large in terms of $Z_{s},\cA_{s},$ and the polarization.

Note that
\begin{equation}
\label{eq:Ysigmas}
  Z(\sigma)_{s} = Z_{s} \cap (Z_{s}-\sigma(s)) \cap\cdots \cap
  (Z_s-[\ell-1]\circ \sigma(s))
\end{equation}
is Zariski closed in $\cA_s$ and
 stable under translation by the subgroup of $\cA_s(\IC)$ of order
$\ell$ that is generated by
$\sigma(s)$.
Observe that if $W'$ is an irreducible component of $Z(\sigma)_s$
of maximal dimension,  then
$\sigma(s)+W'$ is also an irreducible component of $Z(\sigma)_s$. 
We define $W$ to be the union of the top dimensional irreducible components
of $Z(\sigma)_{s}$.
The group generated by $\sigma(s)$ acts on the set of irreducible
components of $W$. 

Recall that $Z_s$ and thus each $Z_s-[k](\sigma(s))$ with $k\in\IZ$
is equidimensional of dimension $\dim Z -1$. All irreducible
components that appear have degree bounded by a constant independent
of the auxiliary prime $\ell$. 
By (\ref{eq:Ysigmas}) and  Proposition \ref{prop:uniformintersectionbound}
the degree $\deg(W)$
is bounded from above by a constant $c\ge 1$ that is independent of
$\ell$. 

The number $N$ of irreducible components of $W$ is at most
$\deg(W)\le c$. If we assume $\ell > N$, then the symmetric group on $N$
symbols contains no elements of order $\ell$. 
So if we assume $\ell >c$, as we may, then $\sigma(s) + W'=W'$ for all
irreducible components $W'$ of $W$.

Now let us fix such an irreducible component $W'$. Then the subgroup
generated
by $\sigma(s)$ lies in the stabilizer $\stab{W'}$ of $W'$. 
 By \cite[Lemme~2.1(ii)]{DavidHindry}, the degree of
the stabilizer $\stab{W'}$ is bounded from above solely in terms
$\deg(W')$ and $\dim W'\le g$. Note also that $\deg(W') \le \deg(W) \le c$. Thus if $\ell$ is large in terms of $c$ and
$g$, then we can arrange $\ell > \deg \stab{W'}$.
But $\stab{W'}$  contains 
 $\sigma(s)$ which has
order $\ell$.
Therefore $B$, the connected component of $\stab{W'}$ containing the neutral
element, has positive dimension. Fix any $P\in W'(\IC)$. Then 
\begin{equation*}
 P+B \subset W' \subset W
\subset Z_{s}
\end{equation*}
contradicts the hypothesis that $Z_{s}$ does not contain a positive
dimensional coset. 
\end{proof}


\section{Lattice Points}\label{SectionLatticePoints}
For our abelian scheme $\cA \rightarrow S$ and subvariety $X \subseteq \cA$, we want to count the number of points in $[N]X \cap Z$ for each $N \gg 1$ where $Z \subseteq \cA$ is of complimentary dimension of $X$ (as constructed in Proposition~\ref{PropositionConstructionOfAuxiliarySubvariety}). It is equivalent to count the intersection points of $[N]X - Z$ and the zero section of $\cA \rightarrow S$. Via the Betti map and a local lift with respect to $\IR^{2g} \rightarrow \IT^{2g}$, we obtain a subset $\tilde U_N \subseteq \IR^{2g}$ from $[N]X - Z$ and we are led to counting lattice point in $\tilde U_N$. 
The goal of this section is to settle the lattice point counting problem.

Suppose $m,m'\in\IN$ and let
  $\psi$ be a function defined on a non-empty open subset   $U$ of  $\IR^{m'}$
with values in $\IR^m$. We suppoes that the  coordinate functions of
  $\psi$ lie in
  $C^1(U)$, the $\IR$-vector space of
  real valued functions on $U$ that are continuously differentiable. 
We write $D_z(\psi)\in \mat{mm'}{\IR}$ for
the jacobian matrix of $\psi$ evaluated at $z\in U$. 
We also set
\begin{equation*}
  |\psi|_{C^1} =  \max \left\{\sup_{x\in U} |\psi(x)|,\sup_{x\in U} \left|\frac{\partial
    \psi}{\partial x_1}(x)\right|,\ldots,
\sup_{x\in U} \left|\frac{\partial \psi}{\partial x_{m'}}(x)\right|\right\} \in \IR
\cup \{\infty\}
\end{equation*}
here $|\cdot|$ is the maximum norm on $\IR^m$.
We write $\vol{\cdot}$ for the Lebesgue measure on $\IR^m$. 
Recall that all open subsets of $\IR^m$ are measurable. 

For $i\in \{0,1,2\}$ let $m_i\in\IN$
and suppose $U_i$ 
is a  non-empty open 
subset
of $\IR^{m_i}$.
Let
$\pi_1\colon \IR^{m_0+m_1+m_2}\rightarrow \IR^{m_0+m_1}$
be defined by $\pi_1(w,x,y) = (w,x)$ 
and 
$\pi_2\colon \IR^{m_0+m_1+m_2}\rightarrow \IR^{m_0+m_2}$
by $\pi_2(w,x,y)=(w,y)$. 

We now suppose $m_0=2$ and $m= 2+m_1+m_2$. 
Let $\phi_1 \colon  U_0\times U_1\rightarrow \IR^m$ and $\phi_2 \colon
U_0\times U_2\rightarrow\IR^m$  have continuously differentiable coordinate
functions 
and satisfy
 $|\phi_1|_{C^1}<\infty$ and
  $|\phi_2|_{C^1}<\infty$. Define
\begin{equation}
\label{def:psiN}
\psi_N(w,x,y) = N\phi_1(w,x) -
\phi_2(w,y)
\end{equation}
where $w\in U_0, x\in U_1,$ and $y\in U_2$. 
Thus $\psi_N$ has target $\IR^m$ and 
 coordinate
functions in $C^1(U)$ where $U=U_0\times U_1\times U_2$.

We write $\phi_{1j}$ and
$\phi_{2j}$ for the coordinate functions of
$\phi_1$ and $\phi_2$, respectively.
By abuse of notation we sometimes write $\phi_{1j}(z)$ for $\phi_{1j}(w,x)$
and $\phi_{2j}(z)$ for $\phi_{2j}(w,x)$
if $z=(w,x,y)$ with $w\in U_0, x\in U_1, y\in U_2$.

The jacobian matrix  
$  D_{(w,x,y)}(\psi_N) \in \mat{m}{\IR}$ equals
\begin{equation*}
\left(
\begin{array}{cccccccc}
  N \frac{\partial \phi_1}{\partial w_1} -  \frac{\partial
      \phi_2}{\partial w_1}
  &N \frac{\partial \phi_1}{\partial w_{2}} -  \frac{\partial
      \phi_2}{\partial w_{2}}
 &
  N \frac{\partial \phi_1}{\partial x_1}  &\cdots 
  &N \frac{\partial \phi_1}{\partial x_{m_1}}
 & 
  - \frac{\partial \phi_2}{\partial y_1}  &\cdots 
  &- \frac{\partial \phi_2}{\partial y_{m_2}} 
\end{array}\right). 
\end{equation*}
evaluated at $(w,x,y)\in U$.
For fixed $(w,x,y)$, the determinant $\det D_{(w,x,y)}(\psi_N)$ 
is a polynomial in $N$ of degree at most $2+m_1$. More precisely, we have
\begin{equation*}
\det D_{(w,x,y)}(\psi_N)=
  \delta_0(w,x,y) N^{2+m_1}  + \delta_1(w,x,y) N^{1+m_1} + 
  \delta_{2}(w,x,y) N^{m_1}
\end{equation*}
where the crucial term is 
\begin{equation}
\label{def:delta0}
  \delta_0(w,x,y) = 
\det\left(
\begin{array}{cccccccc}
 \frac{\partial \phi_1}{\partial w_1}  & \frac{\partial
   \phi_1}{\partial w_{2}} & 
 \frac{\partial \phi_1}{\partial x_1} & \cdots & \frac{\partial
   \phi_1}{\partial x_{m_1}} & 
- \frac{\partial \phi_2}{\partial y_1} & \cdots & -\frac{\partial
   \phi_2}{\partial y_{m_2}}
\end{array}\right)\Big|_{(w,x,y)}. 
\end{equation}

If $x$ is in any power of $\IR$ and $r>0$
 we let $B_r(x)$ denote the open ball of radius $r$ around
$x$ with respect to $|\cdot|$.

\begin{lemma} 
\label{lem:volume}
In the notation above let
 $z_0\in U$ with $\delta_0(z_0)\not=0$.
There exist two  bounded open neighborhoods $U^{\prime\prime} \subset U^\prime$ of $z_0$ in $U$  and
a constant $c\in (0,1]$ with the following properties:
\begin{enumerate}
\item [(i)]
the map $\phi_1$ is injective when
 restricted to $\pi_1(U^{\prime}) \subset\IR^{2 + m_1}$,
 \end{enumerate}
and for 
 all real numbers  $N\ge c^{-1}$ 
 \begin{enumerate}
 \item[(ii)]
the map 
 $\psi_N|_{U^{\prime}}\colon U^{\prime} \rightarrow\IR^m $ is  injective and open,
\item[(iii)] we have
$\vol{\psi_N(U^{\prime\prime})} \ge c N^{2+m_1}$, and
\item [(iv)] we have
 $B_c(\psi_N(U^{\prime\prime})) \subset\psi_N(U^\prime)$.
\end{enumerate}
\end{lemma}
\begin{proof}
As the first order partial derivatives of all $\phi_{ij}$ are
continuous we can find an open neighborhood $U^{\prime}$ of $z_0=(w,x,y)$ in $U$
such that the determinant of
\begin{equation}
\label{eq:invmatrix}
\left(
\begin{array}{ccccccccc}
 \frac{\partial \phi_{11}}{\partial w_1}(\tilde z_1)  & \frac{\partial
   \phi_{11}}{\partial w_{2}}(\tilde z_1) & 
 \frac{\partial \phi_{11}}{\partial x_1}(\tilde z_1) & \cdots & \frac{\partial
   \phi_{11}}{\partial x_{m_1}}(\tilde z_1) & 
- \frac{\partial \phi_{21}}{\partial y_1}(\tilde z_1) & \cdots & -\frac{\partial
   \phi_{21}}{\partial y_{m_2}}(\tilde z_1) \\
 \vdots & \vdots & \vdots & & \vdots & \vdots & & \vdots \\
 \frac{\partial \phi_{1m}}{\partial w_1}(\tilde z_m)  & \frac{\partial
   \phi_{1m}}{\partial w_{2}}(\tilde z_m) & 
 \frac{\partial \phi_{1m}}{\partial x_1}(\tilde z_m) & \cdots & \frac{\partial
   \phi_{1m}}{\partial x_{m_1}}(\tilde z_m) & 
- \frac{\partial \phi_{2m}}{\partial y_1}(\tilde z_m) & \cdots & -\frac{\partial
   \phi_{2m}}{\partial y_{m_2}}(\tilde z_m)
\end{array}\right)
\end{equation}
has absolute value at least $\epsilon=|\delta(z_0)|/2>0$ for  all $\tilde
z_1,\ldots,\tilde z_m\in
U^{\prime}$.  

Observe that $D_{\pi_1(z_0)}(\phi_1)$ is an $m\times(2+m_1)$-matrix
consisting of the first $2+m_1$ columns as in the determinant (\ref{def:delta0}). 
Our hypothesis $\delta_0(z_0)\not=0$ implies that
$D_{\pi_1(z_0)}(\phi_1)$ has maximal rank $2+m_1$. 
By the Inverse Function Theorem we may, after shrinking $U^{\prime}$, assume
that
 $\phi_1$ restricted to $\pi_1(U^{\prime})$ 
is injective. 
This implies (i).

We may shrink $U^{\prime}$ further and assume
that
\begin{equation}
\label{eq:BdeltaB2delta}
U^{\prime\prime} = B_{\delta}(z_0) \subset U^{\prime}=B_{2\delta}(z_0)\subset U,
\end{equation}
for some $\delta > 0$, a property we will need later on. Our constant
$c$ will depend on $\delta$ but not on $N$.

To show injectivity in (ii), let $z,z'\in U^{\prime}$ and $N\in\IR$ be such that $\psi_N(z)=\psi_N(z')$.
Let $j\in \{1,\ldots,m\}$, then 
$N\phi_{1j} (z) - \phi_{2j}(z) = N \phi_{1j}(z') - \phi_{2j}(z')$.
By the Mean Value Theorem 
there exists $\tilde{z}_j\in U^{\prime}$ on the line segment connecting 
$z$ and $z'$ such that the column vector $z-z'$ lies in the kernel of 
\begin{equation*}
\left(N\frac{\partial \phi_{1j}}{\partial w_1}-\frac{\partial \phi_{2j}}{\partial w_1},  N\frac{\partial
   \phi_{1j}}{\partial w_{2}}-\frac{\partial \phi_{2j}}{\partial w_2}, 
 N\frac{\partial \phi_{1j}}{\partial x_1}, \ldots, N\frac{\partial
   \phi_{1j}}{\partial x_{m_1}},
- \frac{\partial \phi_{2j}}{\partial y_1}, \ldots , -\frac{\partial
   \phi_{2j}}{\partial y_{m_2}}\right)\Big|_{\tilde{z}_j}
\end{equation*}
Thus $z-z'$ lies in the kernel of 
$M(\tilde z_1,\ldots,\tilde z_m) \in \mat{m}{\IR}$ whose rows 
are  these  expressions as $j\in \{1,\ldots,m\}$. 

 The
determinant of this matrix can be expressed as
\begin{equation*}
  \tilde\delta_0(\tilde z_1,\ldots,\tilde z_m)N^{2+m_1} + \tilde \delta_1(\tilde z_1,\ldots,\tilde z_m)N^{1+m_1} + \tilde \delta_2(\tilde z_1,\ldots,\tilde z_m)N^{m_1}
\end{equation*}
where $\tilde\delta_0(\tilde z_1,\ldots,\tilde z_m)$ is the determinant
of (\ref{eq:invmatrix}). In particular,
$|\tilde\delta_0(\tilde z_1,\ldots,\tilde z_m)|\ge\epsilon$.

We recall that $|\phi_{1,2}|_{C^1}<\infty$. 
So for $i=1,2$ we find
 $|\tilde\delta_i(\tilde z_1,\ldots,\tilde z_m)|\le C$
where $C$ depends only on $\phi_{1}$ and $\phi_2$.  
For all sufficiently large $N\ge 1$ we have
\begin{alignat}1
\label{eq:detMz1zm}  
\begin{aligned}
|\tilde\delta_0(\tilde z_1,\ldots,\tilde z_m)N^{2+m_1} 
+\cdots+
\tilde \delta_2(\tilde z_1,\ldots,\tilde z_m)N^{m_1}|&\ge 
\epsilon N^{2+m_1} - 2CN^{1+m_1} \\
&\ge \frac{\epsilon}{2} N^{2+m_1}.
\end{aligned}
\end{alignat}
And so in particular, $\det M(\tilde z_1,\ldots,\tilde z_m)\not=0$. As $z-z'$ lies
in the kernel of the relevant matrix, we conclude
$z=z'$. Therefore, $\psi_N|_{U^{\prime}}$ is injective for all large $N$.
We conclude injectivity (ii)

If $N$ is sufficiently large, then 
(\ref{eq:detMz1zm}) implies 
$|\!\det M(z,\ldots,z)|\ge \epsilon N^{2+m_1}/2$ for all $z\in U^{\prime}$. 
In particular, $D_{z}(\psi_N)=M(z,\ldots,z)$ is invertible for all $z\in U^{\prime}$. 
Hence $\psi_N$ is locally invertible on $U^{\prime}$ and $\psi_N|_{U^{\prime}}$ is an
open map. This completes the proof of (ii).


As $\psi_N|_{U^{\prime\prime}}$ is injective and for $N$ large,  Integration by Substitution implies
\begin{equation*}
\vol{\psi_N(U^{\prime\prime})} =   \int_{\psi_N(U^{\prime\prime})} du = 
\int_{U^{\prime\prime}} |\!\det D_{z}(\psi_N)| dz
\ge \frac{\epsilon}{2} N^{2+m_1} \vol{U^{\prime\prime}}.
\end{equation*}
This yields our claim in (iii) for small enough $c$.

To prove (iv) it suffices  to verify that if $z\in U^{\prime\prime}$, then the distance
$\Delta(z)$ 
of $\psi_N(z)$ to $\IR^m \ssm \psi_N(U^{\prime})\not=\emptyset$ is at least $c$, for $c>0$ sufficiently
small and independent of $N$.

 As the set $\IR^m\ssm \psi_N(U^{\prime})$ is closed in $\IR^m$  it contains
$v$ which depends on $z$ and $N$,  such that $\Delta(z) = |\psi_N(z)-v|$. 
As $v$ realizes the minimal distance,  the ball $B_{1/n}(v)$ must meet
$\psi_N(U^{\prime})$ for all $n\in\IN$. Let us fix $z_n\in U^{\prime}$ with 
$|\psi_N(z_n)-v|<1/n$. 
Now $U^{\prime}$ is bounded, so after passing to a convergent subsequence we
may assume that $z_n$ converges towards $z'\in \overline{U^{\prime}}
= \overline{B_{2\delta}(z_0)}$. 

We claim that $z'\not\in U^{\prime}$. Indeed, otherwise $\psi_N(z_n)$ would
converge towards $\psi_N(z') \in \psi_N(U^{\prime})$. But then  
$\psi_N(z')=v\in \IR^m\ssm\psi_N(U^{\prime})$ is a contradiction. 
We conclude
\begin{equation}
\label{eq:zonboundary}
|z'-z_0| = 2\delta. 
\end{equation}

By the Mean Value Theorem we find 
\begin{equation}
\label{eq:mvt2}
\psi_N(z) - \psi_N(z_n) = M(\tilde z_1,\ldots,\tilde z_m) (z-z_n)
\end{equation}
where $M(\cdot)$ is the matrix above and
$\tilde z_1,\ldots,\tilde z_n$ lie on the line segment between $z$ and
$z_n$ and thus in $U'$. As above, the absolute determinant of
this matrix
is at least $\epsilon N^{2+m_1}/2$ for $N$ large enough. The entries of the adjoint matrix
have absolute value bounded by a fixed multiple of $N^{2+m_1}$. We
find
$|M(\tilde z_1,\ldots,\tilde z_m)^{-1}|\le c_1$
for the maximum norm where $c_1>0$ is independent
of $N$ and  $\tilde z_1,\ldots,\tilde z_m$. 
We find that (\ref{eq:mvt2}) implies
\begin{equation*}
|z-z_n| = |M(\tilde z_1,\ldots,\tilde z_m)^{-1}(\psi_N(z)
 - \psi_N(z_n))|
\le 
c_2 
|\psi_N(z) - \psi_N(z_n)| 
\end{equation*}
where $c_2>0$ is independent of $N$.
Hence 
\begin{equation*}
|z-z_n|\le 
c_2 (|\psi_N(z)-v| + |v-\psi_N(z_n)|) = c_2 (\Delta(z) + |\psi_N(z_n)-v|)
\end{equation*}
by our choice of $v$. 
Recall that $|\psi_N(z_n)-v|<1/n$ and $z\in U^{\prime\prime}$ which was defined in (\ref{eq:BdeltaB2delta}), so 
\begin{equation*}
|z_n-z_0| - \delta \le |z_n-z_0| - |z_0 - z| 
\le |z-z_n| \le c_2(\Delta(z)+1/n).  
\end{equation*}
By taking the limit as $n\rightarrow \infty$ we can replace $z_n$ by
$z'$ on the left. We recall (\ref{eq:zonboundary}) and conclude 
$\Delta(z)\ge \delta/c_2$. 
Part (iv) follows as we may assume $\delta/c_2\ge c$. 
\end{proof}

Our  aim is to find many integral points in $\psi_N(U)$.
If $\psi_N(U)$ has volume $v$, one could hope that
 $\psi_N(U)$ 
contains at least $v$ points in $\IZ^m$. Of course, simple examples
show that this does not need to be
true in general.  Blichfeldt's Theorem
guarantees that we can find at least this 
 number of lattice points after possibly translating by a point in
 $\IR^m$. 
In our situation we will be able to translate by a rational point of
controlled denominator. For the
reader's convenience we repeat the hypothesis in the next proposition.

\begin{proposition}\label{PropositionCountingLatticePreparation}
Let  $U_0\subset\IR^2,U_{1}\subset\IR^{m_1},$ and
$U_{2}\subset\IR^{m_2}$
be non-empty open subsets and suppose $\phi_1 \colon U_0\times U_1\rightarrow \IR^m$ 
and $\phi_2 \colon U_0 \times U_2\rightarrow\IR^m$ 
have coordinate functions in $C^1(U_0\times U_1)$ and 
$C^1(U_0\times U_2)$, respectively, where $m=2+m_1+m_2$. 
We suppose that $|\phi_{1,2}|_{C^1}<\infty$. Let
 $z_0\in U=U_0\times
U_1\times U_2$ with 
$\delta_0(z_0)\not=0$.
For $N\in\IR$ we define $\psi_N$ as in (\ref{def:psiN}). 
There exists a  bounded open neighborhood $U^{\prime}$ of $z_0$ in $U$
and a
 constant $c\in (0,1]$ 
with the following property. 
For all integers $N_0\ge c^{-1}$ and all real numbers $N\ge c^{-1}$ we have 
\begin{equation*}
\# \left(\psi_N(U^\prime)\cap N_0^{-1}\IZ^m\right) \ge 
c N^{2+m_1}.
\end{equation*}
Moreover, 
$\phi_{1}|_{\pi_1(U^{\prime})}$  is injective, and $\psi_N|_{U^{\prime}}$ is injective for all $N\ge c^{-1}$.
\end{proposition}
\begin{proof}
Let $U^{\prime\prime} \subset U^\prime$ and $c_1> 0$ be as in
Lemma \ref{lem:volume} and suppose $N\ge c_1^{-1}$.              
Below we will use 
$\vol{\psi_N(U^{\prime\prime})}\ge c_1 N^{2+m_1}$
and $B_{c_1}(\psi_N(U^{\prime\prime})) \subset \psi_N(U^{\prime})$. 

By Blichfeldt's Theorem \cite[Chapter~III.2,
Theorem~I]{cassels:geonumbers}, there exists $x\in\IR^m$ which may
depend on $N$, such that
$\#(-x+\psi_N(U^{\prime\prime}))\cap\IZ^m \ge \vol{\psi_N(U^{\prime\prime})}\ge
c_1 N^{2+m_1}$. 
 So there exist an integer $M\ge c_1 N^{2+m_1}$,
 $a_1,\ldots,a_M\in\IZ^m$,
 and $z_1,\ldots,z_M\in U^{\prime\prime}$ such that
\begin{equation*}
 -x+\psi_N(z_i)=a_i\in\IZ^m\quad\text{for all}\quad
 i\in \{1,\ldots,M\}
\end{equation*}
 and the $a_i$ are pairwise distinct.


There exists $c_2 > 0$ such that 
if $N_0$ is any integer with $N_0 \ge c_2^{-1}$ then 
$B_{c_1}(x')\cap N_0^{-1}\IZ^m\not=\emptyset$ for all $x'\in\IR^m$. 

Let us fix $q\in B_{c_1}(x)\cap N_0^{-1} \IZ^m$ where $x$ comes from
Blichfeldt's Theorem.
Then
\begin{equation*}
q+a_i = (q-x)+ x + a_i = (q-x)+\psi_N(z_i) \in
B_{c_1}(\psi_N(z_i)) \subset\psi_N(U^\prime).
\end{equation*}
Observe $q+a_i \in N_0^{-1}\IZ^m$
for all $i\in \{1,\ldots,M\}$.

We have proved 
$\#(\psi_N(U')\cap N_0^{-1}\IZ^{m}) \ge M\ge c_1 N^{2+m_1}$
for all $N\ge c_1^{-1}$ and all $N_0\ge {c_2}^{-1}$.
The proposition follows by taking $c = \min \{c_1,c_2\}$ 
and by the injectivity statements in (i) and (ii) of Lemma \ref{lem:volume}.
\end{proof}




\section{Intersection Numbers}\label{SectionIntersection}

Let $F$ be an algebraically closed subfield of $\IC$. Let $S$ be a
smooth irreducible curve over $F$ and let $ \pi\colon \mathcal{A} \rightarrow S$ be an abelian scheme over $F$ of relative dimension $g\ge 1$.  
In this section we abbreviate $\IP_F^m$  by $\IP^m$ for integers
$m\ge 1$. 

We will use the basic setup introduced
in \S \ref{SubsectionEmbeddingAbSch}. In particular
$\cA \subset \IP^M \times \IP^m$ is an admissible immersion.

\begin{proposition}\label{PropositionIntersectionNumberPolynomialGrowth}
Suppose  $X$ is an irreducible, closed subvariety  of 
$\mathcal{A}$ defined over $F$ that dominates $S$
and is not generically special. Then there exist
 \begin{itemize}
\item  a constant $c>0$,
 \item a finite and \'etale covering $S'\rightarrow S$ where
 $S'$ is an irreducible curve over $F$,
\item  and finitely many
 closed (not necessarily irreducible) subvarieties $Y_1,\ldots,Y_R$  of $\cA' = \cA \times_S
 S'$
\end{itemize}
  such that the following holds for $X' = X \times_S S'$. For each integer  $N\ge
c^{-1}$, there exists $Y \in \{Y_1,\ldots,Y_R\}$ such that
 $X' \cap [N]^{-1}(Y)$ contains at least $r \ge c N^{2\dim X}$
 irreducible components of dimension $0$.
\end{proposition}

Note that $X'$ from the proposition is a closed subvariety of the
abelian $\cA'$ scheme. It may not be irreducible, but it is
equidimensional of dimension $\dim X$ since
$S'\rightarrow S$ is finite and \'etale. Note also that each irreducible component of $X' \cap [N]^{-1}(Y)$ consists of one $F$-rational point as $X'$ and $[N]^{-1}(Y)$ are defined over $F$.

We will prove Proposition~\ref{PropositionIntersectionNumberPolynomialGrowth} in the next few subsections. 

\subsection{Constructing a Covering $S'\rightarrow S$}
\label{ssec:chooseSprimel}

Further down we will need to pass to a finite and \'etale covering $S'$
of $S$. In this subsection we make some preparations and
mention some facts. 


We recall our convention $F \subset \IC$ and fix $P \in X(F)$  as in Lemma~\ref{LemmaNonDegeneratePoint}.

By assumption on $X$ and $P$, we have an irreducible closed subvariety
$Z \subset \mathcal{A}$ defined over $F$ satisfying the conclusion
of Proposition~\ref{PropositionConstructionOfAuxiliarySubvariety}. In
particular, $\dim Z = \codim_\cA X = n$. 

We fix a  prime number $\ell$ satisfying
\begin{equation}
\label{eq:chooseell}
  \ell > D^{2^{g+1}(g+1)}
\end{equation}
where $D$ comes from (iv) of
Proposition~\ref{PropositionConstructionOfAuxiliarySubvariety}. Later
on, we will impose a second lower bound on $\ell$. 

There is a finite \'etale covering  $S'\rightarrow S$
such that $\cA'/S'$ admits all the $\ell^{2g}$ torsion sections
$S' \rightarrow \cA'[\ell]$ where $\cA'$ is the abelian scheme
$\cA\times_S S'$ over $S'$. We may assume that $S'$ is irreducible.
The prescribed closed immersion $\cA\rightarrow \IP_S^M$ induces a closed immersion $\cA'\rightarrow\IP_{S'}^M$. 

Observe that the induced morphism $\rho \colon \cA'\rightarrow\cA$ is finite
and \'etale. So the pre-image of any irreducible subvariety
$Y$ of $\cA'$
is equidimensional of dimension $\dim Y$.

Let $Z' = Z\times_S S'$, this is a closed subvariety of $\cA'$ that
may not be irreducible. 
It is equidimensional of dimension $\dim Z$.
A further and crucial observation for our argument is that 
the fibers $Z_t$ and $Z'_{t'}$ are equal if $t'\in S'(F)$ maps to $t\in
S(F)$. So by Proposition
\ref{PropositionConstructionOfAuxiliarySubvariety}(iv)
\begin{equation}\label{EqDegreeAfterCover}
\text{the degree of any } Z'_{t'} \subset \IP^M \text{ is bounded by
}D
\text{ for all $t'\in S'(F)$}
\end{equation}
and by the ``Moreover'' part of Proposition
\ref{PropositionConstructionOfAuxiliarySubvariety}
\begin{equation}
  \label{eq:nocosets1}
  \begin{tabular}{c}
\text{at most finitely many fibers of $Z'\rightarrow S'$ over
  $S'(\IC)$ contain a coset}\\
\text{of positive dimension in the respective fiber of $\cA'\rightarrow S'$. }
  \end{tabular}
\end{equation}

\subsection{Local Parametrization and Lattice Points}

We keep the
notation introduced above
and prove the following intermediate counting result.
\begin{lemma}
\label{lem:prepolygrowth}
Let $X$ be as in Proposition \ref{PropositionIntersectionNumberPolynomialGrowth}.
Then there exist
  \begin{itemize}
\item a constant $c>0$,
\item a prime number $\ell$ satisfying \eqref{eq:chooseell},
\item and a finite \'etale covering $S'\rightarrow S$, with $S'$ irreducible, admitting all
  the $\ell^{2g}$ torsion sections $S' \rightarrow \cA'[\ell]$, 
with $\pi'\colon\cA'\rightarrow S'$ the canonical morphism,
 $X' = X\times_S S'$, and
$Z' = Z\times_S S'$
  \end{itemize}
  such that for all integers $N\ge c^{-1}$ the following holds.
There exist $r\ge c N^{2\dim X}$ pairs $(P'_1,Q'_1),\ldots,(P'_r,Q'_r) \in
  X'(\IC)\times Z'(\IC)$ 
such that the $P'_1,\ldots,P'_r$ are pairwise distinct
with the following properties for all $i\in \{1,\ldots,r\}$: 
  \begin{enumerate}
  \item [(i)]
We have
$\pi'(P'_i) = \pi'(Q'_i)$ and $[\ell N](P'_i)=[\ell](Q'_i)$.
\item[(ii)]The Zariski closed subset $Z'_{\pi'(P'_i)}$ of
  $\cA'_{\pi'(P'_i)}$
does not contain any coset of positive dimension.
\item[(iii)]
If $Y'$ is an irreducible closed subvariety of $\cA'$
such that 
 $Q'_i\in  Y'(\IC)$ and $P'_i$ is not isolated in $X'\cap [\ell
  N]^{-1}([\ell](Y'))$, then there exists an $\ell$-torsion section
$\sigma \colon S'\rightarrow \cA'[\ell]$ with $\sigma\not=0$ and 
 $Q'_i \in Y'(\IC)\cap (Y'-\mathrm{im\,}\sigma)(\IC)$. 
  \end{enumerate}
\end{lemma}
\begin{proof}
We make use of the lattice point counting technique from
$\mathsection$\ref{SectionLatticePoints}. 

By Lemma~\ref{LemmaNonDegeneratePoint} we have that $P$ is a smooth point of $\an{X}$ and $X_{\pi(P)}^{\anE}$.
We have $\dim X=g+1-n$, so we can trivialize the family
$\an{X}\rightarrow \an{S}$ 
in  a
neighborhood of $P$ in $\an{X}$ using a smooth map $\tilde \phi_1$ defined on
$U_0\times U_1$ where $U_0\subset \IR^2$ and
$U_1\subset \IR^{2(g-n)}$ are both open and non-empty. 
We may assume that $\tilde\phi_1(0) = P$.  
After possibly shrinking $U_0$
and $U_1$ we compose
$\tilde\phi_1$ with $\tilde b$, the Betti map
$b \colon \cA_\unitdisc \rightarrow\IT^{2g}$ composed by a local inverse of
$\IR^{2g}\rightarrow\IT^{2g}$. 
This yields  a
smooth map $\phi_1 \colon U_0\times U_1\rightarrow \IR^{2g}$ that produces the Betti
coordinates relative to the local parametrization of $\an{X}$. 

Now $P$ is also a smooth point of $\an{Z}$. 
Recall that $\dim Z = n$. 
The same
construction yields a non-empty and open subset
$U_2\subset \IR^{2(n-1)}$ 
and a smooth map
$\tilde\phi_2 \colon U_0\times U_2\rightarrow \an{Z}$
with $\tilde\phi_2 (0)=P$. 
We restrict if necessary and
write $\phi_2 = \tilde b\circ\tilde\phi_2  \colon U_0\times U_2\rightarrow\IR^{2g}$. The subsets $U_0$, $U_1$ and $U_2$ can be chosen to be bounded.

In this setting  $\phi_1$ parametrizes $\tilde b (U)$ where $U\subset \an{X}$ is a
neighborhood of $P$ in $X$ and $\tilde b$ is a local lift of the Betti map 
to $\IR^{2g}$. Similarly $\phi_2$ parametrizes $\tilde{b}(V)$ where $V$ is a neighborhood of $P$ in $\an{Z}$.

We may assume, after shrinking $U_0,U_1$, and $U_2$  if necessary,
that $|\phi_1|_{C^1}<\infty$ and $|\phi_2|_{C^1}<\infty$. 
By Proposition \ref{PropositionConstructionOfAuxiliarySubvariety}(iii)
the fiber $Z_{\pi(P)}$ contains no positive dimensional cosets in $\cA_{\pi(P)}$. 
By (\ref{eq:nocosets1}) and
 up to shrinking $U_0$ we may assume that
\begin{equation}\label{EquationBadFiberControl}
Z_t \text{ contains no positive dimensional cosets in }\cA_t \text{ for all }t \in \pi(\tilde{\phi}_1(U_0 \times U_1)).
\end{equation}

For any $N\in \mathbb{N}$ and $(w,x,y) \in U_0\times U_1\times U_2$ we
define a map
\[
\psi_N(w,x,y)=N\phi_1(w,x)-\phi_2(w,y) \in\IR^{2g}.
\]
Let $\pi_i\colon U_0\times U_1\times U_2 \rightarrow U_0 \times U_i$
be the natural projection for $i=1,2$ and  $\delta_0(w,x,y)$  as above Lemma~\ref{lem:volume}.

Condition (v) of Proposition~\ref{PropositionConstructionOfAuxiliarySubvariety} implies that $\delta_0(0)\not=0$.
So we can apply
Proposition~\ref{PropositionCountingLatticePreparation}. There exists a bounded open  neighborhood $U^\prime$ of $0$ in $U_0\times U_1\times U_2$ and a
 constant $c\in (0,1]$ with the following property. 
For all integers $N_0\ge c^{-1}$ and $N\ge c^{-1}$ we have that $\phi_1|_{\pi_1(U')}$ and $\psi_N|_{U^\prime}$ are injective and
\begin{equation}\label{EquationManyLatticePoints}
\#\left(\psi_N(U^\prime)\cap N_0^{-1}\mathbb{Z}^{2g}\right)\ge c N^{2+2(g-n)}=c N^{2\dim X}.
\end{equation}

Proposition \ref{PropositionCountingLatticePreparation} allows us to
increase  $N_0$. From now on we fix $N_0$ to be a prime number
$\ell$ that satisfies (\ref{eq:chooseell}) and $\ell \ge c^{-1}$.
As in $\mathsection$\ref{ssec:chooseSprimel} we fix a finite \'etale
covering $S'\rightarrow S$ with $S'$ irreducible such that $\cA':= \cA \times_S S' \rightarrow S'$ admits all the $\ell^{2g}$ torsion sections $S' \rightarrow \cA'[\ell]$.


Recall that  $|\phi_2|_{C^1}<\infty$, so is $\phi_2(U_0\times U_2)$ is
bounded and
\begin{equation}\label{EquationFinitelyManyLatticePointsInTheTwoFactorUniformBound}
 \# \left(\phi_2(U_0\times U_2) - \phi_2(U_0\times U_2)\right) \cap
\ell^{-1}\IZ^{2g} \le C \text{ for some $C$
 independent of $N$.}
\end{equation}

Suppose $(w,x,y)\in U^\prime$ satisfies $\psi_N(w,x,y)\in
\ell^{-1}\IZ^{2g}$. Then $\ell N \phi_1(w,x) -
\ell \phi_2(w,y) \in \IZ^{2g}$. 
For the Betti coordinates   we find on $\cA$ that
\begin{equation*}
[\ell N](\tilde\phi_1(w,x)) =
[\ell](\tilde\phi_2(w,y)) \in [\ell](Z)(\IC).
\end{equation*}
Thus we get a mapping
\begin{equation}\label{EquationInjectiveMapLatticePointsIntersectionPoints}
\begin{array}{ccc}
\psi_N(U^\prime) \cap \ell^{-1} \mathbb{Z}^{2g}\ni \psi_N(w,x,y) &\mapsto & 
(\tilde\phi_1(w,x),\tilde\phi_2(w,y))\in X(\IC)\times Z(\IC).
\end{array}
\end{equation}
The image points are of the form $(P_i,Q_i)$
and lie in the same fiber above $S$ and 
with $[\ell N](P_i) = [\ell](Q_i)$.
By \eqref{EquationManyLatticePoints}
these points arise
from at least $cN^{2\dim X}$ elements
of $\psi_N(U')\cap \ell^{-1}\IZ^{2g}$  for all large $N$.
We claim that up to adjusting $c$ the number of different
$P_i$ is also at least $cN^{2\dim X}$. 

So let $(w,x,y)\in U'$ with
$\psi_N(w,x,y) \in \ell^{-1}\IZ^{2g}$
whose image under
\eqref{EquationInjectiveMapLatticePointsIntersectionPoints} is $(P_i,Q_i)$.
Let  $(P_j,Q_j)$ be a further  pair with $P_i=P_j$
that comes from $\psi_N(w',x',y') \in \ell^{-1} \IZ^{2g}$ where
$(w',x',y') \in U'$. Hence
$(w',x')=(w,x)$ as $\phi_1|_{\pi_1(U')}$ is injective. Thus 
$\phi_2(w',y') -\phi_2(w,y) = \psi_N(w,x,y) - \psi_N(w',x',y') \in \ell^{-1}\IZ^{2g}$. 
By \eqref{EquationFinitelyManyLatticePointsInTheTwoFactorUniformBound}
there are at most $C$ possibilities for $\psi_N(w',x',y')$, when
$(x,y,z)$ is fixed. 
So there are at most $C$ possibilities for $(w',x',y')$ as $\psi_N$ is
injective on $U'$; recall
that $C$ may depend on $\ell$ but not on $N$. 

After omitting pairs with duplicate $P_i$
and replacing $c$ by $c/C$ 
we have found
$(P_i,Q_i)$ with pairwise different $P_i$ for $1\le i\le r$ and
 $r\ge c N^{2\dim X}$. 

Let $\rho \colon  \cA'\rightarrow \cA$ denote the canonical morphism.
We fix lifts 
$P'_i,Q'_i\in \cA'(\IC)$ 
of $P_i,Q_i$ respectively in the same fiber of $\cA'\rightarrow S'$.
So
\begin{equation}
\label{eq:ellBpprimeQprime}
[\ell N](P'_i) = [\ell](Q'_i)\quad\text{for all}\quad i\in \{1,\ldots,r\}.
\end{equation}
This yields claim (i) of the lemma. 

As our points $P_i$ lie above 
points in $\pi(\tilde\phi_1(U_0\times U_1))$ we deduce (ii) from
(\ref{EquationBadFiberControl}). 

It remains to prove part (iii). Let $Y'$ be as in (iii), namely $Q'_i \in Y'(\IC)$ and $P'_i$ is
not isolated in $X'\cap [\ell N]^{-1}([\ell](Y'))$
for some $i\in \{1,\ldots,r\}$.
To simplify notation we write $P'=P'_i$ and $Q'=Q'_i$. 

Then there is a sequence $(P^\alpha)_{\alpha\in\IN}$   of pairwise distinct 
points of $X'(\IC)$ that converges in $\an{X'}$ to $P'$ with
$[\ell N](P^\alpha) \in [\ell](Y')(\IC)$
for all $\alpha\in\IN$. 
We fix
$Q^\alpha \in Y'(\IC)$ with
$[\ell N](P^\alpha) = [\ell](Q^\alpha)$ for all
$\alpha\in\IN$. Thus $\pi'(P^\alpha)=\pi'(Q^\alpha)$ and by
continuity the sequence $[\ell](Q^\alpha)$ converges. 
Since $[\ell]$ induces a proper map $(\cA')^{\anE}\rightarrow (\cA')^{\anE}$
 we may assume, after passing to a subsequence, that the $Q^\alpha$ 
 converge in $(Y')^{\anE}$ to some
$Q'' \in Y'(\IC)$. Taking the limit we
 see by continuity and (\ref{eq:ellBpprimeQprime}) that
$[\ell](Q'') = [\ell N](P') = [\ell](Q')$ and in particular $\pi'(Q'')=\pi'(Q')$. Hence
 \begin{equation*}
 Q'' = Q' + T   \in Y'(\IC)
 \end{equation*}
for some  $T$ that is either trivial or of finite prime order $\ell$
in $\cA'_{\pi'(Q')}(\IC)$.

All $\ell^{2g}$ torsion sections $S'\rightarrow\cA'[\ell]$ exist, so 
there is one $\sigma$ with $\sigma(\pi'(Q')) = T$. 
Hence $Q' \in Y'(\IC)\cap(Y'-\mathrm{im\,}\sigma)(\IC)$. 

To complete the proof it remains to verify
$\sigma\not=0$, \textit{i.e.} $T\not=0$. 
For this we assume the converse and derive a contradiction. 
For $\alpha$ large enough, the sequence
member $\rho(P^{\alpha})$  will be  close enough to $\rho(P')=P_i$ 
as to lie in $\tilde\phi_1(U')$. 
As $Q''= Q'$  the analog statement
 holds for the sequence of $\rho(Q^\alpha)$,
\textit{i.e.} $\rho(Q^\alpha)\in \tilde\phi_2(U')$
for all sufficiently large $\alpha$. 
For $\alpha$ sufficiently large  we may write $\rho(P^\alpha) = \tilde\phi_1(w^\alpha,x^\alpha)$ 
and $\rho(Q^\alpha) = \tilde \phi_2(w^\alpha,y^\alpha)$
with $(w^\alpha,x^\alpha,y^\alpha)\in U'$. The condition 
$[\ell N](P^\alpha) = [\ell](Q^\alpha)$
implies $[\ell N](\rho(P^\alpha)) = [\ell](\rho(Q^\alpha))$ and hence
\begin{equation*}
\psi_N(w^\alpha,x^\alpha,w^\alpha)=N\phi_1(w^\alpha,x^\alpha)-\phi_2(w^\alpha,y^\alpha) \in \ell^{-1}\IZ^{2g}.
\end{equation*}
By continuity, the sequence 
$\psi_N(w^\alpha,x^\alpha,w^\alpha)$ is eventually constant. But
$\psi_N$ is injective on $U^\prime$ by Proposition \ref{PropositionCountingLatticePreparation} and hence
 $(w^\alpha,x^\alpha,w^\alpha)$ is eventually constant. So $\rho(P^\alpha)$
 is eventually constant and, as $\rho$ is finite, $P^\alpha$ attains
 only finitely many values.
 But this contradicts the fact that the
 $P^\alpha$ are pairwise distinct and concludes the proof of (iii).
\end{proof}

\subsection{Induction and Isolated Intersection Points}
Let $X, \ell, \cA' \rightarrow S', X',$ and $Z'$ be as in Lemma~\ref{lem:prepolygrowth}. 
The conclusion of Lemma~\ref{lem:prepolygrowth} is already close to what we are
aiming at in Proposition
~\ref{PropositionIntersectionNumberPolynomialGrowth}. However, we must
first deal with the possibility that most $P'_i$ from the lemma
 are not isolated in $X'\cap [\ell N]^{-1}([\ell](Z'))$; otherwise  we
 could just take $Y_1=[\ell]^{-1}([\ell](Z'))$. 
We will handle this in the current subsection
by introducing additional auxiliary subvarieties derived from $Z'$.

Recall that $D$ was introduced in $\mathsection$\ref{ssec:chooseSprimel}
and ultimately comes from Proposition
\ref{PropositionConstructionOfAuxiliarySubvariety}(iv). 
For brevity we write $\cA'[\ell](S')$ for the group of torsion sections
$S'\rightarrow \cA'[\ell]$ of order dividing $\ell$. 
Recall also that $X'$ is equidimensional of dimension $\dim X= g+1-n$. 

We now describe a procedure to construct a finite set $\Sigma$ 
of auxiliary subvarieties. To be more precise we will construct for
each
$k \in \{0,\ldots,n\}$ a finite set
$\Sigma_k$  with the following  properties:
\begin{enumerate}
  \item[(i)] If $Y'\in \Sigma_k$, then $Y'$ is an irreducible closed
    subvariety of $Z'$ with $\dim Y'\le n-k$.
  \item[(ii)] If $Y'\in \Sigma_k$ and $t\in S'(\IC)$
  such that $Y'_t\not=\emptyset$, 
then $\deg Y'_t \le D^{2^k}$. 
\item[(iii)] If $k\le n-1,$ then for all  $Y'\in \Sigma_{k}$ and all
 $\sigma\in
  \cA'[\ell](S')$ such that 
$Y'\not\subset Y'-\mathrm{im\,}\sigma$, all irreducible components of $Y'\cap
  (Y'-\mathrm{im\,}\sigma)$ are elements of $\Sigma_{k+1}$. 
\end{enumerate}

We define $\Sigma_0$ 
to be the set of irreducible components of $Z'$.  Clearly, (i) is
satisfied as $Z'$ is equidimensional of dimension $\dim Z =n$. Moreover, (ii) is satified due to \eqref{EqDegreeAfterCover}.

We construct the remaining $\Sigma_1,\ldots,\Sigma_n$
and verify the properties inductively.
Suppose $k\in \{0,\ldots,n-1\}$ and that $\Sigma_k$ has already been  
constructed. 

Consider the set of all $Y'\in \Sigma_k$ and  $\sigma \in
\cA'[\ell](S')$
with $Y'\not\subset Y'-\mathrm{im\,}\sigma$. 
There are only finitely many such pairs $(Y',\sigma)$ 
and we take as $\Sigma_{k+1}$
 all irreducible components of all
$Y'\cap (Y'-\mathrm{im\,}\sigma)$ that arise this way.
This choice makes (iii) automatically hold true for all $k\in \{0,\ldots,n-1\}$. 
If  $Y''\in \Sigma_{k+1}$ is such an irreducible component, then $Y''\subsetneq Y'\subset
Z'$ and $\dim Y'' \le \dim Y'
-1 \le n-(k+1)$ by (i) applied to $k$. This implies (i) for $k+1$.

We now verify (ii).
If $Y''$ does not dominate $S'$, then the image of $Y''$ in $S'$ is a
point $t$, hence $Y'' = Y''_t$. In this case $Y''_t$ is an irreducible
component of $Y'_t\cap (Y'_t-\sigma(t))$. 
By B\'ezout's Theorem and since $\deg Y'_t =\deg (Y'_t-\sigma(t))$ we
find $\deg Y''_t \le (\deg Y'_t)^2$. 
By (ii) applied to $Y'_t$ this implies $\deg Y''_t \le (D^{2^k})^2 =
D^{2^{k+1}}$.
So (ii) holds for all $Y''$ that do not dominate $S'$.
If $Y''$ dominates $S'$,
then for all but  at most finitely many $t\in S'(\IC)$
all irreducible components of $Y''_{t}$ are also
irreducible components of $Y'_t\cap (Y'_t-\sigma(t))$. 
For such a $t$ we have, again by B\'ezout's Theorem, $\deg Y''_t \le
(\deg Y'_t)^2 \le D^{2^{k+1}}$ which implies
 (ii). For any remaining $t \in S'(\IC)$,
observe that $Y''$ is flat over $S'$ by
\cite[Proposition III.9.7]{Hartshorne}
since $\dim S'=1$ and $Y''$ is irreducible. As cycles of $\IP^M$ the fibers of $Y''$ are pairwise algebraically equivalent. So $\deg Y''_t \le D^{2^{k+1}}$ for all $t \in S(\IC)$, see 
Fulton \cite[Chapters~10.1 and 10.2]{Fulton} on the
conservation of numbers. More precisely let $H_1,\ldots,H_{\dim Y''-1}$ be generic hyperplane sections of $\IP^M \times S \rightarrow S$ such that $Y'' \cap \bigcap_{i=1}^{\dim Y''-1} H_i$ is flat of relatively dimension $0$ over $S$, and then apply
\cite[Corollary~10.2.2]{Fulton} to, using the notation of \textit{loc.cit.}, $Y = \IP^M$, $T=S$, $\alpha_1 = [Y'']$ and $\alpha_i = [H_{i-1}]$ for $i= 2, \ldots, \dim Y''$.
So we conclude (ii) for all fibers of $Y''$ regardless whether it
dominates $S'$ or not. 

We define
\begin{equation*}
  \Sigma = \Sigma_0 \cup \cdots \cup \Sigma_{n}. 
\end{equation*}

It is crucial for us that the following bound involving $Y'\in \Sigma$ is independent of $\ell$: 
\begin{equation}
\label{eq:degYtuniformub}
\deg Y'_t\le D^{2^n}\le D^{2^{g+1}}\text{ for all }t\in S'(\IC)\text{
with } Y'_t\not=\emptyset.
\end{equation}

We are now ready to prove the main result of  this section.

\begin{proof}[Proof of Proposition \ref{PropositionIntersectionNumberPolynomialGrowth}]
Let $X,c, \ell, \cA' \rightarrow S', X',$ and $Z'$ be as in Lemma~\ref{lem:prepolygrowth}. 
We will prove that $\{[\ell]^{-1}([\ell](Y')): Y' \in \Sigma\}$ is the desired set of closed subvarieties of $\cA'$. 

For  $N\ge c^{-1}$, Lemma~\ref{lem:prepolygrowth}
produces $r\ge cN^{2\dim X}$ 
pairs $(P'_1,Q'_1),\ldots,(P'_r,Q'_r) \in X'(\IC)\times Z'(\IC)$ with the stated properties. 

Note that each $Q'_i$ is a point of some element of $\Sigma$. Indeed, it
is a point of $Z'$ whose irreducible components are in $\Sigma_0$.  
To each $i\in \{1,\ldots,r\}$ we assign an auxiliary variety in
$\Sigma_k$ containing $Q'_i$ and with  maximal $k$.

By the Pigeonhole Principle there exist $k$ and an auxiliary variety $Y'\in
\Sigma_k$ that is hit  at least $r/\#\Sigma$ times.
As $\#\Sigma$ is independent of $N$ we may assume, after adjusting
$c$, that
$r\ge cN^{2\dim X}$ and 
$Q'_i \in Y'(\IC)$ for all $i\in \{1,\ldots,r\}$.

Let $Y=[\ell]^{-1}([\ell](Y'))$. We prove that $Y$ is what we want, \textit{i.e.} $X'\cap[N]^{-1}(Y)$ contains at least $r \ge cN^{2\dim X}$ isolated points.


Observe that   $P'_1,\ldots,P'_r$ are points of $X'\cap [N]^{-1}(Y)$. If they are
all isolated in this intersection
then the proposition follows as they are 
pairwise distinct. 

We  assume that some  $P'_i$
is not isolated in $X' \cap [N]^{-1}(Y)$ and will arrive at a contradiction.  
By Lemma \ref{lem:prepolygrowth}(iii)
there exists a non-trivial $\sigma \in \cA'[\ell](S')$ such that
$Y'\cap(Y'-\mathrm{im\,}\sigma)$ contains $Q'_i$. In particular, $Y'$ cannot be a point.

Let us assume $Y'\not\subset Y'-\mathrm{im\,}\sigma$ for now. Thus
properties~(i) and (iii) of $\Sigma_k$ listed above imply 
$1\le \dim Y' \le n-k$ and
 that $Q'_i$ lies in
 an element of $\Sigma_{k+1}$. This contradicts
 the maximality of $k$. 
%
%
Hence $Y'\subset Y'-\mathrm{im\,}\sigma$ and  in particular $Y'_t \subset Y'_t-\sigma(t)$ where $t$
is the image of $Q'_i$ under
$\cA'\rightarrow S'$. 

If $Y'$ dominates $S'$, then $Y'_t$ is equidimensional of dimension $\dim
Y'-1$. If $Y'$ does not dominate $S'$, then $Y'=Y'_t$ is irreducible and in
particular equidimensional. 
In both cases we find $ Y'_t = Y'_t-\sigma(t)$ and the group generated by
$\sigma(t)$ acts on the set of irreducible components of $Y'_t$. 

The number of irreducible components of $Y'_t$ is at most $\deg Y'_t \le D^{2^{g+1}}$
by (\ref{eq:degYtuniformub}). Furthermore,
 $\sigma(t)$ has precise order $\ell$ since $\sigma\not=0$ and $\ell$ is prime. By
 (\ref{eq:chooseell}) we have
$\ell > D^{2^{g+1}}$. As $\ell$ is a prime there is no  non-trivial group homomorphism from
$\IZ/\ell\IZ$ to the symmetric group on $\deg Y'_t$ symbols. 
We conclude that $\sigma(t) + W = W$ for all irreducible components
$W$ of $Y'_t$. 

The stabilizer  $\stab{W}$  in $\cA'_t$ of any irreducible component of $Y'_t$ of $W$, 
satisfies 
$\deg\stab{W}\le \deg(W)^{\dim W + 1}\le \deg(W)^{g+1}$  by \cite[Lemme~2.1(ii)]{DavidHindry}.
We obtain  $\deg W \le \deg Y'_t \le D^{2^{g+1}}$ again
from (\ref{eq:degYtuniformub}).
Putting these bounds together gives $\deg \stab{W}\le 
D^{2^{g+1}(g+1)}$.
But $\stab{W}$ contains $\sigma(t)$, a point of order
 $\ell  > \deg \stab{W}$ by (\ref{eq:chooseell}). 
In particular, $\stab{W}$ has positive dimension. But this implies
that $Y'_t$ contains a positive dimensional coset. 
By property (i) in the construction of
$\Sigma$ we have $Y'_t \subset Z'_t$ and therefore $Z'_t$ contains a
coset of positive dimension. This contradicts Lemma~\ref{lem:prepolygrowth}(ii). 
\end{proof}


\section{Height Inequality in the Total Space}\label{SectionHtTotalSpace}
In this section, and if not stated otherwise, we work with the category of  schemes over
$\IQbar$ and  abbreviate $\IP_{\IQbar}^m$  by $\IP^m$ for integers
$m\ge 1$. Let $S$ be a smooth irreducible curve defined over
$\overline{\mathbb{Q}}$. Let $\pi \colon \mathcal{A} \rightarrow S$ be an abelian scheme over $\overline{\mathbb{Q}}$ of relative dimension $g\ge 1$.

We will use the basic setup introduced
in \S \ref{SubsectionEmbeddingAbSch}. In particular
$\cA \subset \IP^M \times \IP^m$ is an admissible immersion.

Our principal result is the following proposition. It makes use of the
naive height given by (\ref{EqHeightTotal}).

\begin{proposition}\label{PropositionHeightChangeUnderScalarMultiplication}
Suppose $X$ is an irreducible closed
subvariety
of $\cA$ that dominates $S$ and is not generically special, then there exists a constant $c>0$
depending on $X$ and the data introduced above with the following
property. 
For any integer $N\ge c^{-1}$  there exist a non-empty Zariski open
subset
$U_N \subset X$ and a constant $c^\prime(N)$, both of which depend on $N$, such that
\[
h([2^N]Q) \ge c 4^N h(Q) - c^\prime(N)
\quad\text{for all  $Q\in U_N(\overline{\mathbb{Q}})$.}
\]
\end{proposition}


\subsection{Polynomials Defining  Multiplication-by-$2$ on $\cA$}\label{SubsectionScalarMultiplication}

Let $\underline{X} = [X_0:\cdots:X_M]$ denote the projective coordinates on $\IP^M$ and let $\underline{Y} = [Y_0:\cdots:Y_m]$ denote the projective coordinates on $\IP^m$. 
Recall condition (ii) of the admissible setting $\cA \subset \IP^M \times \IP^m$: the morphism
$[2]$ is represented globally on $\cA\subset\IP^M\times\IP^m$ by $M+1$ bi-homogeneous polynomials, homogeneous of degree $4$ in $\underline{X}$ and homogeneous of a certain degree, say $c_0$, in $\underline{Y}$. 
In other words, there exist
$G_0,\ldots,G_M \in \overline{\IQ}[\underline{X}; \underline{Y}]$,
each $G_i$ being homogeneous of degree $4$ on the variables
$\underline{X}$ and homogeneous of degree $c_0$ on the variables $\underline{Y}$, such that $[2](a) = ([G_0(a_1;a_2):\cdots:G_M(a_1;a_2)];a_2)$ for any $a \in \cA(\IC)$. Here we write $a = (a_1;a_2)$ under $\cA \subset \IP^M \times \IP^m$. Note that $c_0$ depends only on the immersion $\cA \subset \IP^M \times \IP^m$.


\subsection{An Auxiliary Rational Map}



\begin{lemma}\label{PropositionFromZToPolynomials}
Let $X$ and $Y$ be  locally closed algebraic subsets 
of $\IP^{M}$ and suppose that $X$ is irreducible. There exists
$\delta\in\IN$ that depends only on $Y$ with the following property.
Suppose  $r\ge 1$ and 
$Q_1,\ldots,Q_r\in X(\IQbar)\cap Y(\IQbar)$ 
for all $i\in \{1,\ldots,r\}$. 
There exist homogeneous polynomials
$\varphi_0,\varphi_1,\ldots,\varphi_{\dim
X}\in \IQbar[X_0,\ldots,X_{M}]$ of degree $\delta$, whose set of common zeros is denoted by $Z\subset \IP^M$,
such that the rational map
$\varphi = [\varphi_0: \cdots:\varphi_{\dim X}] \colon \IP^{M} \dashrightarrow \IP^{\dim X}$ satisfies: 
\begin{enumerate}
\item [(i)] We have $\varphi(Y\backslash Z) = [1:0:\cdots : 0]$ and
$Q_i 
\notin Z(\IQbar)$ for all $i\in \{1,\ldots,r\}$.
\item [(ii)] 
If $C$ is an irreducible subvariety of $X$ and $i\in \{1,\ldots,r\}$ with $Q_i \in C(\IQbar)$ such
that
$\varphi|_{C\backslash Z}$ is constant, then $C\subset \overline
Y$,
where $\overline Y$ is the Zariski closure of $Y$ in $\IP^{M}$. 
  \end{enumerate}
  \end{lemma}

\begin{proof}
The Zariski closure $\overline Y$ of $Y$ in $\IP^M$ is
 the zero set of finitely many homogeneous polynomials $g_1,\ldots,g_m\in 
\IQbar[X_0,\ldots,X_{M}]$. We may assume that $g_1,\ldots, g_m$ all have
the same degree $\delta$. Note that $\delta$ depends only on $Y$. 

We may fix further $g_0\in \IQbar[X_0,\ldots,X_{M}]$, also of degree $\delta$,
such that $g_0(Q_i)\not=0$ for all $i\in\{1,\ldots,r\}$. For example, we can
take $g_0$ to be the $\delta$-th power of a linear polynomial whose
zero set avoids the $Q_i$'s. 
The set of common zeros $Z_G$ of all $g_i$ does not contain any $Q_i$
and thus not all of $Y$. 

We obtain a rational map $G
=[g_0:\cdots:g_m] \colon \IP^{M} \dashrightarrow \IP^m$. 
Observe that $G(Y\backslash Z_G)=[1:0:\cdots:0]$.
Observe also that $G(X\backslash Z_G)$ is constructible in $\IP^m$ 
and its Zariski closure is of dimension at most $\dim X$. 
This image also contains $[1:0:\cdots:0]$, 
so there exist homogenous linear forms $l_1,\ldots,l_{\dim X} \in
\IQbar[X_0,\ldots,X_{m}]$
such that 
\begin{equation}\label{EquationOriginIsolated}
[1:0:\cdots:0] \text{ is isolated in } \scrZ(l_1,\ldots,l_{\dim
X})\cap G(X\backslash Z_G).
\end{equation}
We set $l_0 = X_0$ and consider $[l_0:\cdots:l_{\dim X}]$
as a rational map $\IP^{m} \dashrightarrow\IP^{\dim X}$.
It is well-defined at $[1:0:\cdots:0]\in \IP^{m}(\IQbar)$ and maps this point to
$[1:0:\cdots:0]\in \IP^{\dim X}(\IQbar)$.

We set $\varphi_0 = l_0(g_0,\ldots,g_{m}),\ldots, \varphi_{\dim X} = l_{\dim X}(g_0,\ldots,g_{m})$.
Then
 $\varphi_i$ is homogeneous of degree $\delta$ for all
 $i\in \{0,1,\ldots,\dim X\}$. Let $Z$ be the set of common zeros
of the $\varphi_i$. 
Then $Q_i\not\in Z(\IQbar)$ for all $i$ by construction. 
The rational map $\varphi=[\varphi_0 : \cdots
:\varphi_{\dim X}] \colon \IP^{M}\dashrightarrow\IP^{\dim X}$
is defined on  $Q_i$ and maps
$Y\backslash Z$ to $[1:0:\cdots:0]$. We conclude (i). 

Let $C$ be as in claim (ii). Then $C\backslash Z$ is non-empty and
is mapped to $[1:0:\cdots:0]\in \IP^{\dim X}(\IQbar)$ under $\varphi$. 
So $l_1,\ldots,l_{\dim X}$ vanish on $G(C\backslash Z)$. 
As the image $G(C\backslash Z)\subset G(C\backslash Z_G)$
 contains $G(Q_i) = [1:0:\cdots:0]\in\IP^m(\IQbar)$ we infer from
 \eqref{EquationOriginIsolated} 
that the $g_1,\ldots,g_m$ vanish on $C\backslash Z$. 
Hence $C\backslash Z\subset \overline Y$ and so $C\subset \overline Y$
since $C\backslash Z$ is Zariski dense in the irreducible 
$C$. 
\end{proof}

The map $\varphi$ depends on the collection of $Q_i$ in the previous
proposition, but the degree $\delta$ does not. Also note that the $Q_i$'s are not necessarily pairwise distinct.

\subsection{Height Change under Scalar Multiplication}
The following lemmas are proven by the second-named author in \cite{Hab:Special}. Lemma~\ref{LemmaChangeOfHeightUnderRationalMap} is our main tool to deduce the desired height inequality (Proposition~\ref{PropositionHeightChangeUnderScalarMultiplication}) from the ``division intersection points'' counting (Proposition~\ref{PropositionIntersectionNumberPolynomialGrowth}).
\begin{lemma}\label{LemmaDegreeAndPreimage}
Let $X$ be an irreducible variety over $\mathbb{C}$ 
 and let $\varphi\colon
 X \dashrightarrow  \mathbb{P}_{\mathbb{C}}^{\dim X}$ be a rational
 map. Then for any $Q\in \mathbb{P}^{\dim X}(\mathbb{C})$, the number of zero-dimensional irreducible components of $\varphi^{-1}(Q)$ is at most $\deg \varphi$. By convention we say that $\deg \varphi=0$ if $\varphi$ is not dominant.
\begin{proof}
This is \cite[Lemma~4.2]{Hab:Special}. The crucial point is that
$\mathbb{P}_{\mathbb{C}}^{\dim X}$ is a normal variety.
\end{proof}
\end{lemma}

\begin{lemma}\label{LemmaChangeOfHeightUnderRationalMap}
Let $X\subset \mathbb{P}^M$ be an irreducible
closed subvariety over $\overline{\mathbb{Q}}$ of positive dimension.
Let $\varphi \colon
X\dashrightarrow  \mathbb{P}^{\dim X}$ be the
rational map given by $\varphi=[\varphi_0:\cdots:\varphi_{\dim X}]$
where $\varphi_i$ are homogeneous polynomials with coefficients in
$\overline{\mathbb{Q}}$ that are not all identically zero on $X$ and
have equal degree at most $D\ge 1$. Then there exist a constant
$c=c(X,\varphi)$ and a 
 Zariski open dense subset $U$ of  $X$ such that
 $\varphi_0,\ldots,\varphi_{\dim X}$ have no common zeros on $U$ and
\[
h(\varphi(P))\ge \frac{1}{4^{\dim X}\deg(X)}\frac{\deg \varphi}{D^{\dim X-1}}h(P)-c
\]
for any $P\in U(\overline{\mathbb{Q}})$.
\begin{proof}
This is \cite[Lemma~4.3]{Hab:Special}.
\end{proof}
\end{lemma}

Now we are ready to prove Proposition~\ref{PropositionHeightChangeUnderScalarMultiplication}. 


To $X$, recall that we associated in
Proposition~\ref{PropositionIntersectionNumberPolynomialGrowth} a
finite \'{e}tale covering $S' \rightarrow S$. 
Set $X' =X\times_S S'$. 
Then $X'$ is a closed subvariety of $\cA'=\cA\times_S S'$ 
and equidimensional of dimension $\dim X$. Let
$\rho\colon \cA'\rightarrow \cA$ denote the natural projection, it is
finite and \'etale.
Let $\overline S$ be the Zariski closure of $S$ in $\IP^m$ .
We  fix a smooth projective curve $\overline{S'}$
that contains $S'$ as a Zariski open subset, then
$S'\rightarrow S$ extends to a morphism
$\overline{S'}\rightarrow\overline{S}$.
Some positive power of the pullback of $\mathcal{O}(1)$ under
$\overline{S'}\rightarrow\overline{S}\rightarrow\IP^m$ yields a
closed immersion $\overline{S'}\rightarrow\IP^{m'}$ for some $m' \in \IN$.
The pullback of the closed immersion $\cA\rightarrow \IP^M\times S$
yields
a closed immersion $\cA'\rightarrow\IP^M\times S'$
and thus an immersion $\cA'\rightarrow\IP^M\times \IP^{m'}$. 

We recall that $[2]$ 
on $\cA\subset\IP^M\times\IP^m$ 
is presented globally 
by bihomogeneous polynomials
$G_0,\ldots,G_M$ on $\cA$ described in 
\S \ref{SubsectionScalarMultiplication}. 
The morphism
$\overline{S'}\rightarrow \overline{S}\rightarrow\IP^{m}$
is defined Zariski locally on $\overline{S'}$
by an $(m+1)$-tuple of homogeneous polynomials
in $m'+1$ variables. 
In other words there is a finite open cover $\{S'_{\alpha}\}_{\alpha=1}^{n_1}$
 of $S'$ such that
$\overline{S'}\rightarrow\overline{S}$ is represented
on each $S'_\alpha$
 by a tuple $F_\alpha$ of homogeneous polynomials
of equal degree and no common zero on  $S'_\alpha$. 
Above each $S'_\alpha$ the morphism $[2]$
is defined by $[2](a_1,a_2) =
 ([G_0(a_1,F_\alpha(a_2)):\cdots:G_M(a_1,F_\alpha(a_2))],a_2)$;
here
$a=(a_1,a_2)\in \cA'(\IQbar) \subset \IP^M(\IQbar) \times \IP^{m'}(\IQbar)$
lies above
${S'_{\alpha}}$.
Iterating $[2]$ we find that for all integers $N\ge 1$ and above
each $S'_\alpha$ the morphism $[2^N]$ 
is defined by bihomogeneous polynomials with degree in $a_1$ equal to
$4^N$ and degree in $a_2$ at most $c_1 (4^N-1)/3$; here $c_1 = c_0\deg
F_\alpha$ and $c_0$ is as in \S \ref{SubsectionScalarMultiplication}. 
As for several constants below, $c_1$ may depend on $\cA'$ and $X'$,
but not on $N$.

Let us embed $\cA'$  in $\IP^{(M+1)(m'+1)-1}$ 
by composing the immersion 
$\cA'\rightarrow\IP^M\times\IP^{m'}$ with 
 the
Segre morphism $\IP^M\times\IP^{m'}\rightarrow \IP^{(M+1)(m'+1)-1}$.
After locally inverting the Segre morphism and increasing $n_1$ we obtain
to an open cover 
 $\{V_\alpha\}_{\alpha=1}^{n_1}$ of $\cA'$,
a refinement of
 $\{\cA'|_{S'_\alpha}\}_\alpha$,
such that $[2^N] |_{V_\alpha} \colon V_\alpha \rightarrow \cA'$ is represented by
a tuple of homogeneous polynomials of degree at most $c_2 4^N$
on each $V_\alpha$. Here $n_1$ and $c_2$ are independent of $N$.

 For any irreducible
 component $X'_0$ of $X'$ the restriction $\rho|_{X'_0} \colon X'_0 \rightarrow
 X$ is dominant and $\dim X'_0 = \dim X$. 
So Silverman's height inequality
\cite{Silverman:heightest11} applies; here
we could have also used the Height Machine and Lemma \ref{LemmaChangeOfHeightUnderRationalMap}. 
To prove the proposition, it suffices to find a constant $c>0$
that is independent of $N$ with the
following property. For any integer $N \ge c^{-1}$ there exist an
irreducible component $X'_0$ of $X'$, a non-empty Zariski open subset
$U'_N \subset X'_0,$ and a constant $c'(N,X'_0)$ such that
$h([2^N]Q) \ge c 4^N h(Q) -c'(N,X'_0)$ for all $Q \in U'_N(\IQbar)$.
Then we can take $U_N$ to be a non-empty Zariski open subset of $X$ 
with $U_N\subset \rho(U'_N)$
and $c'(N) = c'(N,X'_0)$.

Let $c_3>0$ be the $c$ in
 Proposition~\ref{PropositionIntersectionNumberPolynomialGrowth}  and
 let $Y_1,\ldots, Y_R$ be the subvarieties of
  $\cA'$ therein. The constant $c_3$  and the varieties
 $Y_1,\ldots,Y_R$ will depend on 
 on
 $X$ and $\mathcal{A}$, but not on $N$.
We work with $2^N$ instead of $N$ in 
Proposition~\ref{PropositionIntersectionNumberPolynomialGrowth}.
So for any sufficienty large (but fixed) $N$ we 
let $P_1,\ldots,P_r \in
X'(\overline{\mathbb{Q}})$ with $r\ge c_3 4^{N\dim X}$ be pairwise distinct points as in
Proposition~\ref{PropositionIntersectionNumberPolynomialGrowth} 
and $Y\in \{Y_1,\ldots,Y_R\}$ such that $[2^N](P_i)\in Y(\IQbar)$
and $P_i$ is isolated in $X'\cap [2^N]^{-1}(Y)$ for all $i$. 

Suppose $X'$ has $n_2$ irreducible components, then $n_2$ is
independent on $N$. 
We apply the Pigeonhole Principle
to find $\alpha\in \{1,\ldots,n_1\}$
and
some irreducible component of $X'$ such that
at least $c_3 4^{N\dim X}/(n_1n_2)$ points $P_i$ lie on $V_\alpha$ and
this component. Replace $X'$ by the said component 
and replace $c_3$ by $c_3/(n_1n_2)$. 
Now we may assume that there is a tuple of
 homogeneous polynomials of equal degree at most $c_2 4^N$  that  
define $[2^N]$ on a Zarski open  subset of
 $X'\subset \cA'\subset \IP^{(M+1)(m'+1)-1}$ that contains all the $P_i$'s.


We apply
Lemma~\ref{PropositionFromZToPolynomials} to
$[2^N](X')\subset \cA' \subset \IP^{(M+1)(m'+1)-1}$,
$Y$, and the points $[2^N](P_1),\ldots,[2^N](P_r)$. 
 Thus we obtain a rational map $\varphi \colon \IP^{(M+1)(m'+1)-1}
\dashrightarrow \IP^{\dim X'}$
defined at all $[2^N](P_i)$ 
that arises from homogenous polynomials
of equal degree $\delta_Y$. Observe that $\delta_Y$ depends only on
$Y\in  \{Y_1,\ldots,Y_R\}$. Let $\delta = \max_{Y\in  \{Y_1,\ldots,Y_R\}} \delta_Y$. Now that $\{Y_1,\ldots,Y_R\}$ is fixed as $N$ varies, this does not
endanger our application. There exists a constant $c_4(\varphi)\ge 0$ such that
\begin{equation}\label{EquationInequalityHeightUltrametric}
h(\varphi(Q))\le \delta h(Q) + c_4(\varphi)
\end{equation}
for any 
$Q$
outside the set of common zeros of the polynomials involved in $\varphi$,
see \cite[Theorem~B.2.5.(a)]{DG2000}. To emphasize that $\varphi$ 
may depend on $N$ we write $c_4(N)$ for $c_4(\varphi)$.

For $N$ as before we define
$\varphi^{(N)}=\varphi  \circ [2^N] \colon
X' \dashrightarrow \IP^{\dim X'}$. Then by
Lemma~\ref{PropositionFromZToPolynomials}(i), 
each $P_i$ is mapped via $\varphi^{(N)}$ to $[1:0:\cdots:0]$.

We would like to invoke Lemma~\ref{LemmaDegreeAndPreimage} to bound 
$\deg \varphi^{(N)}$ from below by $r$. To do this we must verify that
each $P_i$ is isolated in the  fiber of $\varphi^{(N)}$ above $[1:0:\cdots:0]$.
Let us suppose $C\subset X'$ is irreducible, contains $P_i$, and is
inside a fiber of $\varphi^{(N)}$. Apart from finitely many points, $[2^N](C)$ is in a
fiber of $\varphi$. 
Now we apply Lemma~\ref{PropositionFromZToPolynomials}(ii)
to conclude that $[2^N](C)$ is contained in the Zariski closure
of $Y$ inside $\IP^{(M+1)(m'+1)-1}$. 
But $Y\subset\cA'$, so
$C\subset [2^N]^{-1}(Y)$. Now Proposition~\ref{PropositionIntersectionNumberPolynomialGrowth} implies
$C = \{P_i\}$.

This settles our claim that $P_i$ is isolated in the fiber of
$\varphi^{(N)}$ and we conclude
$\deg \varphi^{(N)} \ge r \ge c_3 4^{N\dim X}$.

Recall that 
there is a tuple of
 homogeneous polynomials of equal degree at most $c_2 4^N$  that  
define $[2^N]$ on a Zarski open  subset of
 $X'\subset \cA'\subset \IP^{(M+1)(m'+1)-1}$ that contains all $P_i$. 
So we can describe 
$\varphi^{(N)}$ on this subset of $X'$ using polynomials 
of degree at most $c_2 \delta 4^N$. 
We apply
Lemma~\ref{LemmaChangeOfHeightUnderRationalMap} to
$\varphi^{(N)}$ and the Zariski closure of $X'$ in 
$\IP^{(M+1)(m'+1)-1}$ to 
 conclude that  there exist a constant $c_5>0$, independent of $N$,
 and a constant
$c_6(N)\ge 0$ which may depend on $N$, such that 
\begin{equation}\label{EquationAfterApplyingBothLemmas}
h(\varphi^{(N)}(P)) \ge c_5 4^N h(P) - c_6(N)
\end{equation}
for all $P\in U'_N(\overline{\mathbb{Q}})$ where $U'_N$ is non-empty 
Zariski open in $X'$ and may depend on $N$.

Now by letting $Q=[2^N]P$ and dividing by $\delta$ 
in \eqref{EquationInequalityHeightUltrametric}, we get
by \eqref{EquationAfterApplyingBothLemmas} that
\[
h([2^N]P)\ge \frac{c_5}{\delta} 4^N h(P) - c_7(N)
\]
for all $P\in U'_N(\overline{\mathbb{Q}})$ after possibly shrinking $U'_N$. The proof is complete as
$c_5/\delta$ is independent of $N$.
\qed


\section{N\'{e}ron--Tate Height and Height on the Base}\label{SectionNTBase}
The goal of this section is to prove
Theorem~\ref{TheoremMainResultHeightBound} and the slightly stronger  
Theorem~\ref{TheoremMainResultHeightBoundStronger}.

We will use the basic setup introduced in \S \ref{SubsectionEmbeddingAbSch}.
Thus $S$ is a smooth, irreducible curve over $\IQbar$,
$\pi\colon\cA\rightarrow S$ is an abelian scheme of relative dimension $\ge
1$, and $ \cA \subset \IP_\IQbar^M \times \IP_\IQbar^m$ is
an admissible immersion. All varieties in this sections are defined
over $\IQbar$.

\subsection{Auxiliary Proposition}

We prove the following proposition; recall that both heights below are
defined as in \S\ref{SubsectionEmbeddingAbSch}. 
\begin{proposition}\label{PropositionBeforeMainTheorem}
Assume $X$ is an irreducible closed subvariety of $\mathcal{A}$ 
 that is not generically special. 
Then there exist a non-empty Zariski
open subset $U\subset X$ defined over $\IQbar$ and a
constant $c>0$ depending only on $\mathcal{A}/S$, $X$,
and the admissible immersion  such that
\begin{equation}\label{EquationAuxiliaryEquation}
h(P) \le c \left(1+\hat{h}_{\cA}(P)\right)
\end{equation}
for all $P \in U(\IQbar)$. 
\begin{proof}

By the  Theorem of
Silverman-Tate \cite[Theorem~A]{Silverman}, there exist a  constant
$c_1\ge 0$ 
 such that
\begin{equation}
\label{EquationHeight2}
|\hat{h}_{\mathcal{A}}(P)-h(P)| \le
c_1\bigl(1+h(\pi(P))\bigr)
\end{equation}
for all $P\in \mathcal{A}(\IQbar)$; observe that this proof also holds
without the smoothness assumption when working with line bundles
instead of Weil divisors. 

To prove the proposition we may thus assume that $\pi$ is non-constant
on $X$. Therefore, $X$ dominates $S$. 

Since $X$ is not generically special, we can apply
Proposition~\ref{PropositionHeightChangeUnderScalarMultiplication} to
$X$. 
 There exists a constant $c_2>0$, depending on $X,\cA,$ and its
 admissible immersion, such
 that the following holds. For any integer $N\ge c_2^{-1}$, there
 exist a Zariski open dense subset $U_N\subset X$ and a constant
 $c_3(N)\ge 0$ such that
\begin{equation}\label{EquationHeight1}
h([2^N]P) \ge c_2 4^N h(P) - c_3(N)
\end{equation}
for all $P\in U_N(\IQbar)$; we stress that $U_N$ and 
 $c_3(N)\ge 0$ may depend  on $N$ in addition to $X,\cA,$ and the immersion.

Now for any integer $N \ge c_2^{-1}$ and any $P\in U_N(\IQbar)$, we have
\begin{align*}
\hat{h}_{\cA}([2^N](P)) &\ge h([2^N](P))
-c_1\left(1+h(\pi([2^N](P)))\right) \quad
\text{by (\ref{EquationHeight2})}\\
&= h([2^N](P))
-c_1\left(1+h(\pi(P))\right) \\
&\ge c_2 4^{N} h(P) -c_3(N) -c_1\left(1+h(P)\right)
\text{by (\ref{EquationHeight1}) and $h(\pi(P))\le h(P)$ \eqref{EqHeightTotal}.}
\end{align*}
But $\hat{h}_{\mathcal{A}}([2^N]P)=4^N \hat{h}_{\mathcal{A}}(P)$
and dividing by $4^N$ yields 
\begin{equation*}
\hat{h}_{\cA}(P) \ge \left(c_2 - \frac{c_1}{4^N} \right) h(P) - \frac{c_3(N)+ c_1}{4^N}
\end{equation*}
for all $N\ge c_2^{-1}$ and all $P \in U_N(\IQbar)$. 

Recall that $c_2$ and $c_3$ are independent of $N$. We fix $N$ to be the least
 integer such that $4^{N}\ge 2c_1/c_2$ and
$N\ge c_2^{-1}$. Then 
\[
\hat{h}_{\cA}(P) \ge \frac{c_2}{2} h(P) - \frac{c_3(N)+ c_1}{4^N}
\]
for all
$P \in U_N(\IQbar)$. Since $N$ is fixed now, the Zariski open dense
subset $U_N$ of $X$ is also fixed. 
For an appropriate $c>0$, depending on $N, c_1,c_2,$ and $c_3(N)$ we
conclude 
(\ref{EquationAuxiliaryEquation}). 
\end{proof}
\end{proposition}

\subsection{Proof of Theorem~\ref{TheoremMainResultHeightBound}}
The first inequality in (\ref{eq:TheoremMainResultHeightBound}) follows from the definition
of $h_{\cA,\iota}(\cdot)$ and as the absolute logarithmic Weil height
is non-negative. 
To prove the second inequality we 
may assume, by properties of the height machine, that the two height
functions appearing in the conclusion arise from an admissible immersion. 
We do an induction
on $\dim X$. When $\dim X=0$, this result is trivial. So let us assume
$\dim X\ge 1$.

If  $X$ is generically special then $X^*=\emptyset$ and there is nothing
to show. Otherwise we may apply
 Proposition~\ref{PropositionBeforeMainTheorem} and so the inequality \eqref{EquationAuxiliaryEquation}
 holds for any $x\in (X\setminus Z)(\IQbar)$ for some
 proper closed subvariety $Z$ of $X$ defined over $\IQbar$. Let
 $Z=Z_1\cup \cdots \cup Z_r$ be the decomposition into irreducible
 components. Since $\dim Z_i \le \dim X-1$ we may do induction on the
 dimension. By the induction hypothesis, the inequality \eqref{EquationAuxiliaryEquation} holds for all points in $Z_1^*(\IQbar)\cup \cdots \cup Z_r^*(\IQbar)$. Therefore the inequality \eqref{EquationAuxiliaryEquation} holds for all points in $(X\setminus Z)(\IQbar)\cup Z_1^*(\IQbar)\cup \cdots \cup Z_r^*(\IQbar)$. 
 
To prove that the inequality \eqref{EquationAuxiliaryEquation} holds
for all points in $X^*(\IQbar)$, it suffices to verify
\[
X^*\subset (X\setminus Z) \cup Z_1^* \cup \cdots \cup Z_r^*.
\]
But this is equivalent to the inclusion
\begin{equation}
\label{eq:genericspecialirredcomp}
X \setminus X^*  \supset Z \cap (X\setminus Z_1^*) \cap \cdots \cap (X\setminus Z_r^*)  = (Z\setminus Z_1^*) \cap \cdots \cap (Z\setminus Z_r^*).
\end{equation}
Finally, a generically special subvariety of $\cA$ contained in some $Z_i$ 
will be contained in $X$, and therefore (\ref{eq:genericspecialirredcomp}) holds true. 

Now the inequalities in Theorem~\ref{TheoremMainResultHeightBound} and Theorem~\ref{TheoremMainResultHeightBoundStronger} hold since, by \eqref{EqHeightTotal}, $h(\pi(P)) \le h(P)$ for any $P \in \cA(\IQbar)$.
\qed


\section{Application to the Geometric Bogomolov Conjecture}\label{SectionGBC}
In this section we prove Theorem~\ref{TheoremGBC} over the base field
$\overline{\mathbb{Q}}$ and abbreviate $\IP^m_\IQbar$ by
$\IP^m$. 

More general base fields can be handled using
 the Moriwaki height version of
Theorem~\ref{TheoremMainResultHeightBound}. More details 
are presented in
Appendix~\ref{sec:appendixA}.

There exists a smooth, irreducible, quasi-projective curve $S$ over
$\overline{\mathbb{Q}}$ 
whose function field is $K$. 
We fix an algebraic closure $\overline
K\supset K$ of $K$. 
As in \S \ref{SubsectionEmbeddingAbSch} we can find, up to
removing finitely many points of $S$,
 an abelian scheme $\mathcal{A}\rightarrow S$ 
whose generic fiber is $A$ from Theorem \ref{TheoremGBC}.
We  equip $\cA$  with an admissible immersion
$\cA\rightarrow\IP^M\times\IP^m$, cf. \S\ref{SubsectionEmbeddingAbSch}.
In particular, we have an immersion $\iota \colon S\rightarrow\IP^m$. 
For $s\in S(\IQbar)$  we set
\begin{equation}
\label{def:hS}
h_S(s) = \frac{1}{\deg S} h(\iota(s))
\end{equation}
where $h$ on the right-hand side is the  height on $\IP^m(\IQbar)$
and $\deg S$ is the degree of the Zariski closure of $\iota(S)$ in
$\IP^m$. 
We use the same normalization as in Silverman's work \cite[\S 4]{Silverman},
which will play an important role momentarily.
In addition,
we have the
fiberwise N\'eron--Tate height $\hat
h_\cA \colon \cA(\IQbar)\rightarrow[0,\infty)$, \textit{cf.}
(\ref{eq:fiberwiseNT}). 
On $A$ 
we also have a N\'eron--Tate height
$h_{K,A} \colon A(\overline K)\rightarrow [0,\infty)$, \textit{cf.} the end
of \S \ref{sec:heights}.




Before we get to the nuts and bolts we state Silverman's Height Limit
Theorem. 

Recall that we can represent a  point $ x\in A(\overline K) = (A\otimes_K{\overline
K})(\overline K)$ using a section $S'\rightarrow \cA\times_S S'$
where  $S'$ is a smooth, irreducible curve and 
$\rho \colon S'\rightarrow S$ is generically finite morphism.
We write $\sigma_x$ for the composition $S'\rightarrow \cA\times_S
S'\rightarrow \cA$.
We may evaluate $\hat h_{\cA}$ at $\sigma_x(t) \in
\cA_{\rho(t)}$
for all $t\in S'(\IQbar)$. 

\begin{theorem}[Silverman]
In the notation above we have
\begin{equation}
\label{eq:silvermanlimit}
\lim_{\substack{t \in S'(\IQbar) \\ h_S(\rho(t)) \rightarrow \infty}}
\frac{\hat h_{\cA}(\sigma_x(t))}{h_S(\rho(t))} =  \hat h_{K,A}(x). 
\end{equation}
\end{theorem}
\begin{proof}
This follows from \cite[Theorem~B]{Silverman}  via a base change
argument as follows.
The smoothness condition in this reference 
can be dropped when using line bundles instead of Weil divisors.
Observe that Silverman's Theorem  deals with the rational case $x \in A(K)$, which
comes from  a section $S\rightarrow \cA$. 

We have a morphism $\sigma_x \colon S'\rightarrow \cA$, which
composed with $\cA\rightarrow S$ equals 
 $\rho \colon S'\rightarrow S$.
We write $K'=\IQbar(S')$ and 
$A_{K'} = A\otimes_K K'$. Then $h_{S'} \colon S'(\IQbar)\rightarrow [0,\infty)$ is defined analog to $h_S$ via
an immersion of $S'$ into some projective space and then  normalizing
as in (\ref{def:hS}). Of course, $h_{S'}$ depends  on the choice of
this immersion. 
But we have $h_{S}(\rho(t))/h_{S'}(t)\rightarrow \deg\rho  =[K':K]$
for $t\in S'(\IQbar)$ as $h_{S'}(\rho(t))\rightarrow\infty$ by
quasi-equivalent of heights on curves, \textit{cf.} \cite[Corollary 9.3.10]{BG} and our choice of normalization. 
  Silverman's Theorem applied to $x\in A(K')$ implies
 $\hat h_{\cA}(\sigma_x(t))/h_{S'}(t)\rightarrow
\hat h_{K',A_{K'}}(x)$ 
as $h_{S}(t)\rightarrow \infty$
for $t\in S'(\IQbar)$.
Thus $\hat h_{\cA}(\sigma_x(t))/ h_S(\rho(t)) \rightarrow [K':K]^{-1} 
\hat h_{K',A_{K'}}(x)$ for 
$t\in S'(\IQbar)$ and
$h_{S}(\rho(t))\rightarrow \infty$. 

Now $\hat h_{K,A}$ and $\hat h_{K',A'}$ are related by 
$\hat h_{K',A_{K'}} = [K':K] \hat h_{K,A}$, this follows from the related
statement for naive heights, \textit{cf.} \cite[Remark 9.2]{Conrad}, and
passing to the limit. The factor $[K':K]$ cancels out with the same
factor coming from quasi-equivalence of heights and this yields (\ref{eq:silvermanlimit}).
\end{proof}

Now we complete the proof of Theorem \ref{TheoremGBC}. 

It is enough to prove the theorem for the symmetric
 line bundle $L$ attached to
the closed immersion $A\rightarrow \IP_{K}^M$.

Let $\cX$ be the Zariski closure of $X$ inside $\cA\supset A\supset
X$. 
Then $\cX$ is irreducible and flat over $S$
and $X$ is the generic fiber of $\cX\rightarrow S$.

Therefore by the assumption on $X$ the variety
$\mathcal{X}$ is not generically special.
We will apply Proposition \ref{PropositionBeforeMainTheorem}, so let
$\cU$ be the Zariski open and dense subset of $\cX$ from  this proposition. 

We define $U = \cU\cap X$, where the intersection is
inside $\cA$. This is a Zariski open and dense subset of $X$. 

It suffices to prove that there exists $\epsilon > 0$ such that 
$x\in U(\overline K)$ implies
$\hat h_{K,A}(x)\ge \epsilon$.

Indeed, let $x \in U(K^\prime)$ where $K^\prime$ is a
finite field extension of $K$ contained in $\overline{K}$. As above, 
there are an irreducible, quasi-projective curve $S^\prime$ over
$\IQbar$ with function field $K^\prime$, a generically finite morphism
$\rho \colon S^\prime \rightarrow S$, and a section 
$S^\prime \rightarrow \cA\times_S S'$ determined by $x$.
We write $\sigma_x \colon S'\rightarrow \cA$ for this section composed with 
$\cA\times_S S'\rightarrow\cA$. 

The Zariski closure $\mathcal{Y}$ in $\mathcal{A}$ of the image of $\sigma_x$  is an irreducible closed curve in
$\mathcal{A}$. We have $\mathcal{Y} \subset \mathcal{X}$ as $x\in
X(K')$. Moreover, $\mathcal{Y} \cap \cU \not= \emptyset$
since $x \in U(\overline{K})$. So $\mathcal{Y} \cap \cU$
is a curve that differs from $\cY$ in only finitely many points. 

We fix a sequence
$t_1,t_2,\ldots\in S'(\IQbar)$ 
such that $\lim_{n \rightarrow \infty} h_{S}(\rho(t_n))=\infty$.
Silverman's Theorem implies
\begin{equation}
\label{eq:silvermanconsequence}
\lim_{n \rightarrow \infty} \frac{\hat
h_{\cA}(\sigma_x(t_n))}{h_S(\rho(t_n))} = \hat h_{K,A}(x).
\end{equation}

For $n$ large enough we have $\sigma_x(t_n)\in \cU(\IQbar)$. 
%
%
%
By Proposition \ref{PropositionBeforeMainTheorem} 
there exists a constant $c>0$, independent
of $x,\sigma_x,$ and $n$, such that 
\begin{equation}\label{EquationGB1}
h(\sigma_x(t_n))  \le 
c\left(1+\hat h_{\cA}(\sigma_x(t_n))\right)
\quad\text{for all large integers $n$.}
\end{equation}

By (\ref{EqHeightTotal}) the naive height  $h(\sigma_x(t_n))$ is at
 least
the  height of 
$\iota(\pi(\sigma_x(t_n)))={\iota}(\rho(t_n))\in\IP^m(\IQbar)$.
By our choice (\ref{def:hS}) we have $h(\iota(\rho(t_n))) = \deg(S) h_S(\rho(t_n))$. We insert into
 (\ref{EquationGB1}) and divide by $h_S(\rho(t_n))$ to obtain 
\begin{equation*}
\deg S \le c \frac{1+ \hat h_{\cA}(\sigma_x(t_n))}{ h_S(\rho(t_n))}.
\end{equation*}

Finally, 
we pass to the limit $n\rightarrow \infty$ and recall (\ref{eq:silvermanconsequence}) to conclude 
$\hat h_{K,A}(x)\ge \deg(S)/c$. The theorem follows as $c$ and $\deg(S)$
are independent of $x$. \qed




\appendix
\renewcommand{\thesection}{\Alph{section}}
\setcounter{section}{0}

\section{Passing from $\overline{\mathbb{Q}}$ to any Field of Characteristic $0$}
\label{sec:appendixA}
In this appendix, we 
sketch a proof 
of Theorem~\ref{TheoremGBC} for any $k$ algebraically closed of characteristic $0$. We do this by proving a Moriwaki height version of Theorem~\ref{TheoremMainResultHeightBound}, allowing $\overline{\mathbb{Q}}$ to be replaced by any algebraically closed field of finite transcendence degree over $\overline{\mathbb{Q}}$. Then we repeat the proof of Theorem~\ref{TheoremGBC} for $k=\overline{\mathbb{Q}}$ ($\mathsection$\ref{SectionGBC}) with this new height function to get the result when $\mathrm{trdeg}_{\IQ}k < \infty$. Finally we use essential minimum to reduce to this case.

\subsection{Moriwaki height}
In this subsection we review Moriwaki's height theory
\cite{MoriwakiArithmetic-heig}.

Let $k_0$ be a finitely generated field over $\IQ$ with 
$\trdeg{k_0/\IQ} =d$. 
Moriwaki \cite{MoriwakiArithmetic-heig} developped the following height theory, generalizing the classical height theory for $\overline{\mathbb{Q}}$.


Fix a polarization
$\overline{\mathbf{B}}=(B;\overline{H}_1,\ldots,\overline{H}_d;\tau)$
of the field $k_0$, \textit{i.e.} a 
flat and quasi-projective integral scheme over $\IZ$,
a collection of nef smooth hermitian line bundles
$\overline{H}_1,\ldots,\overline{H}_d$ on $B$ and an isomorphism of
fields $\tau \colon \IQ(B)  \rightarrow k_0$. In most of the literature, including Moriwaki's paper, the isomorphism $\tau$ is omitted as it is fixed.

Let $X$ be an irreducible projective variety over $k_0$ and let $L$ be a line bundle on $X$ which is defined over $k_0$. 
Moriwaki \cite{MoriwakiArithmetic-heig}  defines a height function
\begin{equation}\label{EquationMoriwakiHeightProjectiveVariety}
h^{\overline{\mathbf{B}}}_{k_0,X,L}\colon X(\overline{k_0})\rightarrow \mathbb{R}
\end{equation}
 which is well-defined modulo the set of bounded functions on $X(\overline{k_0})$.
We will  not repeat the exact definition  of the Moriwaki height here,
 but will mention
 some properties. If the field $k_0$ is clear from the context, 
then we abbreviate $h^{\overline{\mathbf{B}}}_{k_0,X,L}$ to $h^{\overline{\mathbf{B}}}_{X,L}$. If furthermore $X$ is clear, then we abbreviate it to $h^{\overline{\mathbf{B}}}_L$.

Before going on, let us make the following remark. If $k_0'$ is a field with an isomorphism $\iota \colon k_0 \rightarrow k_0'$, then 
we have a polarization $\overline{\mathbf{B}}' =
(B;\overline{H}_1,\ldots,\overline{H}_d; \iota \circ \tau)$ of $k_0'$.
For any algebraic closure $\overline{k_0}'$ of $k_0'$, $\iota$ extends
(non-uniquely) to an isomorphism $\overline{k_0} \rightarrow \overline{k_0}'$ which we still denote by $\iota$ by abuse of notation. Then $h^{\overline{\mathbf{B}}'}_{k_0'} \circ \iota = h^{\overline{\mathbf{B}}}_{k_0}$.


As is pointed out by Moriwaki, if $k_0$ is a number field, then we
recover 
the classical height functions. 
Just as the classical height, 
the Moriwaki height \eqref{EquationMoriwakiHeightProjectiveVariety}
satisfies the several  properties.

\begin{proposition}[Height Machine for the Moriwaki
height]\label{PropositionOfMoriwakiHeight}We keep the notation from above.
\begin{enumerate}
\item(Additivity) If $M$ is another line bundle on $X$, then
$h^{\overline{\mathbf{B}}}_{L\otimes
M}=h^{\overline{\mathbf{B}}}_L+h^{\overline{\mathbf{B}}}_M$. 
\item(Functoriality) Let $q\colon X\rightarrow Y$ be a
quasi-finite morphism of projective varieties over $k_0$
 and let $M$ be a line bundle on $Y$. Then
\[
h^{\overline{\mathbf{B}}}_{q^*M}=h^{\overline{\mathbf{B}}}_M \circ q
\]
\item (Boundedness) The function $h^{\overline{\mathbf{B}}}_L$ is bounded below away from the base locus of $L$
. In particular $h^{\overline{\mathbf{B}}}_L$ is bounded on $X(\overline{k_0})$ if $L = \cO_X$.
\item (Northcott) If $L$ is ample, and if $\overline{\mathbf{B}}$ is
big, \textit{i.e.} the $\overline{H}_i$'s are nef and big. Then for
any real numbers $B, D$, the set
\[
\{P\in X(\overline{k_0}): h^{\overline{\mathbf{B}}}_L(P) \le  B, ~ [k_0(P):k_0] \le  D\}
\]
is finite.
\item (Algebraic Equivalence) If $L$ and $M$ are algebraically equivalent and $L$ is ample, then
\[
\lim_{h^{\overline{\mathbf{B}}}_L(P) \rightarrow \infty} \frac{h^{\overline{\mathbf{B}}}_M(P)}{h^{\overline{\mathbf{B}}}_L(P)} = 1.
\]
\end{enumerate}
\begin{proof}
Part (1), (3) and (4) are proven by
Moriwaki \cite[Proposition~3.3.7(2-4)]{MoriwakiArithmetic-heig}.
See \cite[Proposition~2(iv)]{WazirMoriwakiHeight} for a proof of part
(2), note that the smoothness assumption is unnecessary and that $q$
must be generically finite in \cite[Proposition 1.3(2)]{MoriwakiArithmetic-heig}. Part (5) can be proven by a verbalized copy of \cite[Chapter~4, Proposition~3.3 and Corollary~3.4]{FoDG} with the usual height function replaced by the Moriwaki height.
\end{proof}
\end{proposition}
Proposition~\ref{PropositionOfMoriwakiHeight} enables us to transfer
results involving only properties of the height listed in the Height
Machine 
 to the Moriwaki height.


Next we turn to abelian varieties. Let $A$ be an abelian variety over
$k_0$, and let $L$ be a symmetric ample line bundle on $A$ which is
defined over $k_0$. The limit
\begin{equation}\label{EquationNeronTateMoriwaki}
\hat{h}^{\overline{\mathbf{B}}}_L(P)=\lim_{n\rightarrow \infty}
2^{-2n} h^{\overline{\mathbf{B}}}_L([2^n]P)\quad \text{for all}\quad P\in A(\overline{k_0})
\end{equation}
exists and is independent of the choice of a representative of the
height function. 
Then $\hat{h}^{\overline{\mathbf{B}}}_L$ is  called
the \textit{canonical} or \textit{N\'{e}ron--Tate height}
on $A(\overline{k_0})$ attached  of $L$ with respect to $\overline{\mathbf{B}}$. If $k_0=\mathbb{Q}$, then $\hat{h}^{\overline{\mathbf{B}}}_L$ coincides with the usual N\'{e}ron--Tate height over $\overline{\mathbb{Q}}$. Moriwaki \cite[$\mathsection$3.4]{MoriwakiArithmetic-heig} proved that the N\'{e}ron--Tate height $\hat{h}^{\overline{\mathbf{B}}}_L$ is quadratic, \textit{i.e.}
\begin{equation*}
\hat{h}^{\overline{\mathbf{B}}}_L([N]P)=N^2 \hat{h}^{\overline{\mathbf{B}}}_L(P) \quad \text{for
all}\quad P \in A(\overline{k_0}).
\end{equation*}

The following proposition is proven by Moriwaki \cite[Proposition~3.4.1]{MoriwakiArithmetic-heig}.
\begin{proposition}
\begin{enumerate}
\item [(i)] We have $\hat{h}^{\overline{\mathbf{B}}}_L(P) \ge  0$ 
for all $P\in A(\overline{k_0})$.
\item [(ii)] We have  $\hat{h}^{\overline{\mathbf{B}}}_L(P)=0$ for all
$P\in A(\overline{k_0})_{\mathrm{tor}}$.
\item [(iii)] Assume $\overline{\mathbf{B}}$ is big, \textit{i.e.} $\overline{H}_i$'s are nef and big. Then $\hat{h}^{\overline{\mathbf{B}}}_L(P)=0$ if and only if $P$ is a torsion point.
\end{enumerate}
\end{proposition}

\subsection{Height Inequality}
Let $k_0$ be a finitely generated field extension of $\IQ$. Let $\overline{\mathbf{B}}$ be a polarization of $k_0$. Let $\overline{k_0}$ be an algebraic closure of $k_0$.

Let $S$ be a smooth irreducible quasi-projective curve over $k_0$, and
let $\pi \colon \mathcal{A} \rightarrow S$ be an abelian scheme over
$k_0$ of relative dimension $g\ge 1$. 
We fix a smooth, irreducible projective curve $\overline{S}$ over $k_0$
that
contains $S$ as a Zariski open subset. 
Let $\mathcal{M}$ be an ample line bundle on $\overline{S}$ defined
over $k_0$. 
Let $\mathcal{L}$ be a symmetric relatively ample line bundle on
$\mathcal{A}/S$ 
defined over $k_0$. Then we have the following analogue of 
Theorem~\ref{TheoremMainResultHeightBound}.

\begin{theorem}\label{TheoremMainTheorem}
Let $X$ be a closed irreducible subvariety of $\mathcal{A}$  over
$k_0$ and let  $X^*$ be as above Proposition~\ref{TheoremStarSetOpen}
with $k=\overline{k_0}$. Then there exists  $c>0$ depending only on 
$\overline{\mathbf{B}},\mathcal{A}/S,X,\mathcal{L},$ and $\mathcal{M}$
such that 
\[
h^{\overline{\mathbf{B}}}_{S,\mathcal{M}}(P) \le
c\left(1+ \hat{h}^{\overline{\mathbf{B}}}_{\mathcal{A},\mathcal{L}}(P)\right)
\quad \text{for all}\quad P\in X^*(\overline{k_0})
\]
where $h^{\overline{\mathbf{B}}}_{S,\mathcal{M}}$ is the Moriwaki height defined by \eqref{EquationMoriwakiHeightProjectiveVariety}, and $\hat{h}^{\overline{\mathbf{B}}}_{\mathcal{A},\mathcal{L}}(x)$ is the N\'{e}ron--Tate height $\hat{h}^{\overline{\mathbf{B}}}_{\mathcal{L}_{\pi(x)}}(x)$ defined by \eqref{EquationNeronTateMoriwaki} on the abelian variety $\mathcal{A}_{\pi(x)}$.
\end{theorem}
If $k_0$ is a subfield of $\IC$, then we can take
$\overline{k_0} \subset \IC$. In this case we can proceed
as in
the proof of Theorem~\ref{TheoremMainResultHeightBound}, with the
usual height over $\overline{\mathbb{Q}}$ replaced by the Moriwaki
height over $k_0$. In fact we only used $\overline{\mathbb{Q}}$ in the
arguments involving heights, \textit{i.e.}
Proposition~\ref{PropositionHeightChangeUnderScalarMultiplication} (in
fact only Lemma~\ref{LemmaChangeOfHeightUnderRationalMap} and below)
and Proposition~\ref{PropositionBeforeMainTheorem} (for this we need
the Moriwaki height version of Silverman-Tate, which
is \cite[Theorem~2]{WazirMoriwakiHeight}). Now
Proposition~\ref{PropositionOfMoriwakiHeight} provides us with the
Height Machine for $h^{\overline{\mathbf{B}}}_L$ and hence all the arguments are still valid.

For general $k_0$ finitely generated over $\IQ$, we have that $k_0$ is isomorphic to a subfield $k_0'$ of $\IC$ via some $\iota$. Let $\overline{k_0}'$ be the algebraic closure of $k_0'$ in $\IC$, then $\iota$ extends to some $\iota \colon \overline{k_0} \rightarrow \overline{k_0}'$. As explained in the paragraph below \eqref{EquationMoriwakiHeightProjectiveVariety}, we can get a polarization $\overline{\mathbf{B}}'$ of $k_0'$ such that $h^{\overline{\mathbf{B}}'}_{k_0'} \circ \iota = h^{\overline{\mathbf{B}}}_{k_0}$. So we are reduced to the case where $k_0$ is a subfield of $\IC$ and hence we are done.

\subsection{Geometric Bogomolov Conjecture}
Suppose that we are in the situation of Theorem~\ref{TheoremGBC}. There exists a smooth, irreducible, quasi-projective curve $S$ over
$k$ whose function field is $K$. 
We can find, up to removing finitely many points of $S$,
 an abelian scheme $\mathcal{A}\rightarrow S$ 
whose generic fiber is $A$. We write $\mathcal{X}$ for the Zariski
closure of $X$ under $A\subset \cA$. Then $\mathcal{X}$ is a closed
irreducible subvariety of $\cA$. 
We fix an algebraic closure $\overline
K\supset K$ of $K$ and a smooth, irreducible, projective curve
$\overline S$ that contains $S$ as a Zariski open subset. 

The symmetric ample line bundle $L$ extends, up to removing finitely
many points of $S$, to  a symmetric relatively ample line bundle $\cL$
on $\cA/S$. There exists a field $k_0$ finitely generated over $\IQ$
such that $S$, $\mathcal{A}/S$, $\cL$ and $\mathcal{X}$ are defined
over $k_0$.
We treat these objects as being over $k_0$. 

 Let us take an ample line bundle on $\overline{S}$ defined over $k_0$.

By \cite[EGA~IV$_2$, Proposition~2.8.5]{EGAIV} $\mathcal{X}$ is flat
over $S$ and $i^{-1}(\mathcal{X})=X$, so $X$ is the generic fiber of
$\mathcal{X}\rightarrow S$. Therefore $\mathcal{X}$ is not generically
special by the assumption on $X$. Hence $\mathcal{X}^*$ is Zariski
open dense in $\cX$ by Proposition~\ref{TheoremStarSetOpen}.

Take algebraic closures $\overline{k_0}$ in $k$ and
$\overline{k_0(S)}$ in $\overline{K}$. From now on we see $X,A,$ and
$L$ as defined over $k_0(S)$. 
We claim that there exists a constant $\epsilon > 0$ such that
\begin{equation}\label{EqSmallPointsSmallAlgCloField}
\{ P\in X(\overline{k_0(S)}) : \hat{h}_{\overline{k_0}(S),A,L}(P) \le \epsilon \}
\end{equation}
is not Zariski dense in $X$.

We indicate how to modify  the proof in $\mathsection$\ref{SectionGBC}
to prove this claim. The only changes are
\begin{itemize}
\item The Zariksi open subset $\cU$ of $\cX$ is replaced by $\cX^*$.

\item Instead of Proposition~\ref{PropositionBeforeMainTheorem}, we use the generalized version of Theorem~\ref{TheoremMainResultHeightBound}. 
More precisely let $\overline{\mathbf{B}}$ be a big polarization of $k_0$, \textit{i.e.} the $\overline{H}_i$'s are nef and big. 
Apply Theorem~\ref{TheoremMainTheorem} to the subvariety
$\mathcal{X} \subset \cA$ and $(k_0,\overline{\mathbf{B}})$ to obtain
a constant $c>0$.

\item The polarization $\overline{\mathbf{B}}$ is big, so there exists
a sequence of points $t_1,t_2,\ldots \in S(\overline{k_0})$ such that
$\lim_{n\rightarrow \infty} h^{\overline{\mathbf{B}}}_{S,\cM}(t_n)
= \infty$. Also the Moriwaki height version of Silverman's
Theorem \eqref{eq:silvermanlimit} still holds,  see \cite[Theorem~3]{WazirMoriwakiHeight}. In fact Proposition~\ref{PropositionOfMoriwakiHeight} provides us with the Height Machine for $h^{\overline{\mathbf{B}}}_L$ and hence Silverman's original proof still works with the usual height function replaced by the Moriwaki height.
\end{itemize}

We are not done yet because we want to replace $\overline{k_0(S)}$ by
$\overline{K}$ in \eqref{EqSmallPointsSmallAlgCloField}. 
Indeed, $K = k_0(S)\otimes_{k_0} k$ contains $k$ which is an arbitrary field of characteristic
$0$ and therefore possibly not finitely generated over $\IQ$.   
To proceed
we prove the following statement
on the essential minimum
\begin{equation*}
\mu_{\mathrm{ess}}(X) = \inf \left\{ \epsilon > 0 : \{P \in
X(\overline{k_0(S)}) : \hat h_{\overline{k_0}(S),A,L}(P)\le \epsilon \}
\text{ is Zariski dense in } X \right\}.
\end{equation*}
The analog $\mu_{\mathrm{ess}}(X_K)$  
where $X_K = X \otimes_{k_0(S)} K$ is defined similarly but
involves $K$.

\medskip
\noindent\textbf{Claim:} 
If $\mu_{\mathrm{ess}}(X_K) = 0$,
then $\mu_{\mathrm{ess}}(X) = 0$.
\medskip


S.~Zhang  proved two inequalities
relating  the
 essential minima
and the height of a subvariety of an abelian variety
in the  number field case \cite{ZhangArVar}.
To prove our claim we require
Gubler's
\cite[Corollary~4.4]{Gubler:Bogo}  version of Zhang's inequalities
for function fields.
See \S 3 of \cite{Gubler:Bogo} for the definition of the height of a
subvariety of $A$. 

More precisely, $\mu_{\mathrm{ess}}(X)=0$ if and only if
$\hat{h}_{\overline{k_0}(S),A,L}(X) = 0$. Moreover,
$\mu_{\mathrm{ess}}(X_K)=0$ if and only if $\hat{h}_{K,A_K,L}(X_K)
=0$. Finally, we sketch how to prove the equality
$\hat{h}_{\overline{k_0}(S),A,L}(X) = \hat{h}_{K,A_K,L_K}(X_K)$ which
settles our claim. The base change involves only an extension of the
field of constants $k/\overline{k_0}$ under which the naive
height on projective height remains unchanged. The height of a
subvariety of some projective space can be defined as the height of a
Chow form and is thus invariant under base change as well. Finally, the
canonical height of a subvariety of $A$ is a limit as
in Tate's argument and is thus invariant under base change.


\section{Proposition~\ref{TheoremStarSetOpen} for Higher Dimensional Base}
\label{sec:appendixB}In this appendix, we explain how to generalize Proposition~\ref{prop:specialpoints} to higher dimensional base. We work under the frame of $\mathsection$\ref{SectionStarSetOpen} except that we do not make any assumption on $\dim S$. In other words, out setting is as follows. 

Let $k$ be an algebraically closed field of characteristic $0$. Let $S$ be a smooth irreducible quasi-projective variety over $k$ and fix an algebraic closure $\overline{K}$ of $K=k(S)$. Let $A$ be an abelian variety over $\overline{K}$.

We start with the following proposition.

\begin{proposition}\label{PropositionTorsionBoundedSet}
Assume $A^{\overline K/k}=0$.
The order of any point  in 
\begin{equation}
\label{eq:torsionboundedset2}
  \{ P\in\tor{A(\overline{K})} : [K(P):K]\le D \}
\end{equation}
 is
bounded in terms of $A$ and $D$ only. 
\end{proposition}
\begin{proof}
Fix an irreducible projective variety $\overline{S}$ over $k$ whose function field is $K$. We may take $S$ as a Zariski open dense subset of $\overline{S}$ such that $A$ extends to an abelian scheme $\cA \rightarrow S$, namely the generic fiber of $\cA \rightarrow S$ is $A$. We may furthermore shrink $S$ such that $S$ is smooth and that $\overline{S} \setminus S$ is purely of codimension $1$. We denote by $i \colon A \rightarrow \cA$ the natural morphism.

Any $P \in \tor{A(\overline{K})}$ defines a morphism $\sigma_P \colon \mathrm{Spec}\overline{K} \rightarrow A$. Suppose the order of $P$ is $N$. Then the Zariski closure of $\im{i\circ \sigma_P}$, which we denote by $\mathcal{T}$, is irreducible, dominates $S$ and satisfies $[N]\mathcal{T}=0$. So $\mathcal{T}$ is an irreducible component of the kernel of $[N] \colon \cA \rightarrow \cA$ by comparing dimensions. In particular $\mathcal{T} \hookrightarrow \ker[N]$ is an open and closed immersion. But $\ker[N] \rightarrow S$ is finite \'{e}tale, 
so is $\mathcal{T} \rightarrow S$. Thus $\mathcal{T} \rightarrow S$ is an \'{e}tale covering of degree $[K(P):K]$.

In other words, any torsion point $P \in \tor{A(\overline{K})}$ yields an \'{e}tale covering of $S$ of degree $[K(P):K]$.

By Lemma~\ref{LemmaFinitelyManyEtaleCovers} below, the compositum $F$ in $\overline K$ of all such extensions $K(P)$ of $K$ of degree at most $D$ is a finite field extension of $K$. For $S$ a curve we cited \cite[Corollary 7.11]{Voelklein}. In particular, $P \in A(F)$ for all $P$ in \eqref{eq:torsionboundedset2}.

Now $A^{\overline K/k}=0 $. So the Lang--N\'eron Theorem, cf. \cite[Theorem~1]{LangNeron} or  \cite[Theorem~7.1]{Conrad}, implies that $A(F)$ is a finitely generated group. Thus  $[M](P)=0$ for some $M\in\IN$ that is independent of $P$. Our claim follows. 
\end{proof}

\begin{lemma}\label{LemmaFinitelyManyEtaleCovers}
Let $k$ be an algebraically closed field of characteristic $0$ and let $S$ be a smooth irreducible quasi-projective variety over $k$. Then for any integer $D>0$ there are at most finitely many \'{e}tale coverings of $S$ of degree $\le D$.
\end{lemma}
\begin{proof} It suffices to prove that the \'{e}tale fundamental group $\pi_1(S)$ is topologically finitely generated.

Let $k_0$ be an algebraically closed subfield of $k$ that is of finite transcendence degree over $\overline{\IQ}$ such that $S$ is defined over $k_0$. Then we can fix an embedding $k_0 \hookrightarrow \IC$. We write $S_0$ for the descent of $S$ to $k_0$, namely $S = S_0 \otimes_{k_0} k$. We also write $S_{\IC}= S_0 \otimes_{k_0} \IC$ for the base change of $S_0$ to $\IC$.

Let $s_0 \colon \mathrm{Spec}k_0 \rightarrow S_0$ be a geometric point. Denote by $s \colon \mathrm{Spec} k \rightarrow S$, resp. $s_{\IC} \colon \mathrm{Spec}\IC \rightarrow S_{\IC}$, the corresponding geometric points.

It is a classical result that the topological fundamental group $\pi_1(S_{\IC}^{\anE},s_{\IC})$ is finitely generated. Hence by Riemann's Existence Theorem \cite[Expos\'{e}~XII, Th\'{e}or\`{e}me~5.1]{SGA1} the \'{e}tale fundamental group $\pi_1(S_{\IC},s_{\IC})$ is topologically finitely generated. But then $\pi_1(S_{\IC},s_{\IC}) \cong \pi_1(S_0,s_0)$ by \cite[Corollary~6.5 and Remark~6.8]{CadoretFundamentalGroup}. So $\pi_1(S_0,s_0)$ is topologically finitely generated. Then again by \cite[Corollary~6.5 and Remark~6.8]{CadoretFundamentalGroup}, we have $\pi_1(S,s) \cong \pi_1(S_0,s_0)$. So $\pi_1(S,s)$ is topologically finitely generated.
\end{proof}

Now we are ready to prove:

\begin{proposition}
let $V_0$ be an irreducible variety defined over $k$ and $V = V_0\otimes_k  \overline{K}$. By abuse of notation we consider $V_0(k)$ as a subset of $V(\overline K)$. Define $\Sigma = V_0(k)\times A_{\mathrm{tor}} \subset V(\overline
K)\times A(\overline K)$.

Let $Y$ be
  an irreducible closed subvariety of $V\times A$ such that $Y(\overline K) \cap
  \Sigma$ lies Zariski dense in $Y$. 
If $A^{\overline K/k}=0$ 
then 
$Y = (W_0\otimes_k  \overline K)\times (t+B)$ where $W_0\subset V_0$ is an
irreducible closed subvariety, $t\in \tor{A(\overline{K})}$ and
$B$ is an abelian subvariety of $A$. 
\end{proposition}

The only difference of this proposition with Proposition~\ref{prop:specialpoints} is that we do not make any assumption on $\dim S$.

As we have pointed out, in the proof of Proposition~\ref{prop:specialpoints} the only place  where we used the assumption $\dim S=1$ is to prove the statement involving \eqref{eq:torsionboundedset}. But for $S$ of arbitrary dimension this follows from Proposition~\ref{PropositionTorsionBoundedSet}.

\section{Hyperbolic Hypersurfaces of Abelian Varieties}
\label{sec:appendixC}
Suppose $F$ is an algebraically closed field of
characteristic zero.
For each integer $d\ge 0$ we let $F[X_0,\ldots,X_M]_d$ be the vector
space of homogeneous polynomials of degree $d$ in $F[X_0,\ldots,X_M]$
together with $0$.
In this section we identify $F[X_0,\ldots,X_M]_d$ with 
$\IA^{M+d\choose M}(F)$, where we abbreviate $\IA^M =\IA_F^M$
and $\IP^M = \IP_F^M$.

Brotbek's deep result  \cite{DamienHyperbolicity} implies that
a generic
sufficiently ample hypersurface in a smooth projective variety over $\IC$ is 
 hyperbolic. In the very particular case of an abelian variety we give
 an independent proof 
 that involves an explicit bound on the
 degree. Recall that an
irreducible subvariety of an abelian variety is hyperbolic if and only if it does
 not contain a coset of positive dimension by 
 the Bloch--Ochiai Theorem. The main results of this paper do not
depend on the Bloch--Ochiai Theorem.

\begin{proposition}
\label{prop:abvarhyperbolic}
  Let $A$ be an abelian variety over $F$ of dimension $g\ge 1$ 
  with $A\subset \IP^M$ and suppose $d\ge g-1$. There exists a Zariski open and dense subset $U\subset
\IA^{{M+d\choose M}}$, whose complement in $\IA^{{M+d\choose M}}$ has codimension at least $d+2-g$, such that if $f\in U(F)$, then 
  $A \cap \scrZ(f)$
does not contain any positive dimensional coset.
\end{proposition}

A direct corollary of this proposition is the following statement. Let
$L$ be a very ample line bundle on $A$ giving rise to a projectively
normal, closed immersion $A \hookrightarrow \IP^M$
and say $P\in A(F)$. Then the
hypersurface defined by a generic choice of a section in
$H^0(\IP^n,\mathcal{O}(d))$ vanishing at $P$ is hyperbolic for $d \ge g$. 

Suppose $V$ is an irreducible, closed subvariety
of $\IP^M$ with ideal $I \subset F[X_0,\ldots,X_M]$. 
 We  write $I_d = I \cap F[X_0,\ldots,X_M]_d$. 
Then $F[X_0,\ldots,X_M]_d/I_d$ is a finite dimensional $F$-vector
space and the Hilbert function of $V$ is defined as 
\begin{equation*}
  \mathscr{H}_V(d) = \dim (F[X_0,\ldots,X_M]_d/I_d)
\end{equation*}
for all $d\ge 0$. 

Lower bounds for $\mathscr{H}_V(d)$ were obtained by Nesterenko,
Chardin, Sombra, and others. We only require a very basic inequality. 

Let $r=\dim V\ge 0$. 
After permuting coordinates, which does not affect the problem, we
suppose that $X_0 \not\in I$ and that $X_1/X_0,\ldots,X_r / X_0$ are
algebraically independent elements when taken as in the function field
of $V$. It follows that the composition
\begin{equation*}
  F[X_0,\ldots,X_r]_d \xrightarrow{\text{inclusion}} F[X_0,\ldots,X_M]_d \rightarrow 
F[X_0,\ldots,X_M]_d / I_d
\end{equation*}
is injective. Therefore, 
\begin{equation*}
  \mathscr{H}_V(d) \ge \dim  F[X_0,\ldots,X_r]_d = {r+d\choose r}
\end{equation*}
for all $d\ge 0$. 

We assume $r \ge 1$ is an integer. 
By basic properties of the binomial coefficients we have
${r+d \choose r}  \ge {1+d \choose 1} = d +1$. So
\begin{equation}
  \label{eq:HilbFunclb}
  \mathscr{H}_V(d) \ge  {\dim V+d \choose \dim V} \ge d+1
\end{equation}
if $\dim V\ge 1$.

We begin with a preliminary lemma that involves  cosets in $A$
with fixed stabilizer.

\begin{lemma}
\label{lem:hyperbolicfirststep}
Let $M,A,$ and $g$ be as in the proposition above with $g\ge
2$. 
Let $B$ be an abelian subvariety of $A$ of positive dimension
such that $d\ge \max\{1,g-\dim B\}$. 
There exists a Zariski open and dense subset $U\subset
\IA^{{M+d\choose M}}$, whose complement in $\IA^{{M+d\choose M}}$ has
codimension at least $d+1+\dim B-g$, such that if $f\in U(F)$, then 
  $A \cap \scrZ(f)$
does not contain any translate of $B$.
\end{lemma}
\begin{proof}
Say $N = {M+d\choose M}\ge M+d \ge g\ge 2$. For the
proof we abuse notation and consider elements in
$\IP^{N-1}(F)$ as  classes of homogeneous polynomials
in $F[X_0,\ldots,X_M]\ssm\{0\}$ of degree $d$ up-to scalar
multiplication. 
So $f(P)=0$ is a well-defined statement for $f\in \IP^{N-1}(F)$ and
$P\in \IP^M(F)$
and the
incidence set  
\begin{equation*}
   \{(f,P) \in \IP^{N-1}(F) \times A(F) : f(P) = 0 \}
\end{equation*}
 determines a Zariski closed subset
$Z\subset \IP^{N-1}\times A$.

We consider the two projections $\pi \colon \IP^{N-1}\times A\rightarrow
\IP^{N-1}$ and $\rho \colon \IP^{N-1}\times A\rightarrow A$.

Then $\rho|_Z \colon Z\rightarrow A$ is surjective and each
 fiber of $\rho|_Z$ is  linear variety of dimension $N-2$. 

 Say $B$ is an abelian subvariety  of $A$ with $\dim B\ge 1$.
 Let  $\varphi \colon A\rightarrow
 A/B$ denote the quotient map. We write $\boldsymbol{\varphi} = (\mathrm{id}_{\IP^{N-1}}, \varphi) \colon \IP^{N-1}\times A
 \rightarrow \IP^{N-1}\times (A/B)$; this morphism sends
 $(f,P)$ to $(f,\varphi(P))$.
The fibers of $\boldsymbol{\varphi}$ have dimension $\dim B$ and so 
 \begin{alignat}1
\nonumber
 \{ (f,P) \in Z(F) : &\dim_{(f,P)} {\boldsymbol{\varphi}|_Z}^{-1}(\boldsymbol{\varphi}(f,P))
 \ge \dim B \} \\
\label{eq:Zvarphi}
  &= \{(f,P) \in Z(F) : P+B \subset \scrZ(f) \}
 \end{alignat}
 defines a Zariski closed subset $Z_{\varphi}$ of $Z$. If
 $Z_{\varphi}$ is empty then the lemma follows with $U = \IA^N\backslash\{0\}$. Otherwise, let $W_1,\ldots,W_r$
 be the irreducible components of $Z_{\varphi}$.

 If $P\in
 A(F)$, then the fiber of $\rho|_{Z_\varphi}$ above $P$ is empty or has dimension
 \begin{equation*}
  \dim I(P+B)_d - 1 = \dim F[X_0,\ldots,X_M]_d -1- \mathscr{H}_{P+B}(d) 
  \le N-1 - (d+1).
 \end{equation*}
where we used \eqref{eq:HilbFunclb}.
So any non-empty fiber of $\rho|_{W_i}$ has dimension at most $N-1-(d+1)$. 
and by the Fiber Dimension Theorem
we conclude
\begin{equation}
  \label{eq:Zvarphiub}
  \dim W_i \le   
  N-1 - (d+1) + \dim \rho(W_i) \le N-d+g-2
\end{equation}
for all $i\in \{1,\ldots,r\}$
as $\dim \rho(W_i)\le \dim A \le g$. 


If $(f,P)\in W_i(F)$, then $\{f\}\times (P+B) \subset
Z_\varphi$. So if $(f,P)$ is not contained in any $W_j$ with
$i\not=j$, then $\{f\}\times(P+B)\subset W_i$ by the irreducibility of $P+B$. 
We conclude that a general fiber of $\pi|_{W_i}$ has dimension at
least $\dim B$. By the Fiber Dimension Theorem we find 
that $\dim W_i \ge \dim B + \dim \pi(W_i)$ for all
$i\in \{1,\ldots,r\}$;
note that $\pi(W_i)$ is Zariski closed in $\IP^{N-1}$. 
 
Together with \eqref{eq:Zvarphiub} we conclude
$\dim \pi(Z_\varphi) = \max_{1\le i\le r} \dim \pi(W_i)
\le N-d+g-2-\dim B$
and thus
\begin{equation*}
  \codim_{\IP^{N-1}} \pi(Z_\varphi) \ge d+1-g+\dim B.
\end{equation*}

As $\pi(Z_\varphi)$ is Zariski closed in $\IP^{N-1}$ we
 conclude that $U'=\IP^{N-1}\ssm \pi(Z_\varphi)$ is Zariski open and
dense in $\IP^{N-1}$ if $d \ge g-\dim B$. 
If $f\in U'(F)$, then there is no $P\in A(F)$ with $P+B \subset
\scrZ(f)$. Otherwise, we have in particular $f(P)=0$ and so
$(f,P) \in Z(F)$ which entails the contradiction $(f,P) \in Z_\varphi(F)$ by
\eqref{eq:Zvarphi}. 
The lemma follows if we take $U$ to be the preimage of $U'$
under the cone map
 $\IA^N\ssm\{0\} \rightarrow \IP^{N-1}$.
 \end{proof}

\begin{proof}[Proof of Proposition \ref{prop:abvarhyperbolic}]
If $g=1$, then $A\cap\scrZ(f)$ is finite for a
generic $f$ that is homogenous and of degree $d$. 
The proposition is  clearly true in this case.

Now say $g\ge 2$.  If $f\in F[X_0,\ldots,X_M]_d$,
then a coset contained in $\scrZ(f)\cap A$ is already contained in some
irreducible component $X$ of
$\scrZ(f)\cap A$. By B\'ezout's Theorem, 
$\deg X$ is bounded solely in terms of $d$ and $A$; here
$\deg(\cdot)$ denotes the usual degree as a subvariety of $\IP^{M}$. 

By a theorem of Bogomolov, \cite[Theorem 1]{BogomolovUeno}, the maximal
cosets contained in  $X$ 
are translates of  abelian subvarieties whose degree 
are bounded in terms of $\deg X,A,$ and the chosen polarization only.  Observe that
the proof of Bogomolov's Theorem works for algebraically closed fields in
characteristic zero. 
As $A$ contains only finitely many abelian subvarieties of given
degree, Bogomolov produces a finite set of abelian subvarieties that depends only on 
 $\deg X$ and $A\subset\IP^M$, thus only on $d$ and $A\subset\IP^M$. 

For any abelian subvariety $B\subset A$ of positive dimension that
arises in this set we write $U_B$ for the Zariski open and dense set
produced by Lemma \ref{lem:hyperbolicfirststep}. 
To rule out that $X$ contains a coset of positive dimension it
suffices to take $f\in U(F)$ where $U=\bigcap_B U_B$ is the intersection
over the finite set from Bogomolov's Theorem.
\end{proof}


\bibliographystyle{amsplain}
\bibliography{literature}


\vfill


\end{document}